\documentclass[11pt]{amsart}
\usepackage{lmodern}
\usepackage[T1]{fontenc}
\usepackage{microtype}
\usepackage{amssymb}
\usepackage{amsthm}
\usepackage{amscd}
\usepackage{hyperref}
\usepackage{cite}
\usepackage{verbatim}

\newtheorem{theorem}{Theorem}[section]
\newtheorem{lemma}[theorem]{Lemma}
\newtheorem{prop}[theorem]{Proposition}
\newtheorem{cor}[theorem]{Corollary}

\newtheorem*{theorem*}{Theorem}
\newtheorem*{cor*}{Theorem}

\theoremstyle{definition}
\newtheorem{definition}[theorem]{Definition}
\newtheorem{notation}[theorem]{Notation}
\newtheorem{remark}[theorem]{Remark}

\newtheorem{example}[theorem]{Example}

 \newenvironment{claimproof}{\begin{proof}}{\end{proof}}
\newcounter{claimcounter}
\numberwithin{claimcounter}{theorem}
\newenvironment{claim}{\stepcounter{claimcounter}{\textbf{Claim \theclaimcounter:}}}{}

\makeatletter

\def\dotminussym#1#2{%
  \setbox0=\hbox{$\m@th#1-$}%
  \kern.5\wd0%
  \hbox to 0pt{\hss\hbox{$\m@th#1-$}\hss}%
  \raise.6\ht0\hbox to 0pt{\hss$\m@th#1.$\hss}%
  \kern.5\wd0}

\mathchardef\mhyphen="2D

\allowdisplaybreaks[2]

\newcommand{\fn}[1]{\mathbf{#1}}
\newcommand{\sat}[1]{\ensuremath{(#1)\mathop{:}H_{#1}^\infty}}
\newcommand{\dsat}[1]{\ensuremath{[#1]\mathop{:}H_{#1}^\infty}}

\begin{document}

\title[Effective bounds in differential polynomial rings]{Proof mining and effective bounds in differential polynomial rings}
\author{William Simmons and Henry Towsner}
\date{\today}
\thanks{Partially supported by NSF grant DMS-1600263.}
\address{Department of Mathematics and Computer Science, Hobart and William Smith Colleges, 300 Pulteney Street, Geneva, NY 14456, USA}
\email{wsimmons@hws.edu}
\urladdr{\url{https://math.hws.edu/simmons/}}
\address {Department of Mathematics, University of Pennsylvania, 209 South 33rd Street, Philadelphia, PA 19104-6395, USA}
\email{htowsner@math.upenn.edu}
\urladdr{\url{http://www.math.upenn.edu/~htowsner}}

\begin{abstract}
Using the functional interpretation from proof theory, we analyze nonconstructive proofs of several central theorems about polynomial and differential polynomial rings. We extract effective bounds, some of which are new to the literature, from the resulting proofs. In the process we discuss the constructive content of Noetherian rings and the Nullstellensatz in both the classical and differential settings. Sufficient background is given to understand the proof-theoretic and differential-algebraic framework of the main results. 
\end{abstract}

\maketitle

\section{Introduction}

This paper is concerned with proofs of finitary statements which pass through an ultraproduct construction as an intermediate step.  The basic idea is illustrated by the following theorem:
\begin{theorem*}[\!\!\cite{vdDS}, Theorem 2.5]
For every $n$ and $d$, there is a bound $b$ so that whenever $K$ is a field and $\Lambda$ is a finite set of generators in $K[X_1,\ldots,X_n]$ with total degree bounded by $d$, the following implication holds: if either $f\in(\Lambda)$ or $g\in(\Lambda)$ for all $fg\in(\Lambda)$ such that $fg$ has total degree $\leq b$,  then $(\Lambda)$ is prime.
\end{theorem*}
Their proof proceeds as follows: suppose this were false for some $n$ and $d$.  That is, for each $b$, there exists some field $k_b$ and some $\Lambda_b$ in $k_b[X_1,\ldots,X_n]$ with total degree bounded by $d$ satisfying the assumption but with $(\Lambda_b)$ not prime.  They take an ultraproduct $K=\prod_{\mathcal{U}}k_b$ and $\Lambda=\prod_{\mathcal{U}}\Lambda_b$ and then work in the ring $K[X_1,\ldots,X_n]$ to obtain a contradiction.  Details on the ultraproduct construction can be found, for instance, in \cite{keisler2010ultraproduct}, but this will not concern us here, since our interest is on how to eliminate---``unwind''---the use of ultraproducts.

The disadvantage to such a proof is that it appears to be non-constructive---one assumes, towards a contradiction, that the $\Lambda_b$ exist for all $b$, but this does not directly tell us how large the bound $b$ actually is.  By eliminating the ultraproduct construction from these proofs, we will obtain explicit calculations of these bounds.

\subsection{Unwinding Ultraproduct Proofs}

The essential technique comes from proof theory: one views a proof of a property $\sigma$ in an ultraproduct as a sequence of statements
\[\sigma_1,\ldots,\sigma_n,\sigma\]
where each step follows from the earlier ones.  In order to obtain a direct proof, we replace each step ``$\sigma_i$ is true in the ultraproduct'' with some related fact ``$\sigma'_i$ is true in every field''.  For most of the results we are interested in, the conclusion $\sigma$ is the same as $\sigma'$, which is, of course, the point.  The difficulty is that, for intermediate statements, this may not be true: sometimes $\sigma_i$ is true in an ultraproduct, but \emph{not} true in arbitrary fields.  In this case we need to replace $\sigma_i$ with some different formula $\sigma'_i$.

It turns out that the right translation is a tool called the \emph{monotone functional interpretation} \cite{MR1428007}.  The functional (or ``Dialectica'') interpretation was introduced by G\"odel \cite{MR0102482}; the monotone variant was developed by Kohlenbach to make it easier to apply to ordinary mathematical proofs.  (See also \cite{avigad:MR1640329,gerhardy2007proof, kohlenbach:MR2445721,towsner_worked, troelstra1973metamathematical} for more background on the functional interpretation.)

The functional interpretation tells us to look at the syntactic form of the statement $\sigma$ in (a suitable language of) first-order logic.  Our main conclusions, like Theorem 2.5 of \cite{vdDS}, turn out to be equivalent to statements where the relevant quantifiers have a $\forall\exists$ pattern---what are called $\Pi_2$ statements.  (As we will discuss below, it requires some care to see this since not all the quantifiers ``count'' towards this patten.  In many cases, including Theorem 2.5 of \cite{vdDS}, this equivalence is not obvious.)  Relatedly, the functional interpretation is essentially the identity on $\Pi_2$ statements, as we would expect.

Intermediate steps, however, may be more complicated.  For example, the proofs below will use Hilbert's Basis Theorem, which may be stated as:
\begin{quote}
  For every $n$ and every increasing sequence of ideals $I_1\subseteq I_2\subseteq\cdots\subseteq I_i\subseteq\cdots$ in $K[X_1,\ldots,X_n]$, there is an $m$ so that for all $m'>m$, $I_m=I_{m'}$.
\end{quote}
Specifically, the proofs we are interested in use the fact that Hilbert's Basis Theorem holds in an ultraproduct.  In order to obtain a quantitative proof of our main theorem, we need to translate Hilbert's Basis Theorem into a quantitative fact that holds in every field.

For our purposes, we will take the $I_i$ to be finitely generated.  Carefully formalizing this in the right language leads to a $\Pi_3$ statement:
\[\forall \{D_i\}\forall\{\Lambda_i\}\ \exists m\ \forall m'\ \cdots\]
where each $D_i$ is a natural number and each $\Lambda_i$ is a finite set of polynomials of degree $\leq D_i$\footnote{Technically, first-order logic cannot include the outer quantifier over the $\{D_i\},\{\Lambda_i\}$, but a standard trick is to add symbols to the language which will represent these sequences, and this is equivalent to allowing a single universal quantifier on the outside of the formula, which is precisely what we need here.}.
The functional interpretation tells us to replace this with a function bound \cite{Hertz}:
\begin{quote}
  For every $n$, every function $\fn{D}$, and every function $\fn{F}$, there is an $M$ so that whenever ($\Lambda_1)\subseteq(\Lambda_2)\subseteq\cdots$ is an increasing sequence of finitely-generated ideals in $K[X_1,\ldots,X_n]$ where each polynomial in $\Lambda_i$ has total degree $\leq \fn{D}(i)$, there is an $m\leq M$ so that $\Lambda_{\fn{F}(m)}\subseteq(\Lambda_m)$.
\end{quote}
The conclusion is weaker: we no longer ask for an $I_m=(\Lambda_m)$ which is the ultimate union of the sequence of ideals, but instead ask for a long interval on which the sequence seems to have stabilized.  In return, we obtain that the bound is uniform---it does not depend on the field $K$ or the particular polynomials $\Lambda_i$, only on bounds on the total degree of the $\Lambda_i$---and can be computed from $n$, $\fn{D}$, and $\fn{F}$.

(The functional interpretation of $\Pi_3$ statements always have roughly this form, characterized by the appearance of a function like $\fn{F}$.  Tao introduced the term \emph{metastability} to describe this property in the context of convergence of sequences \cite{taopigeonhole}.  Such statements also naturally occur from the use of Kreisel's no-counterexample interpretation \cite{MR0049135,MR0051193}.)

When Hilbert's Basis Theorem appears as an intermediate statement in a proof of a $\Pi_2$ statement, we never need the full quantitative version.  Instead, there is always a single choice of the function $\fn{F}$ which suffices to complete that proof; what the specific function $\fn{F}$ needed to complete the proof is must be extracted from the remainder of the proof.  (For the applications in this paper, we only need the case where $\fn{F}(i)=i+1$, and we will compute explicit bounds in this case below.)

What the functional interpretation guarantees is that any application of Hilbert's Basis Theorem as an intermediate step in the proof of a $\Pi_2$ statement can actually be replaced by the version with a function bound for the right choice of $\fn{F}$: in other words, a proof of a $\Pi_2$ statement never needs to use an $I_m$ which is truly maximal; it always suffices to choose $I_m$ so that $I_{\fn{F}(m)}\subseteq I_m$ for some big enough function $\fn{F}(m)$.

\subsection{Insights from unwound proofs}
The functional interpretation is a formal technique whose properties are established by rigorous theorems (\!\!\cite{kohlenbach:MR2445721} presents many of the main technical results). Such theorems establish critical properties of the functional interpretation such as modularity: rather than transforming an entire proof at once, one can refine large steps into sublemmas whose interpretations are more tractable \cite{towsner_worked}.

Importantly, though, the reader need not be familiar with the details in order to understand the output of the method, which can be expressed as constructive arguments in standard mathematical language. Even for the practitioner, the functional interpretation can be used as a heuristic for converting unformalized nonconstructive proofs into algorithmic ones without having to go through a formal language \cite{gerhardy2007proof}, \cite{towsner_worked}. Accordingly, throughout the paper we keep the machinery of the functional interpretation in the background while retaining the product: explicit procedures for computing desired objects and quantitative bounds on the complexity of those procedures.

A great virtue of the functional interpretation is that it is systematic and applies in many situations. Proofs often have  a ``hidden combinatorial core'' \cite{kohlenbach:MR2445721} that the functional interpretation can identify. (For instance, the intricate combinatorics in Szemer\'{e}di's proof of his regularity lemma \cite{MR540024} automatically emerge from the functional interpretation \cite{tao2005szemer, goldbring2014approximate}.) 

In this paper, we aim to use the perspective provided by the functional interpretation to analyze ultraproduct proofs of the Nullstellensatz and related results found in \cite{vdDS} and \cite{HTKM}. Our systematic use of tactics suggested by the functional interpretation gives an alternate route to effective versions of these important theorems. 







\subsection{Plan of the paper}

We briefly review the Nullstellensatz and differential Nullstellensatz prior to outlining our path in the rest of the paper. A standard form of Hilbert's Nullstellensatz states that for an algebraically closed field $K$, given an ideal $I\subseteq K[X_1,\dots,X_n]$, the radical ideal $\sqrt{I}$ consists of all polynomials over $K$ that vanish on the common zero locus $\mathbf{V}(I)$ of $I$. This is a nonconstructive statement, owing to the existential nature of the definition of radical ideals, and the usual proofs are also ineffective. 

The ``effective Nullstellensatz'' is the problem of finding uniform bounds on radical ideal membership valid for all fields and only depending on the number of algebraic unknowns and degrees of the generators. Brownawell, Koll\'{a}r, Dub\'{e}, and subsequent authors have employed analytic, algebraic, and combinatorial techniques to show that single exponential bounds suffice \cite{brownawell1987bounds, kollar1988sharp, dube1993combinatorial}. In contrast, van den Dries and Schmidt justify their focus on nonstandard methods by observing that ``by concentrating on \emph{existence proofs} for bounds, rather than on their construction, it is possible to gain a lot of efficiency of exposition'' \cite{vdDS}. We do not explicitly state their results and arguments, but we cite the corresponding nonconstructive analogues. Our unwindings of their proofs show that van den Dries and Schmidt's ultraproduct strategy is not only elegant, but also (implicitly) preserves much more  effective content than one might suppose. 



Our other basic source for nonstandard proofs is \cite{HTKM} by Harrison-Trainor, Klys, and Moosa, who adapt the techniques of van den Dries and Schmidt to the more complicated differential case. Differential fields enrich the field structure by adding commuting derivations (additive endomorphisms obeying the usual product rule for derivatives). Ritt \cite{ritt1932differential} and Raudenbush \cite{raudenbush1934ideal, raudenbush1936analog} enunciated  differential-algebraic versions of the basis theorem and Nullstellensatz, the latter of which Cohn \cite{cohn1941analogue} and Seidenberg \cite{seidenberg1956elimination} approached from an algorithmic angle (without giving explicit bounds). The effective differential Nullstellensatz consists of giving bounds on radical differential ideal membership or, equivalently, consistency of systems of polynomial differential equations.

Recently there has been considerable interest in analyzing the effective content of the differential Nullstellensatz \cite{MR2556127, d2014effective,gustavson2016new, MR3448164, gustavson2016effective}, with the methods employed coming from algebra and model theory \cite{socias1992length, pierce2014fields,freitag2016effective}. Other constructive problems in differential algebra and differential algebraic geometry have also gained attention\cite{grigor1987complexity, boulier1995representation, hrushovski2000effective, freitag2014strong, freitag2015differential, li2016computation, gustavson2016bounds}.







In the rest of the section, we preview the technical part of the paper. We list our main results and indicate how they should be understood.  Because our results require stating a series of explicit functions bounding various properties, we include an index to where the definitions of these functions can be found and, where appropriate, where bounds on their rate of growth are proven:

\begin{table}[h]\caption{Table of Notations and Bounds}
\begin{tabular}{lcc}
\hline
Function & Definition & Calculated Bounds\\
\hline
 $\mathfrak{d}_n$ &Notation \ref{note:dn} (p. \pageref{note:dn})&\\ 
$\mathfrak{e}$& Notation \ref{note:enb} (p. \pageref{note:enb})&\\
$\zeta_n$& Notation \ref{note:zeta} (p. \pageref{note:zeta})&\\
$\mathfrak{p}_n$& Notation \ref{note:pn} (p. \pageref{note:pn})&Lemma \ref{thm:pn_bounds} (p. \pageref{thm:pn_bounds})\\
$\mathfrak{m}$&Notation \ref{not:local_noetherian} (p. \pageref{not:local_noetherian})&Lemma \ref{thm:m_bound} (p. \pageref{thm:m_bound})\\
$\mathfrak{m}^*$&Notation \ref{not:local_noetherian} (p. \pageref{not:local_noetherian})&Lemma \ref{thm:m_bound} (p. \pageref{thm:m_bound_star})\\
$\mathfrak{g}$& Notation \ref{note:g} (p. \pageref{note:g})&\\
$\mathfrak{u}_{\fn{F}}$& Notation \ref{note:uf} (p. \pageref{note:uf})&Lemma \ref{thm:uf_bounds} (p. \pageref{thm:uf_bounds})\\
$\mathfrak{u}_{\fn{F}}^+$& Notation \ref{note:uf_plus} (p. \pageref{note:uf_plus})&Lemma \ref{thm:uf_bounds} (p. \pageref{thm:uf_bounds})\\
$\mathfrak{f}$& Notation \ref{note:fF} (p. \pageref{note:fF})&Lemma \ref{thm:fF_bounds} (p. \pageref{thm:fF_bounds})\\
$\mathfrak{h}_{n,m}$& Notation \ref{note:hgamma} (p. \pageref{note:hgamma})&Lemma \ref{thm:hgamma_bounds} (p. \pageref{thm:hgamma_bounds})\\
$\mathfrak{i}^{\mathrm{sat}}_{n,m}$&Notation \ref{note:isat} (p. \pageref{note:isat}) & Lemma \ref{thm:isat_bounds} (p. \pageref{thm:isat_bounds})\\
$\mathfrak{i}^{\mathrm{cohere}}_{n,m}$&Notation \ref{note:icohere} (p. \pageref{note:icohere}) & Lemma \ref{thm:icohere_bounds} (p. \pageref{thm:icohere_bounds})\\
$\mathfrak{z}^k$&Notation \ref{note:zk} (p. \pageref{note:zk}) & Lemma \ref{thm:zk_bounds} (p. \pageref{thm:zk_bounds})\\
$\mathfrak{i}^{\mathrm{char}}_{n,m}$&Notation \ref{note:ichar} (p. \pageref{note:ichar}) & Lemma \ref{thm:ichar_bounds} (p. \pageref{thm:ichar_bounds})\\
$\mathfrak{k}$&Theorem \ref{thm:eff diff null} (p. \pageref{thm:eff diff null}) & Lemma \ref{thm:diff null bounds} (p. \pageref{thm:diff null bounds}) \\
$\mathfrak{j}_{n,m}$&Notation \ref{not:j_function} (p. \pageref{not:j_function}) & Lemma \ref{thm:ritt bounds} (p. \pageref{thm:ritt bounds})\\
\hline
\end{tabular}
\end{table}


\newpage

Our goal in Section \ref{sec:explicit} is to unwind the existence proof of a bound on prime ideals given by Theorem 2.5 of \cite{vdDS}. The main ingredients are finitary counterparts of prime ideals and vector space bases, as well as bounds on flat extensions of polynomial rings (Theorem \ref{thm:flat} and Lemma \ref{thm:faithful}). The result \ref{thm:vdds_2.5}, which has a form typical of others in the paper, is

\begin{theorem*}Let $n, d$ be given.  If $\Lambda\subseteq K[X_{[n]}]_{\leq d}$ is such that $(\Lambda)$ is prime up to $\mathfrak{p}_n(d)$ then $(\Lambda)$ is prime.
\end{theorem*}

The subscript denotes a bound on the degree of the polynomials in question (see Definition \ref{def:bound poly}). ``Primality up to some value'' (Definition \ref{21}) is a ``local'' - in particular, easily seen to be computable - notion of primality suggested by the functional interpretation.  The symbol $\mathfrak{p}_n$ represents a certain recursively-defined bounding function (Notation \ref{note:pn}) on the degree of possible counterexamples to primality of an ideal $I\subseteq K[X_1,\dots,X_n]$, where $K$ is an arbitrary field. Such a bound is implicit in the sense that we must analyze the recursive definition in order to establish the growth rate of $\mathfrak{p}_n$ in comparison to some well-known benchmark. See Section \ref{sec:bounds explicit} for such an analysis.

Our actual bounds tend to be rough, and it is not surprising that in many cases carefully optimized arguments (e.g.,Theorem 3.4 of \cite{gustavson2016effective}) give tighter bounds. The functional interpretation's output is dependent on its input and so cannot improve on the implicit constructive content of a given ineffective proof. It is nonetheless meaningful to expose that content, especially since general classes of bounds are often of most intrinsic interest. Like the bounds obtained in \cite{socias1992length, MR2556127,gustavson2016effective,MR3448164}, our main bounds are non-primitive recursive. This indicates either an actual complexity barrier or the need for fundamentally new ideas that can qualitatively lower the bounds beyond what any existing proofs provide.

In Section \ref{sec:interpretation}, we will give a short introduction to the version of the functional interpretation being used to produce these quantitative versions, using the results of Section \ref{sec:explicit} as examples.




Section \ref{sec:local noetherian} deals with the underlying complexity of Noetherianity and its consequences. Using a bound on Dickson's Lemma from the literature (Theorem \ref{Dickson bound}) as a shortcut, we unwind proofs of Hilbert's Basis Theorem and the Nullstellensatz (Theorems \ref{local noetherianity} and \ref{thm:vdDS_2.7}):

\begin{theorem*}
Suppose $(\Lambda_1)\subseteq (\Lambda_2)\subseteq \cdots \subseteq K[X_{[n]}]$ with $\Lambda_i\subseteq K[X_{[n]}]_{\leq{\bf{D}}(i)}$. Then there is a $j\leq  \mathfrak{m}^*({\bf D},n)$ such that $(\Lambda_{j+1})\subseteq (\Lambda_j)$.
\end{theorem*}

\noindent (Here $\fn{D}$ denotes a given nondecreasing function from $\mathbb{N}$ to $\mathbb{N}$ and $\mathfrak{m}^*$ is the aforementioned bound on Dickson's Lemma.)

\begin{theorem*}[Based on \cite{vdDS}, Cor 2.7(ii)]
For any $n,d$ there is $m={\mathfrak{m}^*(i\mapsto \mathfrak{p}_n^i(d),n)}$ so that if $\Lambda\subseteq K[X_{[n]}]_{\leq d}$ and $f^k\in(\Lambda)$ (for any $k$) then $f=\sum_{i}c_ir_i$ where each $r_i^{2^m}\in(\Lambda)$.
\end{theorem*}

The use of Noetherianity is the key factor driving our bounds, as well as those found in other papers that examine the complexity of the differential Nullstellensatz. Morally, many of the bounds we extract are non-primitive recursive because the original nonconstructive proof invokes Noetherianity. (Moreno Soc\'{i}as proved non-primitive recursiveness of bounds on Hilbert's Basis Theorem in \cite{socias1992length}. Also, Simpson has shown in the sense of reverse mathematics that proving Hilbert's Basis Theorem is equivalent to proving that Ackermann's function--well known to not be primitive recursive--is a total function \cite{simpson:MR961012}.) The resulting bounds, as shown in Section \ref{sec:bounds explicit}, are far larger than the doubly-exponential bounds from the slick proof of the effective Nullstellensatz in \ref{thm:rad_limit} which does not appeal to Noetherianity.



Section \ref{sec:bounds differential} is the longest of the paper and establishes many of our basic effective results on differential polynomial rings. Our treatment is largely self-contained, but additional details on differential algebra are found in, e.g., \cite{kaplansky1957introduction, MR0568864, buium1999differential, MR1921694, Marker:model_thy_of_fields}. 
After describing the framework we establish bounds relating differential ideals and algebraic ideals (for example, bounds on the complexity of coherent sets and Rosenfeld's lemma, Proposition \ref{coherent bound} and Lemma \ref{thm:Kolchin_3.8.5}). The basic ingredient for these bounds is a quantitative version of the theorem that there are no infinite descending chains of autoreduced sets (Theorem \ref{thm:dickson_iteration_bound}). 

Our efforts in this section culminate in several new bounds. In \cite{freitag2015differential}, Freitag, Li, and Scanlon remark that ``producing explicit equations for differential Chow varieties in specific cases would require effectivizing Theorem 6.1'' of \cite{HTKM}, which is precisely the content of Theorem \ref{thm:char set} and its corollary \ref{thm:min prime ideal}: 

\begin{cor*}
Suppose $\Lambda\subseteq K\{X_{[n]}\}_{\leq b}$ and let $P$ be a minimal prime $\Delta$-ideal containing $\Lambda$.  Then $P$ has a characteristic set $\Sigma\subseteq K\{X_{[n]}\}_{\leq \mathfrak{i}^{\mathrm{char}}(b)}$.
\end{cor*}

\noindent (In the differential setting, subscripts now indicate bounds on order as well as degree; see Definition \ref{def:diff degree}.) Moreover, Corollary \ref{thm:partial prime bound} and Lemma \ref{thm:ichar_bounds} give an explicit bound, not on primality of differential ideals (which is open and equivalent to the well-known Ritt problem \cite{golubitsky2009generalized}), but on the weaker Theorem 5.4 of \cite{HTKM} that bounds only one factor:

\begin{cor*}
Let  $\Lambda\subseteq K\{X_{[n]}\}_{\leq b}$ be given with $1\not\in[\Lambda]$. If either $f\in [\Lambda]$ or $g\in[\Lambda]$ for all $f,g\in K\{X_{[n]}\}$ with $fg\in[\Lambda]$ and $f\in K\{X_{[n]}\}_{\leq \mathfrak{i}^{\mathrm{char}}(b)}$, then $[\Lambda]$ is prime.
\end{cor*}



Section \ref{sec:ritt noetherian} is concerned with bounds on what we call ``Ritt-Noetherianity'', the Noetherianity of radical differential ideals. The finitary version \ref{thm:local ritt} of the Ritt-Raudenbush basis theorem is new:  

\begin{theorem*}
  Let $i_0, \Lambda, \Lambda_0\subseteq\Lambda_1\subseteq\cdots, \fn{D}, \fn{F}, d$ be given such that:
  \begin{itemize}
  \item $\Lambda\subseteq K\{X_{[n]}\}_{\leq d}$ is autoreduced, and
  \item $\Lambda_i\subseteq K\{X_{[n]}\}_{\leq\fn{D}(i)}$ for all $i$.
  \end{itemize}

Then there is an $i\in[i_0, \mathfrak{j}(i_0,\fn{D},\fn{F},d,\Lambda)]$ so that $\Lambda_{\fn{F}(i)}\subseteq\{\Lambda\cup\Lambda_i\}$.
\end{theorem*}

We analyze the corresponding explicit bounds in \ref{thm:ritt bounds}. Using our finitary basis theorem, it is possible to unwind the proof of the differential Nullstellensatz found in \cite{HTKM} (Corollary 4.5/Theorem 6.3), but we do not include the details here; see the discussion at the beginning of Section \ref{sec:ritt noetherian}.

With the unwinding work behind us, in Section \ref{sec:bounds explicit} we show how to interpret the bounds produced by the functional interpretation in preceding sections. For most of our results we analyze the functions' growth rates and find their place in the Grzegorczyk hierarchy of fast-growing functions \cite{odifreddi1999classical}. 
For instance, the bound $\mathfrak{p}_n$ on primality in Section \ref{sec:explicit} lies in the second stage of the fast-growing hierarchy. To minimize disruption, we place in Appendix \ref{sec:ordinals} the results on ordinal arithmetic needed to justify the calculations in this section. 







\section{Explicit Bounds for Testing Primality}\label{sec:explicit}

In \cite{vdDS}, van den Dries and Schmidt prove a number of results about ultraproducts of polynomial rings $K[X]$, and derive the existence of uniform bounds independent of $K$.  Since these results are used extensively in \cite{HTKM}, in this section we obtain effective versions using the methods described in the previous section.

We have two purposes: to demonstrate, in the simpler algebraic setting, the methods we will later use in the differential setting, and to produce actual explicit bounds we will need later.  In some cases, effective proofs have been given by other means (often before \cite{vdDS}), and these bounds are often substantially more efficient than those given by unwinding the ultraproduct arguments.  When this happens, we will sometimes simply cite the known bounds; at other times, unwinding the ultraproduct proof illustrates a useful technique, so we will also describe the less efficient proof.

Throughout this section we are concerned with an arbitrary field $K$ and its polynomial extension $K[X_1,\ldots,X_n]$.

\begin{notation}
  We abbreviate $K[X_1,\ldots,X_n]$ by $K[X_{[n]}]$.  More generally, we abbreviate $K[X_{i},X_{i+1},\ldots,X_{j}]$ by $K[X_{[i,j]}]$.
\end{notation}

We prefer this to the more common abbreviations $K[\vec X]$ or even $K[X]$ because we wish to be explicit about the number of variables.

\subsection{Internal Flatness and Faithful Flatness}

\begin{definition}\label{def:bound poly}
  We write $K[X_{[n]}]_{\leq d}$ for the set of polynomials in $K[X_{[n]}]$ of total degree at most $d$.

  We say $K[X_{[n]}]$ is \emph{internally flat bounded by ${\bf D}$} if for every $b$, whenever $f_1,\ldots,f_k\in K[X_{[n]}]_{\leq b}$ are coefficients of a homogeneous linear equation $\sum_i f_i y_i=0$ and $g_1,\ldots,g_k\in K[X_{[n]}]$ is a solution, there exist $h_{ij}\in K[X_{[n]}]_{\leq {\bf D}(b)}$ and $c_j\in K[X_{[n]}]$ so that $\sum_i f_ih_{ij}=0$ for each $j$ and $\sum_j c_jh_{ij}=g_i$ for each $i$.
\end{definition}
Internal flatness states that every solution to $\sum_i f_iy_i=0$ is a linear combination of solutions of bounded degree.  The name ``internal flatness'' refers to the fact when $K=\prod K_i$ is an ultraproduct, $K[X_{[n]}]_{int}=\prod K_i[X_{[n]}]$ is a flat extension of $K[X_{[n]}]$ if and only if there is some $\fn{D}$ so that most $K_i[X_{[n]}]$ are internally flat bounded by $\fn{D}$.

\begin{remark}
Although we will not need this notion, we can define internal flatness for any graded ring $R=\oplus_i R_i$ with $R_{\leq i}=\oplus_{|j|\leq i}R_j$: $R$ is internally flat bounded by ${\bf D}, {\bf S}$ if for every $k, b$, whenever $f_1,\ldots,f_k\in R_{\leq b}$ are coefficients of a homogeneous linear equation $\sum_i f_i y_i=0$ and $g_1,\ldots,g_k\in R$ are a solution, there exist $h_{ij}\in R_{\leq {\bf D}(k,b)}$ with $1\leq j\leq {\bf S}(k,b)$ and $c_j\in R$ so that $\sum_if_ih_{ij}=0$ for each $j$ and $\sum_j c_jh_{ij}=g_i$ for each $i$.

The bounds $k, {\bf S}$ are unnecessary for polynomial rings because $K[X_{[n]}]_{\leq d}$ is finite dimensional with dimension bounded in $n, d$.
\end{remark}

\begin{notation}\label{note:dn}
  We write $\mathfrak{d}_n(b)=(2b)^{2^n}$.
\end{notation}

\begin{theorem}[\!\!\cite{MR0349648,MR2051617}]\label{thm:flat}
  $K[X_{[n]}]$ is internally flat bounded by $\mathfrak{d}_n$.

  More generally, given a system of $m$ homogeneous equations with coefficients in $K[X_{[n]}]_{\leq b}$, the space of solutions is generated by solutions in $K[X_{[n]}]_{\leq \mathfrak{d}_n(mb)}$.
\end{theorem}

We also expect an analog of faithful flatness.  It is standard that a flat extension is faithfully flat exactly when the extension does not create solutions to any unsolvable inhomogeneous linear equations with coefficients from the base ring.  Then ``internal faithful flatness'' just says that if an inhomogeneous equation is solvable, the size of the solution should be bounded in the degrees of the cofficients.  This is the same as giving bounds on ideal membership.
\begin{lemma}[\!\!\cite{MR1512302}]\label{thm:faithful}
For any $n$ and any $f_i\in K[X_{[n]}]_{\leq b}$, if $\sum_{i\leq k} f_ig_i=h$ where $h$ has degree $b$ then there are $g'_1,\ldots,g'_k\in K[X_{[n]}]_{\leq \mathfrak{d}_n(b)}$ such that $\sum_{i\leq k}f_ig'_i=h$.
\end{lemma}

\subsection{Bounds on Primality}

Working in the ultraproduct setting with $K=\prod_{\mathcal{U}}K_i$ and $K[X_{[n]}]_{int}=\prod_{\mathcal{U}}(K_i[X_{[n]}])$, van den Dries and Schmidt \cite{vdDS} show
\begin{theorem}
  If $I$ is an ideal in $K[X_{[n]}]$ then $I$ is prime iff $IK[X_{[n]}]_{int}$ is prime in $K[X_{[n]}]_{int}$.
\end{theorem}

The standard analog of this is
\begin{theorem}
There is a function $\mathfrak{p}_n(b)$ so that for any $\Lambda\subseteq K[X_{[n]}]_{\leq b}$ with $(\Lambda)$ not prime, there are $f, g\in K[X_{[n]}]_{\leq \mathfrak{p}_n(b)}$ so that $fg\in(\Lambda)$ but $f,g\not\in(\Lambda)$.
\end{theorem}
In \cite{schmidt1987bounds} Schmidt-G\"ottsch shows that $\mathfrak{p}_n(b)$ has the form $b^{\beta(n)}$ for some $\beta$.  Here we extract bounds with a worse dependence on $b$ directly from the simpler proof given in \cite{vdDS}.

\begin{notation}\label{note:enb}
  $\mathfrak{e}(n,b)=2^{(b+\mathfrak{d}_{n-1}(b))^{n-1}+1}b+b+\mathfrak{d}_{n-1}(b)$.
\end{notation}

\begin{lemma}
  Suppose $\phi:K(X_1)\rightarrow L$ is a field extension and $\lambda_1,\ldots,\lambda_k\in K[X_{[n]}]_{\leq b}$.  Writing $\phi$ for the map $\phi:K(X_1)[X_{[2,n]}]\rightarrow L[X_{[2,n]}]$ as well, any solution in $L$ to $\sum_i \phi(\lambda_i)y_i=0$ is a linear combination of images under $\phi$ of solutions from $K[X_{[n]}]_{\leq \mathfrak{e}(n,b)}$.
\end{lemma}
\begin{proof}
  By internal flatness, solutions to $\sum_i\phi(\lambda_i)y_i$ are linear combinations of solutions from $L[X_{[2,n]}]_{\leq \mathfrak{d}_{n-1}(b)}$.  Let $M_1,\ldots$ list the $\leq (b+\mathfrak{d}_{n-1}(b))^{n-1}$ monomials of degree $\leq b+\mathfrak{d}_{n-1}(b)$ in $X_{[2,n]}$; then we may rewrite $\sum_i\phi(\lambda_i)y_i=0$ as 
\[\sum_i\phi(\sum_j f_{i,j}M_j)(\sum_{j'}y_{i,j'}M_{j'})=0\]
where $f_{i,j}\in K[X_1]_{\leq b}$.  So we may expand this into a system of $\leq (b+\mathfrak{d}_{n-1}(b))^{n-1}$ equations of the form
\[\sum_i\sum_j \phi(f_{i,j})y_{i,j_0-j}=\sum_k \phi(g_{k,j})x_k=0.\]

The solutions to a single equation $\sum_i \phi(g_{k,j_0})x_k=0$ are generated by the solutions of the form $(\phi(g_{k,j_0}),\ldots,-\phi(g_{1,j_0}),\ldots)$ (because $L$ is a field); by substituting $x_1=\sum_k z_k\phi(g_{k,j_0})$ and $x_k=-z_k\phi(g_{1,j_0})$, we obtain a system of equations with one fewer equation and coefficients in $\phi(K[X_1]_{\leq 2b})$.

Repeating this, we eventually reduce to a single equation whose solutions are generated by the image of solutions from $K[X_{[1]}]_{\leq 2^{(b+\mathfrak{d}_{n-1}(b))^{n-1}}b}$.  Undoing the sequence of substitutions, we see that the original $x_i$ are generated by the images of solutions from $K[X_{[1]}]_{\leq 2^{(b+\mathfrak{d}_{n-1}(b))^{n-1}+1}b}$, and so the $y_i$ are generated by the images of solutions from $K[X_{[2,n]}]_{\leq 2^{(b+\mathfrak{d}_{n-1}(b))^{n-1}+1}b+b+\mathfrak{d}_{n-1}(b)}$.
\end{proof}

\begin{lemma}[Based on {\cite{vdDS}}, Lemma 2.3]\label{thm:vdds_2.3}
For any $n,b$ and $\Lambda\subseteq K[X_{[n]}]_{\leq b}$, if $f\in K[X_1]$ has degree $>\mathfrak{e}(n,b)$ and is irreducible then for any $g\in K[X_{[n]}]$ such that $fg\in(\Lambda)$, also $g\in(\Lambda)$.
\end{lemma}
\begin{proof}
Let $\lambda_1,\ldots$ enumerate $\Lambda$.  Suppose $fg=\sum_i a_i\lambda_i$.  Since $f$ is irreducible, $L=K[X_1]/(f)$ is a field; let $\phi:K\rightarrow L$ be the natural embedding.  We have a solution $\sum_i \phi(a_i)\phi(\lambda_i)=0$ in $L[X_2,\ldots,X_n]$.  By the previous lemma, the solutions are generated by the images of solutions from $K[X_{[n]}]_{\leq \mathfrak{e}(n,b)}$.

Since $f$ has degree $>\mathfrak{e}(n,b)$, we also have $a_i=\sum_j c_ja_{ij}+fq_i$.  Therefore
\begin{align*}
  fg
&=\sum_i a_i\lambda_i\\
&=\sum_i(\sum_j c_ja_{ij}+fq_i)\lambda_i\\
&=\sum_j c_j\sum_i a_{ij}\lambda_i +f\sum_i \lambda_i q_i\\
&=f\sum_i\lambda_iq_i
\end{align*}
and therefore $g=\sum_i\lambda_iq_i\in(\Lambda)$.
\end{proof}

\begin{lemma}[Based on \cite{vdDS}, Corollary 2.4]\label{thm:irreducibles}
For any $n$, $b$ and any $\Lambda\subseteq K[X_{[n]}]_{\leq b}$, one of the following holds:
\begin{itemize}
\item there is an $f\in K[X_1]_{\leq \mathfrak{e}(n-1,k)}$ and a $g$ so $\deg(g)\leq \mathfrak{e}(n,b)+\mathfrak{d}_{n-1}(b)+b$, $fg\in(\Lambda)$, but $g\not\in(\Lambda)$, or
\item whenever $f\in K[X_1]$ and $fg\in(\Lambda)$, $g\in(\Lambda)$.
\end{itemize}
\end{lemma}
\begin{proof}
Suppose there is some $f\in K[X_1]$ and some $g$ so that $fg\in(\Lambda)$ but $g\not\in(\Lambda)$.  We may choose $f$ with minimal degree such that this happens.  Then $f$ is irreducible---if $f=f_0f_1$ then $f_0(f_1g)\in(\Lambda)$, so either $f_1g\not\in(\Lambda)$ (so $f_0$ is a witness of smaller degree) or $f_1g\in(\Lambda)$ (so $f_1$ is a witness of smaller degree).  By the previous lemma, $f\in K[X_1]_{\leq \mathfrak{e}(n,b)}$.  

We have $fg=\sum_i a_i\lambda_i$.  Let $L=K[X_1]/(f)$ and let $\phi:K[X_1]\rightarrow L$ be the natural embedding, so $\sum_i \phi(a_i)\phi(\lambda_i)=0$, and so by internal flatness, $\phi(a_i)=\sum_j c_j a_{ij}$ where $a_{ij}\in L[X_{[2,n]}]_{\leq \mathfrak{d}_{n-1}(b)}$.  We may assume $a_{ij}=\phi(a'_{ij})$ with $a'_{ij}\in K[X_{[n]}]_{\leq \mathfrak{e}(n,b)+\mathfrak{d}_{n-1}(b)}$.  We have $a_i=\sum_j c_ja_{ij}+fq_i$ and $\sum_i \lambda_i a_{ij}=fq'_j$.  Since $\deg(\lambda_i a_{ij})\leq \mathfrak{e}(n,b)+\mathfrak{d}_{n-1}(b)+b$, also $\deg(q'_j)\leq \mathfrak{e}(n,b)+\mathfrak{d}_{n-1}(b)+b$.

There must be some $j$ so $q'_j\not\in(\Lambda)$, and therefore $f,q'_j$ is the witness to the first case.  Otherwise, towards a contradiction, each $q'_j=\sum_{i}\lambda_{i} q'_{ji}$, and therefore
\begin{align*}
fg
&=\sum_ia_i\lambda_{i}\\
&=\sum_i (\sum_j c_ja_{ij}+fq_i)\lambda_{i}\\
&=\sum_j c_j\sum_i a_{ij}\lambda_{i}+\sum_i fq_i\lambda_{k,i}\\
&=\sum_j c_j fq'_j+\sum_i fq_i\lambda_{i}\\
&=f\sum_jc_j\sum_i\lambda_iq'_{ji}+f\sum_i q_i\lambda_{i}\\
&=f\sum_i\lambda_i(\sum_j c_jq'_{ji}+q_i),
\end{align*}
and therefore $g=\sum_i \lambda_{i}(\sum_j c_jq'_{ji}+q_i)$ so $g\in(\Lambda)$, giving the needed contradiction.
\end{proof}

\begin{definition}\label{21}
  We say an ideal $I\subseteq K[X_{[n]}]$ is \emph{prime up to $b$} if whenever $fg\in(\Lambda)$ with $f,g\in K[X_{[n]}]_{\leq b}$, either $f\in(\Lambda)$ or $g\in(\Lambda)$.
\end{definition}

\cite{vdDS} shows that if $I\subseteq K(X)[Y_{[n]}]$ is a prime ideal then the ideal it generates is prime in $K(X)_{int}[Y_{[n]}]$.  In our setting, this amounts to comparing two different gradings on $K(X)[Y_{[n]}]$: we could assign $X^i\prod_j Y_j^{k_j}$ either the grade $|i|+\sum_j k_j$ or the grade $\sum_j k_j$.  We wish to show that sufficient primality in the first grading implies primality in the second grading.

\begin{notation}
  We write $K(X)[Y_{[n]}]_{\leq_Y r}$ for those elements of $K(X)[Y_{[n]}]$ whose total degree in the $Y$ variables is at most $r$.  We write $K(X)[Y_{[n]}]_{\leq_{X,Y}r}$ for those elements of $K(X)[Y_{[n]}]$ whose total degree in $X, 1/X, Y_{[n]}$ is at most $r$.
\end{notation}

The proof in the ultraproduct involves using the fact that $K_{int}(X)$ is freely generated over $K(X)$.  Therefore, given $fg=\sum_i a_i\lambda_i$ where the $a_i$ come from $(K_{int}(X))[Y_{[n]}]$, we can view the $a_i$ as coming from $K(X,Z_{[m]})[Y_{[n]}]$ where the $Z_{[m]}$ are a basis for some subspace large enough to contain the $a_i$.

In the finitary world, the analog of the basis $Z_{[m]}$ is a ``local basis'': a collection of elements $Z_1,\ldots,Z_m$ such that, on the one hand, each $a_i$ is algebraic in $Z_{[m]}$ using ``small'' coefficients (in the sense of the grading), but there are no algebraic dependencies among the $Z_{[m]}$ even using much larger coefficients.

\begin{definition}
We write $K(X)_{\leq d}$ for $K[X,1/X]_{\leq d}$.  Let $S$ be a set of elements in $K(X)$ and let ${\bf F}:\mathbb{N}\rightarrow\mathbb{N}$.  An \emph{${\bf F}$-local basis} for $S$ is a set $Z$ and a bound $w$ such that:
\begin{itemize}
\item $S\subseteq K(X,Z)_{\leq w}$,
\item if $z\in Z$ then $z\not\in K(X,Z\setminus\{z\})_{\leq {\bf F}(w)}$.
\end{itemize}
\end{definition}

\begin{lemma}\label{thm:local_basis}
  For any $S\subseteq K(X,S)$ and any ${\bf F}$, letting ${\bf F}'(x)=x{\bf F}(x)$, there is an ${\bf F}$-local basis $Z,w$ such that $Z\subseteq S$ and $w\leq ({\bf F}')^{|S|}(1)$.
\end{lemma}
\begin{proof}
  Set $S_0=S$ and $w_0=1$.  Given $S_i, w_i$, if this is an ${\bf F}$-local basis for $S$, we are done.  Otherwise, define $S_{i+1},w_{i+1}$ as follows: chose $z\in S_i$ so that $z\in K(X,S_i\setminus\{s\})_{\leq {\bf F}(w_i)}$ and set $S_{i+1}=S_i\setminus\{z\}$ and $w_{i+1}=w_i{\bf F}(w_i)$.  Then for each $s\in S$, since $s\in K(X,S_i)_{\leq w_i}$, also $s\in K(X,S_{i+1})_{\leq w_i {\bf F}(w_i)}$.

Since $S_{i+1}\subsetneq S_i$, this process stops in at most $|S|$ steps.
\end{proof}

\begin{notation}\label{note:zeta}\ \\
\begin{itemize}
  \item $\zeta_0(n,d)={{n+\mathfrak{d}_n(d)}\choose n}$, the number of monomials of degree $\leq \mathfrak{d}_n(d)$ in $n$ variables,
  \item $\zeta_1(n,d,b)=({b+n\choose n}+2)\zeta_0(n,d)$,
  \item $\zeta_2(n,d,b)=(\zeta_1(n,d,b)+1)^{2^{\zeta_1(n,d,b)}-1}$.
  \end{itemize}
\end{notation}

This leads to the following crucial result.  We show that sufficient primality in the sense of $\leq_{X,Y}$ implies primality in the sense of $\leq_Y$.  This is a weak form of the result we are attempting to prove: we begin with an ideal which is prime only for $f,g$ with total degree in both $X$ and $Y$ bounded, and we obtain primality for $f,g$ with $Y$-degree bounded, but arbitrary $X$-degree.

\begin{lemma}\label{thm:prime_ext_alg_reg}
Let $n, b\leq d$ be given.  Then whenever $\Lambda\subseteq K(X)[Y_{[n]}]_{\leq_{X,Y}b}$ so that $(\Lambda)$ is prime up to $\zeta_1(n,d,b)\zeta_2(n,d,b)$ in the sense of $\leq_{X,Y}$, also $(\Lambda)$ is prime in $K(X)[Y_{[n]}]$ up to $d$ in the sense of $\leq_Y$.
\end{lemma}
\begin{proof}
Let $d$ and $\Lambda\subseteq K(X)[Y_{[n]}]_{\leq_{X,Y}b}$ be given so that $(\Lambda)$ is prime up to $\zeta_1(n,d,b)\zeta_2(n,d,b)$ in the sense of $\leq_{X,Y}$.

Let $f,g\in K(X)[Y_{[n]}]_{\leq_Y d}$ be given with $fg\in(\Lambda)$.  This implies that $fg=\sum_ia_i\lambda_{i}$ and, by Lemma \ref{thm:faithful}, we may assume the $a_i\in K(X)[Y_{[n]}]_{\leq_Y \mathfrak{d}_n(d)}$.  Note that we may assume $|\Lambda|\leq {{b+n}\choose n}$, the dimension of $K(X)[Y_{[n]}]_{\leq_Y b}$ as a vector space over $K(X)$.

We enumerate the monomials in $Y_{[n]}$ appearing in the $a_i$ as $M_0,\ldots,M_j,\ldots$.  There are at most $\zeta_0(n,d)$ such monomials.  We may write $a_i=\sum_j a_{ij}M_j$, $f=\sum_j u_jM_j$, and $g_j=\sum_jv_jM_j$ where $a_{ij},u_j,v_j$ are elements of $K(X)$.

Let $S_0=\{a_{ij}\}_{i,j}\cup\{u_j,v_j\}_j$.  Note that $|S_0|\leq\zeta_1(n,d,b)$.  Let ${\bf F}$ be the function given by ${\bf F}(x)=x\zeta_1(n,d,b)+1$.  By Lemma \ref{thm:local_basis} there is an $S\subseteq S_0$ and a $w\leq\zeta_2(n,d,b)$ so that $S,w$ is an ${\bf F}$-local basis for $S_0$.

We have
\[(\sum_j u_jM_j)(\sum_j v_jM_j)=fg=\sum_i\sum_j a_{ij}M_j\lambda_{i}.\]
Writing each $u_i,v_i,a_{ij}$ as an element of $K(X,S)_{\leq w}$---that is, as a rational polynomial involving $X,S$ where the degrees of the top and bottom add to at most $w$---we may multiply through to clear denominators.  So we have
\[(\sum_j u'_jM_j)(\sum_j v'_jM_j)=\sum_i\sum_j a'_{ij}M_j\lambda_{i}\]
where the $u',v',a'_{ij}$ are polynomials in $X,S$ with degrees bounded by $w\zeta_1(n,d,b)$. 

We will now rearrange our sums to focus on monomials from $S$.  Write $M^*_0,\ldots,M^*_j,\ldots$ for the monomials in $S$ arranged so that $M^*_0=1$ and $M^*_iM^*_j=M^*_{ij}$.  We then write
\[(\sum_j u''_jM^*_j)(\sum_j v''_jM^*_j)=\sum_j (\sum_i a''_{ij}\lambda_{i})M^*_j\]
where the $u''_j,v''_j,a''_{ij}$ are elements of $K[X,Y_{[n]}]$ with $X$ degree bounded by $w\zeta_1(n,d,b)$ and $Y_{[n]}$ degree bounded by $\mathfrak{d}_n(d)$.  By our choice of pseudobasis, we can separate this out by monomial: for each $j$,
\[\sum_ia''_{ij}\lambda_{i}=\sum_{k_0+k_1=j}u''_{k_0}v''_{k_1}.\]

We follow the standard argument to solve this monomial by monomial, keeping track of bounds along the way.  We show by induction on $J$ that there are $k_0,k_1$ with $k_0+k_1=J$ so that for each $j<k_0$, $u''_j=\sum_i b_{ij}\lambda_{i}$ and for each $j<k_1$, $v''_j=\sum_i c_{ij}\lambda_{i}$.

Suppose we have chosen such $k_0,k_1$.  Then 
\[\sum_i a''_{iJ}\lambda_{i}=\sum_{j\leq J}u''_{j}v''_{J-j}=u''_{k_0}v''_{k_1}+\sum_{j<k_0}u''_jv''_{J-j}+\sum_{j<k_1}u''_{J-j}v''_j,\]
so
\[u''_{k_0}v''_{k_1}=\sum_i (a''_{iJ}+\sum_{j<k_0}v''_{J-j}b_{ij}+\sum_{j<k_1}u''_{J-j}c_{ij})\lambda_{i}.\]

Since $\Lambda$ is prime up to $w\zeta_1(n,d,b)$ in the sense of $\leq_{X,Y}$, we have either $u''_{k_0}=\sum_i b_{ik_0}\lambda_{i}$ (and we replace $k_0$ with $k_0+1$) or $v''_{k_1}=\sum_i c_{ik_1}\lambda_{i}$ (and we replace $k_1$ with $k_1+1$).

We may continue until either $k_0={{w\zeta_1(n,d)+w}\choose w}$ or $k_1={{w\zeta_1(n,d)+w}\choose w}$.  Suppose the first case happens (the second is symmetric); then we have
\[f=\sum_j u''_jM^*_j=\sum_j \sum_i b_{ij}\lambda_{i}M^*_j=\sum_i (\sum_j b_{ij}M^*_j)\lambda_{i},\]
and therefore $f=\sum_i b_i\lambda_{i}$.
\end{proof}

We now arrive at the main result of this section: showing that we can ``upgrade'' from internal primality up to a certain point to actual primality.

\begin{notation}\label{note:pn}\ \\
\begin{itemize}
\item $\mathfrak{p}_1(d)=d$,
\item $\upsilon(n,d)=\zeta_1(n-1,\mathfrak{p}_{n-1}(d),d)\zeta_2(n-1,\mathfrak{p}_{n-1}(d),d)$,
\item $\rho(n,d)=\max\{2{{\upsilon(n,d)+n}\choose n}\upsilon(n,d),\mathfrak{e}(n-1,d)\}$,
\item $\mathfrak{p}_n(d)=\rho(n,d)$.
\end{itemize}
\end{notation}

\begin{theorem}[Based on \cite{vdDS}, Theorem 2.5]\label{thm:vdds_2.5}
Let $n, d$ be given.  If $\Lambda\subseteq K[X_{[n]}]_{\leq d}$ is such that $(\Lambda)$ is prime up to $\mathfrak{p}_n(d)$ then $(\Lambda)$ is prime.
\end{theorem}
\begin{proof}
  By induction on $n$.  When $n=1$ this is straightforward: $(\Lambda)$ is principal iff there is a single element of $\Lambda$ generating the ideal.

So suppose $n>1$.  First, suppose that for each $i$ there is an $h_i\in K[X_i]_{\leq \rho(n,d)}$ with $h_i\in(\Lambda)$.  Then $K[X_{[n]}]/(\Lambda)$ is a field extension of $K$ where each $X_i$ is algebraic of degree $\leq \rho(n,d)$.  In particular, any element $f$ of $K[X_{[n]}]$ may be written $f=f_0+f'$ where $f'\in(\Lambda)$ and $f_0$ has total degree $\leq \rho(n,d)n$.  So if $fg\in(\Lambda)$ then we have $fg=(f_0+f')(g_0+g')=f_0g_0+c$ with $c\in(\Lambda)$.  Therefore $f_0g_0\in(\Lambda)$ and since $f_0g_0$ has degree $\leq 2\rho(n,d)n$, the assumption applies, and either $f_0\in(\Lambda)$ or $g_0\in(\Lambda)$.  Therefore either $f=f_0+f'\in(\Lambda)$ or $g=g_0+g'\in(\Lambda)$.

So suppose this does not hold: for some $i\leq n$, $K[X_i]_{\leq \rho(n,d)}\cap(\Lambda)=\emptyset$.  We will apply the inductive hypothesis to the ring $K(X_i)[X_{[1,i-1]},X_{[i+1,n]}]$.  By rearranging the variables, it suffices to assume $i=1$.

\begin{claim}
  For every $u\in K[X_1]$, if $fu\in(\Lambda)$ then $f\in(\Lambda)$.
\end{claim}
\begin{claimproof}
We apply Lemma \ref{thm:irreducibles}.  It suffices to rule out the first case: suppose there were an $f\in K[X_1]_{\mathfrak{e}(n-1,d)}$ and a $g\in K[X_{[n]}]_{\leq \mathfrak{e}(n-1,d)+\mathfrak{d}_{n-1}(d)+d}$ so that $fg\in(\Lambda)$ but $g\not\in(\Lambda)$.  Since, by assumption, $f\not\in(\Lambda)$, this violates the primality of $(\Lambda)$ up to $\mathfrak{e}(n-1,d)+\mathfrak{d}_{n-1}(d)+d$.
\end{claimproof}

\begin{claim}
  $(\Lambda)$ is prime in $K(X_1)[X_{[2,n]}]$ up to $\upsilon(n,d)$ in $X_{[1,n]}$-degree.
\end{claim}
\begin{claimproof}
Suppose $fg\in(\Lambda)$ with $f,g\in K(X_1)[X_{[2,n]}]_{\leq \upsilon(n,d)}$, so $fg=\sum_i a_i\lambda_{i}$.  Clearing denominators, $f'g'h=\sum_i a'_i\lambda_{i}$ where $h\in K[X_1]_{\leq 2{{\upsilon(n,d)+n}\choose n}\upsilon(n,d)}$ and $f',g'\in K[X_{[n]}]_{\leq\upsilon(n,d)}$.

Since $K[X_1]_{\leq 2{{\upsilon(n,d)+n}\choose n}}\cap(\Lambda)=\emptyset$, we have $h\not\in(\Lambda)$.  By primality of $(\Lambda)$ up to $2\upsilon(n,d)$, $f'g'\in(\Lambda)$, and so, without loss of generality, $f'\in(\Lambda)$.  Then $f=f'/h'$ for some $h'\in K[X_1]$, and since $f'=\sum_i b_i\lambda_i$, also $f=\sum_i (b_i/h')\lambda_i$, and therefore $f\in(\Lambda)$.
\end{claimproof}

\begin{claim}
$(\Lambda)$ is prime in $K(X_1)[X_{[2,n]}]$.
\end{claim}
\begin{claimproof}
Since $(\Lambda)$ is prime up to $\upsilon(n,d)$ in $K(X_1)[X_2,\ldots,X_n]$ in $X_{[1,n]}$-degree, by Theorem \ref{thm:prime_ext_alg_reg}, also $(\Lambda)$ is prime up to $\mathfrak{p}_{n-1}(d)$ in $K(X_1)[X_{[2,n]}]$ in $X_{[2,n]}$-degree.  By the inductive hypothesis applied to $K(X_1)[X_{[2,n]}]$, we have that $(\Lambda)$ is prime in $K(X_1)[X_{[2,n]}]$.
\end{claimproof}

We can now complete the proof: suppose $fg\in(\Lambda)$ in $K[X_{[n]}]$ (with $\deg(fg)$ arbitrary).  Then certainly $fg\in(\Lambda)$ in $K(X_1)[X_{[2,n]}]$, so without loss of generality, $f=\sum_i b_i\lambda_{i}$ with the $b_i\in K(X_1)[X_{[2,n]}]$.  Clearing denominators again, $fu=\sum_i b'_i\lambda_{i}$ with the $b'_i\in K[X_{[n]}]$ and $u\in K[X_1]$.  Applying the first claim above, we must have $f\in(\Lambda)$, completing the proof.
\end{proof}

\section{A Version of the Functional Interpretation}\label{sec:interpretation}

The results in the previous section and the remainder of the paper are produced by applying a syntactic translation---a version of the functional interpretation---to the original proofs.  Several versions of the functional interpretation have been developed for nonstandard analysis \cite{MR938372, MR2964881,MR3325571}.  The specific version we use is detailed in \cite{1804.10809}.  The results in this paper are not obtained by an entirely mechanical application of that translation.  As is typical in proof mining, a certain amount of ``hand optimization'' was necessary, as was some care in choosing the right formulations and right proofs to make the application of the functional interpretation more manageable.

While the full generality of the interpretation would extend this paper unreasonably, we can now outline some of the main ideas.  This discussion is purely motivation, and is not needed to follow the proofs in the remainder of the paper.

Fields and differential fields are given by first-order theories in the language of rings or differential rings (the language with symbols for $0$, $1$, $+$, $\cdot$ and, in the differential case, finitely many derivatives $\delta_1,\ldots,\delta_m$).  However the statements we are interested in---for instance, the theorems in the previous section---can not be expressed in this language.

One natural way to express these statements is by allowing quantifiers over natural numbers as well.  For example, consider ``internal flatness''---the fact that the ring $K[X_{[n]}]_{int}$ is flat over $K[X_{[n]}]$.  If we have already fixed $K[X_{[n]}]_{int}$ and polynomials $f_1,\ldots,f_k\in K[X_{[n]}]$, and a degree $d$, the statement
\begin{quote}
  every solution to $\sum_if_iy_i=0$ is a linear combination of solutions of degree $\leq d$
\end{quote}
is expressed by a first-order formula (with parameters for the elements of $K$).  To get the general statement of flatness, we need to quantify over the degree $d$ and over the degree of the polynomials $f_1,\ldots,f_k$ and the number of variables---the flatness of $K[X_{[n]}]_{int}$ is expressed by
\begin{align*}
\forall^{\mathbb{N}}n\ \forall^{\mathbb{N}}b\ \forall f_1,\ldots,f_k\in K[X_{[n]}]_{\leq b}\ \exists^{\mathbb{N}}d &\text{ every solution to $\sum_if_iy_i=0$ is a linear }\\
&\text{ combination of solutions of degree }\leq d.
\end{align*}
Taken literally, we might expect that we need to quantify over the number of polynomials as well, but the dimension of the space of polynomials is bounded already by $b$ and $n$.  (This is a small example of the sort of hand optimization one can do to simplify the work needed.)  Strictly speaking, these quantifiers should be understood as countable conjunctions, not simply quantifiers, because the specific first-order formula depends on the values $b$ and $d$.

This is an example of what we mean by a $\Pi_2$ sentence for the purposes of this paper: the natural number quantifers follow the pattern $\forall\exists$.  The ``matrix''---the purely first-order part on the inside of the sentence---has no computational content (this is similar to the role of purely internal formulas in \cite{MR2964881}), so we only count the numeric quantifiers when considering the sentence's complexity.

Given that a sentence of this kind---a sentence built from first-order formulas using conjunctions and disjunctions---holds in every ultraproduct of rings, our interpretation translates it to some fact which holds in every ring.  (There is an equivalent, more purely syntactic version of this claim, which does not refer to ultraproducts.  We could work in a suitable theory of nonstandard arithmetic, with quantifiers over the standard natural numbers as in \cite{MR2964881} taking the place of our countable conjunctions and disjunctions.  The original ultraproduct proofs could be formalized in such a theory, and a suitable metatheory would show that it is possible to obtain a purely standard proof of the same conclusion.)

With $\Pi_2$ statements, the interpretation is quite direct: it tells us that in this case we can shift all the numeric quantifiers to the outside, as in
\begin{align*}
\forall^{\mathbb{N}}n\ \forall^{\mathbb{N}}b\ \exists^{\mathbb{N}}d\ \forall f_1,\ldots,f_k\in K[X_{[n]}]_{\leq b}&\text{ every solution to $\sum_if_iy_i=0$ is a linear }\\
&\text{ combination of solutions of degree }\leq d
\end{align*}
and then obtain a bound on $d$ as a function of $b$ and $n$.  That is, there is a function $\fn{D}$ so that for \emph{every} ring $K$,
\begin{align*}
\forall^{\mathbb{N}}n\ \forall^{\mathbb{N}}b\ \forall f_1,\ldots,f_k\in K[X_{[n]}]_{\leq b}&\text{ every solution to $\sum_if_iy_i=0$ is a linear }\\
&\text{ combination of solutions of degree }\leq \fn{D}(b).
\end{align*}
This is what we called internal flatness bounded by $\fn{D}$ in the previous section.

Slightly more precisely, if we write $\varphi(n,b,d)$ for the formula expressing
\begin{quote}
for any polynomials $f_1,\ldots,f_k$ of degree $\leq b$, every solution to $\sum_if_iy_i=0$ is a linear combination of solutions of degree $\leq d$,
\end{quote}
there is a theorem (for instance, a suitable formulation of the transfer theorem of nonstandard analysis) which tells us that an ultraproduct $\prod_{\mathcal{U}}K_i$ satisfies $\forall^{\mathbb{N}}n\forall^{\mathbb{N}}b\exists^{\mathbb{N}}d\ \phi(n,b,d)$ if and only if, for each $n$ and $b$, there is a $d$ so that
\[\{i\mid K_i\vDash\phi(n,b,d)\}\in\mathcal{U}.\]

The power of the functional interpretation is giving a version of this equivalence for more complicated statements.  For instance, suppose we know that
\[\prod_{\mathcal{U}}K_i\vDash \forall^{\mathbb{N}}x\exists^{\mathbb{N}}y\forall^{\mathbb{N}}z\ \psi(x,y,z)\]
where, for all natural numbers $x,y,z$, $\psi(x,y,z)$ is a first-order formula.  Then a theorem tells us that, for every $x$ and every function $\fn{Z}:\mathbb{N}\rightarrow\mathbb{N}$, there is a $y$ so that
\[\{i\mid K_i\vDash\phi(x,y,Z(y))\}\in\mathcal{U}.\]

The form of these bounds becomes progressively more complicated as the alternations of countable conjunctions and disjunctions becomes more complicated.  In particular, as the statements become more complicated, one can no longer find exact values with sufficient uniformity to get the equivalence we need; instead, one finds only \emph{bounds} on the values.  For instance, the actual formulation, as given in \cite{1804.10809}, says that 
\[\prod_{\mathcal{U}}K_i\vDash \forall^{\mathbb{N}}x\exists^{\mathbb{N}}y\forall^{\mathbb{N}}z\ \psi(x,y,z)\]
is equivalent to saying that, for every $x$ and every $\fn{Z}:\mathbb{N}\rightarrow\mathbb{N}$, there is a $Y$ so that
\[\{i\mid \exists y\leq Y\ K_i\vDash\phi(x,y,Z(y))\}\in\mathcal{U}.\]
Since $Y$ is a finite natural number, this is equivalent to the version in the previous paragraph.  However the correct inductive definition of the interpretation requires working with the bound $Y$ rather than the exact value $y$, and for more complicated sentences, working with the bounded version is unavoidable.

\section{Hilbert's Basis Theorem, Noetherianity, and the Nullstellensatz}\label{sec:local noetherian}

For some results we will need an effective version of Hilbert's Basis Theorem---that is, of the Noetherianity of $K[X_{[n]}]$.  Such theorems are given without bounds in several places in the literature, such as Hertz \cite{Hertz} and Perdry and Schuster \cite{perdry2011noetherian}. Moreno Soc\'{i}as proved that bounds on the length of ascending chains of polynomial ideals are non-primitive recursive in the number of indeterminates (Cor. 7.5, \cite{socias1992length}). As a warm-up for the differential case, we use our methods to obtain an effective basis theorem and Nullstellensatz. 

To give bounds on Hilbert's Basis Theorem, we use a function given by Figueira et al \cite{MR2858898} to bound witnesses to Dickson's Lemma. (Le\'{o}n S\'{a}nchez and Ovchinnikov give related bounds in \cite{MR3448164}.)  For the remainder of the discussion we fix an arbitrary monotonically increasing function $\bf{D}:\mathbb{N}\rightarrow\mathbb{N}$.
\begin{notation}
Consider a nonempty finite set $X$ with elements from $\mathbb{N}^{n_1}, \dots,\mathbb{N}^{n_r}$. Let $\tau_{X}$ (or simply $\tau$ when $X$ is understood) be the multiset containing one copy of $n_i$ for every element of $X$ belonging to $\mathbb{N}^{n_i}$.

Given any multiset $\tau$ containing a natural number $k>0$, we denote by $\tau_{\langle k,i,{\bf{D} } \rangle}$ the multiset obtained by removing one copy of $k$ from $\tau$ and introducing $k\cdot({\bf{D}}(i)-1)$ new copies of $k-1$. This operation introduces 0 into the multiset if $k=1$. If $\tau$ contains 0, define $\tau_{\langle 0,i,{\bf{D} } \rangle}$ to be the result of removing one copy of $0$ from $\tau$.
\end{notation}

\begin{example}
Suppose $X= \{(1,2,3), (4,5,6), (1,2)\}$. Then $\tau$ is the multiset $\{3,3,2\}$ and $\tau_{\langle 3, i,{\bf{D}} \rangle}$ is the multiset $\{3, 2,\dots ,2\}$ containing $3\cdot({\bf{D}}(i) -1) +1$ copies of $2$. The multiset $\tau_{\langle 2, i,{\bf{D} }\rangle}$ is $\{3,3, 1,\dots, 1\}$ and contains $2\cdot({\bf{D}}(i)-1)$ copies of $1$.
\end{example}

We can compare multisets lexicographically:
\begin{prop}
The collection of finite multisets on $\mathbb{N}$ is well ordered by the relation $\leq_{multi}$ defined as follows:

$\sigma\leq_{multi} \tau$ if and only if $\sigma=\tau$ or $\tau$ contains strictly more copies of $k$ than does $\sigma$, where $k$ is the greatest value such that $\tau$ and $\sigma$ contain different numbers of copies of $k$.
\end{prop}

\begin{example}\ \\
\vspace{-.25cm}
\begin{itemize}

\item $\{1,1,1,1,1\} \leq_{multi} \{2\}$
\item $\{3, 1, 0\} \leq_{multi} \{3,2\}$
\end{itemize}

\end{example}

Note that $\tau_{\langle k,i, {\bf{D}} \rangle}<_{multi} \tau$, whence the following recursive definition makes sense:

\begin{notation}\label{not:local_noetherian}
We define $\mathfrak{m}_{\tau,{\bf{D}}}(i)$ by:

\begin{itemize}
\item $\mathfrak{m}_{\emptyset, {\bf{D}}}(i) =0$.
\item $\mathfrak{m}_{\tau,{\bf{D}}}(i) = 1+ \mathfrak{m}_{\tau_{\langle \text{min }\tau,i,{\bf{D}} \rangle},{\bf{D}}}(i+1)$, where $\tau\neq \emptyset$ and $\text{min }\tau$ is the least element of the multiset $\tau$.
\end{itemize}
\noindent For future convenience, denote the expression $\mathfrak{m}_{\{n\},{\bf{D}}+1}(0)+1$ by  $\mathfrak{m}^*({\bf{D}},n)$. 
\end{notation}

\begin{example}
Let ${\bf{D}}(i)=i+2$. 
\begin{align*}
\mathfrak{m}_{\{2\},{\bf{D}}}(0)&= 1+ \mathfrak{m}_{\{1,1\},{\bf{D}}}(1)\\
& =2 + \mathfrak{m}_{\{1,0,0\},{\bf{D}}}(2)\\
&= 3 +  \mathfrak{m}_{\{1,0\},{\bf{D}}}(3)\\
&= 4 +  \mathfrak{m}_{\{1\},{\bf{D}}}(4)\\
&= 5+  \mathfrak{m}_{\{0,0,0,0,0\},{\bf{D}}}(5)\\
&= 6+  \mathfrak{m}_{\{0,0,0,0\},{\bf{D}}}(6)\\
& \vdots\\
&= 10.
\end{align*}

\end{example}

When $\vec a=(a_1,\ldots,a_n)\in\mathbb{N}^n$, we write $|\vec a|$ to represent $\text{max}_{i\leq n} \{a_i\}$ (the infinity norm).  We write $(a_1,\ldots,a_n)\preceq(b_1,\ldots,b_n)$ if for each $i\leq n$, $a_i\leq b_i$. The bound we need concerns sequences $\vec{a_1},\vec{a_2},\dots$ such that for each $i$, $|\vec a_i|\leq{\bf{D}}(i)$.

\begin{theorem}[See \cite{MR2858898}, Lemma V.I] \label{Dickson bound}
Let $\vec{a_1},\vec{a_2},\dots,$ be a sequence in $\mathbb{N}^n$ such that for each $i$, $|\vec a_i|\leq{\bf{D}}(i)$. There exist $i<j\leq \mathfrak{m}^*({\bf D},n)$ such that $\vec a_i\preceq\vec a_j$. 
\end{theorem}

\begin{remark} The existence of bounds follows from Dickson's Lemma, which implies that there are no infinite \emph{bad sequences} such that $\vec a_i \not\preceq \vec a_j$ for all $i<j$ (equivalently, $(\mathbb{N}^n,\preceq)$ is a well-quasiordering) \cite{MR2858898}. 
\end{remark}

We now give an effective version of Hilbert's Basis Theorem.
\begin{theorem}\label{local noetherianity}
Suppose $(\Lambda_1)\subseteq (\Lambda_2)\subseteq \cdots \subseteq K[X_{[n]}]$ with $\Lambda_i\subseteq K[X_{[n]}]_{\leq{\bf{D}}(i)}$. Then there is a $j\leq  \mathfrak{m}^*({\bf D},n)$ such that $(\Lambda_{j+1})\subseteq (\Lambda_j)$.
\end{theorem}

\begin{proof}
We associate the monomial $X_1^{a_1}\cdots X_n^{a_n}$ with the tuple $\vec a$.  We place a linear ordering on $\mathbb{N}^n$, and so also on monomials, by saying $\vec a<\vec b$ if either $\sum_{1\leq k\leq n}a_k<\sum_{1\leq k\leq n}b_k$ or both $\sum_{1\leq k\leq n}a_k=\sum_{1\leq k\leq n}b_k$ and, taking $l$ least so $a_l\neq b_l$, $a_l<b_l$.  (The linear ordering $<$ should not be confused with the partial ordering $\preceq$.)
 
We define a sequence of elements of $K[X_{[n]}]$ as follows.  Suppose we have defined $f_i$ for $i<j$ so that $f_i\in K[X_{[n]}]_{\leq {\bf D}(i)}$. 
We \emph{reduce} each element of $\Lambda_j$ by $f_1,\ldots,f_{j-1}$. That is, if $f\in \Lambda_j$ and $f$ contains monomials divided by the leading monomial (i.e., greatest with respect to $<$) of $f_1$, divide $f$ by $f_1$; the remainder $r_1$ is a reduction of $f$ with respect to $f_1$. 

The total degree $\text{deg}(r_1)= \text{deg}(f-\alpha\cdot f_1)$ for some $\alpha$ such that $\text{deg}(\alpha\cdot f_1)\leq\text{deg}(f)$, so $r_1\in K[X_{[n]}]_{\leq\textbf{D}(j)}$.  Reduce $r_1$ with respect to $f_2$, the resulting remainder with respect to $f_3$, and so on. Since $\Lambda_j\subseteq K[X_{[n]}]_{\leq\textbf{D}(j)}$, it follows that the reductions are contained in $K[X_{[n]}]_{\leq {\bf D}(j)}$.  If all elements reduce to $0$, we are done.  Otherwise, we take $f_j$ to be the reduction with the greatest leading monomial.

Let $\vec a_1,\ldots,\vec a_j,\ldots$ be the leading monomials of $f_1,\ldots,f_j,\ldots$.  If $i<j$, then $\vec a_i\not\preceq\vec a_j$ because $f_j$ is reduced with respect to $f_i$ and hence $\vec a_i$ does not divide $\vec a_j$.  Note that if the sum of the entries of $\vec a_j$ is bounded by ${\bf{D}}(j)$, then $|\vec a|\leq {\bf{D}}(j)$ so by  \ref{Dickson bound} this process must terminate at some $j\leq  \mathfrak{m}^*({\bf D},n)$.
\end{proof}

The final result from \cite{vdDS} we need is Corollary 2.7(ii).  We include the following proof, which is the unwinding of the proof in \cite{vdDS}.
\begin{theorem}[Based on \cite{vdDS}, Cor 2.7(ii)]\label{thm:vdDS_2.7}
For any $n,d$ there is $m={\mathfrak{m}^*(i\mapsto \mathfrak{p}_n^i(d),n)}$ so that if $\Lambda\subseteq K[X_{[n]}]_{\leq d}$ and $f^k\in(\Lambda)$ (for any $k$) then $f=\sum_{i}c_ir_i$ where each $r_i^{2^m}\in(\Lambda)$.
\end{theorem}
\begin{proof}
We will produce a tree of finitely generated ideals as follows.  When $\sigma$ is a node of this tree, we write $\Gamma_\sigma$ for the finite set of generators.  We will inductively maintain that:
\begin{itemize}
\item $\Gamma_\sigma\subseteq K[X_{[n]}]_{\leq \mathfrak{p}_n^{|\sigma|}(d)}$, and
\item if $\sigma\sqsubseteq\tau$ then $\Gamma_\sigma\subseteq\Gamma_\tau$.
\end{itemize}
We begin by setting $\Gamma_{\langle\rangle}=\Lambda$.  Given $\Gamma_\sigma$, we check whether $f\in(\Gamma_\sigma)$; if so, $\sigma$ is a leaf.  If not, since $f^k\in(\Gamma_\sigma)$, $(\Gamma_\sigma)$ is not prime, so we find $gh\in(\Gamma_\sigma)$ with $g,h\not\in(\Gamma_\sigma)$ and $\deg(g),\deg(h)\leq \mathfrak{p}_n(\mathfrak{p}_n^{|\sigma|}(d))$.  We define $\Gamma_{\sigma^\frown\langle 0\rangle}=\Gamma_\sigma\cup\{g\}$ and $\Gamma_{\sigma^\frown\langle 1\rangle}=\Gamma_\sigma\cup\{h\}$.  Inductively we see that $\Gamma_\sigma\subseteq K[X_{[n]}]_{\leq \mathfrak{p}_n^{|\sigma|}(d)}$.

The previous theorem ensures that each branch has length $\leq \mathfrak{m}^*(i\mapsto \mathfrak{p}_n^i(d),n)$, so the tree has at most $2^{\mathfrak{m}^*(i\mapsto \mathfrak{p}_n^i(d),n)}$ leaves.  Take $m={\mathfrak{m}^*(i\mapsto \mathfrak{p}_n^i(d),n)} $.  Note that the ideal corresponding to each leaf contains $f$, so for each $\sigma$ we have $f=\sum_i c_{i,\sigma}\gamma_{i,\sigma}$.  Fix some leaf $\sigma_0$, and consider the system of equations of the form $\sum_i \gamma_{i,\sigma_0}y_{i,\sigma_0}-\sum_j \gamma_{j,\sigma}y_{j,\sigma}=0$.  This is a system of at most $2^m$ equations 
whose coefficients have degree at most $\mathfrak{p}_n^m(d)$.  The $\{c_{i,\sigma}\}$ give a solution, and so by Lemma \ref{thm:flat}, there are solutions $c'_{i,j,\sigma}$ such that $c_{i,\sigma}=\sum_j d_jc'_{i,j,\sigma}$ and the $c'_{i,j,\sigma}$ have degree at most $\mathfrak{d}_n(\mathfrak{p}_n^m(d)2^m)$.

Let $f_{j}=\sum_i \gamma_{i,\sigma_0}c'_{i,j,\sigma_0}$.  Note that, since for each $j, j'\text{ and }\sigma$, $\sum_i c'_{i,j,\sigma_0}\gamma_{i,\sigma_0}-\sum_{i'}c'_{i',j',\sigma}\gamma_{i',\sigma}=0$, also $f_j=\sum_{i'} c'_{i',j',\sigma}\gamma_{i',\sigma}$, so $f_j\in(\Gamma_\sigma)$ for each leaf $\sigma$.  We now show inductively that if $|\sigma|=i$ then $f_j^{2^{m-i}}\in(\Gamma_\sigma)$.  For leaves this is immediate.  If $f_j^{2^{m-i}}\in(\Gamma_{\sigma^\frown\langle 0\rangle})\cap(\Gamma_{\sigma^\frown\langle 1\rangle})$, recall that there are $g,h$ so $\Gamma_{\sigma^\frown\langle 0\rangle}=\Gamma_\sigma\cup\{g\}$ and $\Gamma_{\sigma^\frown\langle 1\rangle}=\Gamma_\sigma\cup\{h\}$, so $f_j^{2^{m-i}}=\sum_i \gamma_{i,\sigma}u_i+gu=\sum_i \gamma_{i,\sigma}v_i+hv$, so
\[f_j^{2\cdot 2^{m-i}}=(\sum_i \gamma_{i,\sigma}u_i+gu)(\sum_i \gamma_{i,\sigma}v_i+hv)=\sum_i \gamma_{i,\sigma}u'_i+ghuv,\]
so $f_j^{2^{m-(i-1)}}\in(\Gamma_\sigma)$.

In particular, $f_j^{2^m}\in(\Gamma)$.  Since $f=\sum_i c_{i,\sigma_0}\gamma_{i,\sigma_0}=\sum_i \sum_j d_j c'_{i,j,\sigma_0}\gamma_{i,\sigma_0}=\sum_j d_j f_j$, we have shown the claim.
\end{proof}

In fact, these bounds are embarrassingly poor compared to those given by a different method.
\begin{theorem}\label{thm:rad_limit}
Suppose $\Lambda\subseteq K[X_{[n]}]_{\leq d}$ and $f^k\in(\Lambda)$ with $\deg(f)\leq d$.  Then $f^{\mathfrak{d}_{n+1}(d+1)}\in(\Lambda)$.
\end{theorem}
\begin{proof}
  We use the Rabinowitsch trick: since $f^k\in(\Lambda)$, by the Nullstellensatz we have $1=\sum_ig_i\lambda_i+g(1-Yf)$ for some $\lambda_i\in \Lambda$ and $g_i,g\in K[X_{[n]},Y]$. By Lemma \ref{thm:faithful}, we may assume that the $g_i$ have degree $\leq \mathfrak{d}_{n+1}(d+1)$.  Therefore, substituting $1/f$ for $Y$ and multiplying both sides by $f^{\mathfrak{d}_{n+1}(d+1)}$ to clear denominators, we get $f^{\mathfrak{d}_{n+1}(d+1)}=\sum_i g'_i\lambda_i$.
\end{proof}

\noindent Note, however, that the proof-mined result is a little more uniform: unlike in \ref{thm:rad_limit}, there is no restriction on the degree of $f$ in \ref{thm:vdDS_2.7}. 

\section{Bounds in Differential Polynomial Rings}\label{sec:bounds differential}

\subsection{Rankings and Faithful Flatness}

We now turn to the differential case. Henceforth we fix a field $K$ of characteristic 0 equipped with a set $\Delta=\{\delta_1,\ldots,\delta_m\}$ of commuting partial derivations (i.e., additive homomorphisms satisfying the usual product rule).  We write $\Theta=\{\delta_1^{k_1}\cdots\delta^{k_m}_m\mid k_1,\ldots,k_m\geq 0\}$ for the set of $\Delta$-operators; for convenience we will also refer to these as derivations.  By the term ``derivative'' we mean an expression like $\theta X_i$; i.e., a differential indeterminate $X_i$ to which a derivation $\theta\in \Theta$ has been applied. The \emph{order} of $\theta X_i$ (or $\theta$, if we only care about the derivation) is $\sum_{j=1}^m k_j$.

\begin{definition}\label{not:ranking}
A  \emph{ranking} $<$ on a set of derivatives is a well-ordering such that for all derivatives $u,v$ and nontrivial $\theta\in \Theta$:
\begin{enumerate}
\item If $u< v$, then $\theta u < \theta v$ and
\item $u<\theta u$.
\end{enumerate}
\end{definition}

Throughout our discussion, we assume a fixed ranking of order-type $\omega$, so we can associate a derivative $\theta X_i$ to a natural number $o(\theta X_i)$. In order to give concrete bounds, we further assume an \emph{orderly ranking} on derivatives: $\delta_1^{k_1}\cdots\delta^{k_m}_mX_r <\delta_1^{l_1}\cdots\delta^{l_m}_mX_s $ if and only if $\sum_i k_i < \sum_i l_i$ or $\sum_i k_i = \sum_i l_i$ and $k_1=l_1, k_2=l_2,\dots, k_i<l_i$ for some $1\leq i\leq m$ or all these quantities are the same and $r<s$. This allows us to equate the differential polynomial ring $K\{X_{[n]}\}$ in $n$ differential indeterminates with the algebraic polynomial ring $K[Z_1,\ldots]$ in countably many indeterminates, where $Z_{o(\theta X_i)}$ is associated with $\theta X_i$. We write $o(f)$ for $o(\theta X_i)$, where $\theta X_i$ is the highest-ranking derivative in $f$.

There are two natural gradings we might consider: either treating the degree of polynomials and the number of algebraic indeterminates separately, or combining these into a single grading.

\begin{definition}\label{def:diff degree}
  We write $K[Z_{[b]}]=K[Z_1,\ldots,Z_b]$.

  $K\{X_{[n]}\}_{\leq b}$ is $K[Z_{[b]}]_{\leq b}$.

  $K\{X_{[n]}\}_{\leq b,d}$ is $K[Z_{[b]}]_{\leq d}$.
\end{definition}

\begin{remark}\label{rmk:order vs ranking}
Order and ranking are related but not equal. The ranking can be significantly greater because the number of derivatives of a given order grows with the order. For example, using our chosen orderly ranking, if $n=1,m=2$, then $\delta_2X_1$ is the second-least derivative in the ranking but $\delta_1\delta_2X_1$ is the fifth-least even though the order only increased by 1. There are ${N+ m-1 \choose m-1}\cdot n$ derivatives of order $N$ in $K\{X_{[n]}\}, \Delta=\{\delta_1,\dots,\delta_m\}$. It follows that the ranking of a derivative of order $N$ is at most ${N+ m-1 \choose m-1}\cdot n\cdot (N+1)$.
\end{remark}

\begin{definition}
  The \emph{leader} of a differential polynomial $f\in K\{X_{[n]}\}\setminus K$ is the greatest derivative (in the ranking) appearing in $f$.  The \emph{initial} $I_f$ of $f$ is the coefficient of the highest-degree term in the leader of $f$, considering $f$ as a univariate polynomial in the leader. The \emph{separant} $S_f$ of $f$ is the initial of any proper derivative $\theta f$ of $f$ (equivalently, the formal partial derivative in the usual calculus sense with respect to the leader.) The \emph{rank} of $f$ is the ordered pair $(\mu_f, deg(\mu_f))$ consisting of the leader $\mu_f$ of $f$ and the highest degree in which it appears in $f$; ranks are compared lexicographically.
\end{definition}

See \cite{MR1921694} for further discussion and examples of rankings, leaders, and related notions. 

We first show a version of internal flatness for $K\{X_{[n]}\}$.
\begin{lemma}[Based on \cite{HTKM}, 4.1, flatness]
Whenever $f_1,\ldots,f_k\in K\{X_{[n]}\}_{\leq b}$ and $\sum_i g_if_i=0$, there exist $h_{ij}\in K\{X_{[n]}\}_{\leq b,\mathfrak{d}_b(b)}$ and $c_j\in K\{X_{[n]}\}$ so that $\sum_i h_{ij}f_i=0$ for each $j$ and $\sum_j c_jh_{ij}=g_i$ for each $i$. 
\end{lemma}
\begin{proof}
  Consider an equation $\sum_{i\leq l}f_iY_i=0$ with the $f_i$ in $K\{X_{[n]}\}_{\leq b}$ and suppose $\sum_i f_ig_i=0$.  We may write the solutions $g_i$ as polynomials in those variables $Z_k$ with $k>b$, with coefficients in $K[Z_{[b]}]$: $g_i=\sum_j g_{ij}W_j$.  Since $W_j$ is transcendental over $K[Z_{[b]}]$, $\sum_i f_ig_i=0$ implies that, for each $j$, $\sum_i f_ig_{ij}=0$.  Internal flatness of $K[Z_{[b]}]$ says that the $g_{ij}$ must be linear combinations of solutions in $K[Z_{[b]}]_{\leq \mathfrak{d}_b(b)}$.  The space of such solutions is finite dimensional, say $\vec u_1,\ldots,\vec u_r$ where each $\vec u_l=\langle u_{l0},\ldots,u_{lk}\rangle$ with $\sum_i f_iu_{li}=0$.  So for each $i$ and  $j$, the $g_{ij}$ must be a linear combination of such solutions, $g_{ij}=\sum_l c_{jl}u_{li}$.  Then $g_i=\sum_j \sum_l c_{jl}u_{li}W_j=\sum_l (\sum_j c_{jl}W_j)u_{li}$.  So, setting $c'_l=\sum_j c_{jl}W_j$, we have $g_i=\sum_l c'_lu_{li}$.
\end{proof}

\begin{lemma}[Based on \cite{HTKM}, 4.1, faithful flatness]
For any $n$ and any $f_i,h\in K\{X_{[n]}\}_{\leq b}$, if $\sum_{i\leq k}f_ig_i=h$ then there are $g'_i\in K\{X_{[n]}\}_{\leq b,\mathfrak{d}_b(b)}$ such that $\sum_{i\leq k}f_ig'_i=h$.
\end{lemma}
\begin{proof}
  Suppose $\sum_i f_ig_i=h$ where $f_i,h\in K\{X_{[n]}\}_{\leq b}= K[Z_{[b]}]_{\leq b}$.  Then we may write each $g_i$ as a sum of monomials, and have $g_i=g^+_i+g^-_i$ where $g^+_i$ consists only of those monomials in $Z_{[b]}$ and $g^-_i$ contains all monomials with at least one term outside of $Z_{[b]}$.  Then $\sum_i f_ig^-_i=0$, so $\sum_i f_ig^+_i=h$.  By Lemma \ref{thm:faithful}, there are $g'_i\in K[Z_{[b]}]_{\leq \mathfrak{d}_b(b)}= K\{X_{[n]}\}_{\leq b,\mathfrak{d}_b(b)}$ so that $\sum_i f_ig'_i=h$.
\end{proof}

\subsection{Stratified Ideals and Autoreduced Sets}

\begin{notation}
  If $\Lambda\subseteq K\{X_{[n]}\}\setminus K$ is finite, we write $H_\Lambda$ for $\prod_{\lambda\in\Lambda}I_\lambda S_\lambda$.
\end{notation}

\begin{notation}
Given $\Lambda\subseteq K\{X_{[n]}\}$, we denote by $(\Lambda), [\Lambda]$, and $\{\Lambda\}$, respectively, the ideal, differential ideal (i.e., closed under derivation), and perfect ideal (radical ideal generated by $[\Lambda]$; it is automatically differential in our case) generated in $K\{X_{[n]}\}$ by $\Lambda$.  
\end{notation}

\begin{notation}
We frequently work with the following saturation ideals:
\[\sat{\Lambda}=\{g\mid \exists n\ H_\Lambda^ng\in(\Lambda)\}\]
and
\[\dsat{\Lambda}=\{g\mid \exists n\ H_\Lambda^ng\in[\Lambda]\}.\]
\end{notation}

Because the ideals $\sat{\Lambda}$ and $\dsat{\Lambda}$ need not be finitely generated, we need a more nuanced way to work with them if we want effective bounds.

\begin{definition}
  A \emph{stratified ideal} $\langle \Lambda_k\rangle_k$ in $K\{X_{[n]}\}$ is an increasing sequence $(\Lambda_1)\subseteq(\Lambda_2)\subseteq\cdots\subseteq K\{X_{[n]}\}$ so that, for each $k$, $\Lambda_k\subseteq K\{X_{[n]}\}_{\leq k}$.
\end{definition}

We identify a stratified ideal $\langle \Lambda_k\rangle_k$ with the ideal $(\bigcup_k\Lambda_k)$.  Note that there is no assumption that $K\{X_{[n]}\}_{\leq k}\cap(\bigcup_k\Lambda_k)=(\Lambda_k)$---new elements of $K\{X_{[n]}\}_{\leq k}$ might appear in $\Lambda_{k'}$ with $k'$ much larger than $k$.\footnote{Formally, we are representing membership in these ideals as an existential property.}

We pick canonical stratifications associated with the ideals $\sat{\Lambda}$ and $\dsat{\Lambda}$.

\begin{notation}\label{note:stratified}
  We write
\[\Lambda_{(k)}^H=\{g\in K\{X_{[n]}\}_{\leq k}\mid H_\Lambda^kg\in(\Lambda)\},\]
\[\Lambda_{[k]}=\{\theta \lambda\mid \lambda\in\Lambda \text{ and } o(\theta\lambda)\leq o(\lambda)+k\},\]
and
\[\Lambda_{[k]}^H=\{g\in K\{X_{[n]}\}_{\leq k}\mid H_\Lambda^kg\in(\Lambda_{[k]})\}.\]
\end{notation}

Before discussing primality of stratified ideals, we introduce the important topic of \emph{autoreduced sets}.  

\begin{definition}
A differential polynomial $f\in K\{X_{[n]}\}$  is \emph{partially reduced} with respect to $g\in K\{X_{[n]}\}\setminus K$ if $f$ has no proper derivative of the leader $\mu_g$ of $g$. $f$ is \emph{reduced} with respect to $g$ if $f$ is partially reduced with respect to $g$ and additionally the degree of $\mu_g$ in $f$ is strictly less than the degree of $\mu_g$ in $g$. $f$ is reduced with respect to a subset $S\subseteq K\{X_{[n]}\}\setminus K$ if $f$ is reduced with respect to every element of $S$, and $S$ is \emph{autoreduced} if every element of $S$ is reduced with respect to every other element of $S$. 
\end{definition} 
 
If a finite set $\Lambda\subseteq K\{X_{[n]}\}_{\leq b}=K[Z_{[b]}]_{\leq b}$ is autoreduced, one may check that $2b\choose b$ is an upper bound on the number of distinct monomials in $Z_{1},\dots, Z_b$ and hence on the cardinality of $\Lambda$.  In particular, this implies that $H_\Lambda\in K\{X_{[n]}\}_{\leq b,2b{2b\choose b}}$.

Given $f\in K\{X_{[x]}\}$, we can find a closely related remainder $\tilde{f}$ that is reduced with respect to $\Lambda$; the process of obtaining $\tilde{f}$ is called \emph{pseudodivision}. The exact remainder obtained from reduction depends on the sequence of elements from $\Lambda$, but we can still find effective bounds on the complexity regardless of the choices made during pseudodivision.

\begin{example}
Choose a ranking in which $z>\delta_1x> \delta_2y$. Let $g_1=\delta_2y(\delta_1x)^2+x\delta_1x$, $g_2=\delta_2y\delta_1^2x +x$, and $f=z+ x\delta_1^2x + T_f$ (where the trailing terms $T_f$ have lower rank). We illustrate a single step of pseudodivision of $f$ with respect to $g_1$ and $g_2$, respectively:

\begin{itemize}
\item $f$ contains a proper derivative of the leader of $g_1$, so we differentiate $g_1$ to obtain 
\[\delta_1g_1=(2\delta_2y\delta_1x +x)\delta_1^2x+ (\delta_1\delta_2y +1)(\delta_1x)^2. \]

\noindent Multiply $f$ by $S_{g_1}$ and subtract a suitable multiple of $\delta_1g_1$. The remainder $r_1$ after one step is 
\[S_{g_1}f-x\cdot\delta_1g_1= S_{g_1}(z+T_f) -x\cdot(\delta_1\delta_2y +1)(\delta_1x)^2.\]

\item $f$ contains the leader of $g_2$ but no proper derivatives thereof, so to pseudodivide $f$ by $g_2$ we simply multiply $f$ by $I_{g_2}$ and divide to obtain

\[I_{g_2}f-x\cdot g_2= I_{g_2}(z+T_f)-x^2.\]

\end{itemize}

\noindent Note that the actual leader $z$ of $f$ never came into play; we only eliminated terms that prevented $f$ from being reduced with respect to $g_1$ or $g_2$. Also, if we were reducing $f$ with respect to the set $\{g_1,g_2\}$, we would instead continue reducing $r_1$ with respect to $g_1$ until the remainder $r$ was reduced with respect to $g_1$. We would then reduce $r$ with respect to $g_2$. See \cite{MR1921694} for further explanation of reduction algorithms.
\end{example}

\begin{notation}\label{note:g}
$\mathfrak{g}(b,d)=d(1+b)^{d}$.
\end{notation}

\begin{lemma}[Based on \cite{MR0568864}, I(9), Proposition 1] \label{pseudodivision bounds}
  Let $\Lambda\subseteq K\{X_{[n]}\}_{\leq b}$ be autoreduced and let $f\in K\{X_{[n]}\}_{\leq d}$. Then there exist $\tilde{f}\in K\{X_{[n]}\}_{\leq d, \mathfrak{g}(b,d)}$ reduced with respect to $\Lambda$ and $k_\lambda,l_\lambda\leq \mathfrak{g}(b,d)$ for each $\lambda\in \Lambda$ such that 
\[\left(\prod_{\lambda\in\Lambda}I_\lambda^{k_\lambda}S_\lambda^{l_\lambda}\right) f-\tilde{f}\in (\Lambda_{[d]}).\]
\end{lemma} 
\begin{proof}
Pseudodivide repeatedly to reduce the highest-ranking derivative in the current remainder that is not reduced with respect to $\Lambda$. This rank decreases at each step (either by reducing order or degree), so the process terminates with a remainder $\tilde{f}$ that is reduced with respect to $\Lambda$. Repeated multiplication by initials and separants throughout this process yields the form $\left(\prod_{\lambda\in\Lambda}I_\lambda^{k_\lambda}S_\lambda^{l_\lambda}\right) f-\tilde{f}\in (\Lambda_{[d]})$ for some $k_\lambda,l_\lambda$.

Recursively define ${\bf B}(i)$ as follows: 

\begin{itemize}
\item ${\bf B}(0)=d$,
\item ${\bf B}(i+1)={\bf B}(i)(1+b)$.
\end{itemize}

We claim that ${\bf B}(i)$ is an upper bound on the degree of the remainder after reducing $i$-many times the highest-ranking derivative not reduced with respect to~$\Lambda$.

By hypothesis, $d={\bf B}(0)$ bounds the degree of $f$, so we start out correctly. For the inductive step, suppose that the remainder after reducing $i$-many derivatives with respect to $\Lambda$ belongs to $K\{X\}_{\leq d,{\bf B}(i)}$. The number of division steps required for the next-highest-ranking derivative is at most the current degree in that derivative, hence it is bounded by ${\bf B}(i)$. Each multiplication by an initial or separant increases the degree by at most $b$. The degree after reducing the $i+1$-st derivative is consequently bounded by $\text{(current bound) } + b\cdot\!\text{(maximal number of divisions)}= {\bf B}(i)+b\cdot {\bf B}(i)={\bf B}(i)(1+b)={\bf B}(i+1).$ Pseudodivision cannot increase the order, so the new remainder belongs to $K\{X\}_{\leq d,{\bf B}(i+1)}$. Thus the final bound on the degree of $\tilde{f}$ is ${\bf B}(d)=d(1+b)^d=\mathfrak{g}(b,d)$. 

To find a bound on $k_\lambda,l_\lambda$ we add up the total number of division steps  required to reduce each derivative.  Again we see that $\mathfrak{g}(b,d)$ is an upper bound: $d+ d(1+b)+ \dots + d(1+b)^{d-1}= d\left(\frac{(1+b)^d-1}{b}\right)\leq\mathfrak{g}(b,d)$.
  
\end{proof}

\subsection{Local Primality}

The right notion of primality for stratified ideals is a ``local primality'' notion which tells us not only that when $fg\in(\bigcup_k\Lambda_k)$ that either $f\in(\bigcup_k\Lambda_k)$ or $g\in(\bigcup_k\Lambda_k)$, but incorporates a bound ${\bf F}$, so that $fg\in(\Lambda_k)$ implies that either $f\in (\Lambda_{{\bf F}(k)})$ or $g\in(\Lambda_{{\bf F}(k)})$.

\begin{definition}
  Let $\langle \Lambda_k\rangle_k$ be a stratified ideal.  We say $\langle\Lambda_k\rangle_k$ is \emph{${\bf F}$-prime up to $b$} if for each $k\leq b$, whenever $fg\in(\Lambda_k)$, either $f\in(\Lambda_{{\bf F}(k)})$ or $g\in(\Lambda_{{\bf F}(k)})$.
\end{definition}

We also need the bounded version of primality, analogous to the notion of ``prime up to $d$''.

\begin{definition}
  Let $\langle \Lambda_k\rangle_k$ be a stratified ideal.  We say $\langle\Lambda_k\rangle_k$ is \emph{boundedly ${\bf F}$-prime up to $b$} if for each $k\leq b$, whenever $fg\in(\Lambda_k)\cap K\{X_{[n]}\}_{\leq k}$, either $f\in(\Lambda_{{\bf F}(k)})$ or $g\in(\Lambda_{{\bf F}(k)})$.  
\end{definition}

We always assume that the function ${\bf F}$ is monotone---that is, $a\leq b$ implies ${\bf F}(a)\leq {\bf F}(b)$.

\begin{notation}\label{note:uf}
Given $\fn{F}$, let $\fn{F}_d(b)=\fn{F}(\mathfrak{p}_d(b))$.  Set $\mathfrak{u}_{{\bf F}}(x)=\fn{F}_x^{\mathfrak{m}^*(i\mapsto {\bf F}^{i}_x(x),x)}(x)$.
\end{notation}

\begin{lemma}[Based on \cite{HTKM}, 4.1a]\label{bounded prime implies prime}
If $\Lambda\subseteq K\{X_{[n]}\}_{\leq b}$ is such that $\langle\Lambda^H_{(k)}\rangle$ is boundedly ${\bf F}$-prime up to $\mathfrak{p}_d(\mathfrak{u}_{{\bf F}}(d))$, then $\langle\Lambda^H_{(k)}\rangle$ is $\mathfrak{u}_{{\bf F}}$-prime up to $d$.
\end{lemma}
\begin{proof}
To show $\mathfrak{u}_{{\bf F}}$-primality up to $d$, we must show the statement for each $d'\leq d$, but without loss of generality (because ${\bf F}$ is monotone) it suffices to show that whenever $fg\in(\Lambda^H_{(d)})$, either $f\in(\Lambda^H_{(\mathfrak{u}_{{\bf F}}(d))})$ or $g\in(\Lambda^H_{(\mathfrak{u}_{{\bf F}}(d))})$.  

Consider $\Lambda^*_{(k)}=\Lambda^H_{(k)}\cap K\{X_{[n]}\}_{\leq d,k}$.  Observe that $\langle\Lambda^*_{(k)}\rangle$ is also boundedly ${\bf F}_d$-prime up to $\mathfrak{p}_d(\mathfrak{u}_{{\bf F}}(d))$: if $fg\in(\Lambda^*_{(k)})\cap K\{X_{[n]}\}_{\leq d,k}\subseteq(\Lambda_{(k)}^H)$ for $k\leq \mathfrak{p}_d(\mathfrak{u}_{{\bf F}}(d))$ then, a fortiori, $f\in(\Lambda_{({\bf F}(k))}^H)\cap K\{X_{[n]}\}_{\leq d,k}$, so $f\in(\Lambda^*_{({\bf F}(k))})$.

Consider the sequence of ideals $(\Lambda^*_{(d)})\subseteq(\Lambda^*_{(\fn{F}_d(d))})\subseteq(\Lambda^*_{(\fn{F}^2_d(d))})\subseteq\cdots$. By Theorem \ref{local noetherianity}, there is some $i\leq\mathfrak{m}^*(i\mapsto\fn{F}^i_d(d),d)$ so that $(\Lambda^*_{(\fn{F}^{i+1}_d(d))})\subseteq (\Lambda^*_{(\fn{F}^{i}_d(d))})$.  Let $k=\fn{F}^i(d)\leq\mathfrak{u}_{\fn{F}}(d)$, so $d\leq k\leq \mathfrak{u}_{{\bf F}}(d)$ and $\Lambda^*_{({\bf F}_d(k))}\subseteq (\Lambda^*_{(k)})$.  In particular, since $\mathfrak{p}_d(k)\leq \mathfrak{p}_d(\mathfrak{u}_{{\bf F}}(d))$, if $fg\in(\Lambda^*_{(k)})\cap K\{X_{[n]}\}_{\leq d,\mathfrak{p}_d(k)}$ then either $f$ or $g$ belongs to $(\Lambda^*_{(\fn{F}(\mathfrak{p}_d(k)))})=(\Lambda^*_{(\fn{F}_d(k))})\subseteq(\Lambda^*_{(k)})$, so $(\Lambda^*_{(k)})$ is prime up to $\mathfrak{p}_d(k)$.  By Theorem \ref{thm:vdds_2.5}, $(\Lambda^*_{(k)})$ is prime.

Now suppose $fg\in(\Lambda^H_{(d)})\subseteq(\Lambda^H_{(k)})$.  Then $fg\in K[Z_{[m]}]$ for some $m$.  Let $M_0,\ldots,M_j,\ldots$ enumerate the monomials over the variables $Z_{[d+1,m]}$.  We may write $fg=\sum_{i,j}u_{i,j}\gamma_i M_j$ with $\gamma_i\in \Lambda^H_{(d)}\subseteq\Lambda^*_{(k)}$, $f=\sum_j b_jM_j$ and $g=\sum_j c_jM_j$.  Then for each $j$,
\[\sum_{j_0+j_1=j}b_{j_0}c_{j_1}=\sum_iu_{i,j}\gamma_i.\]

We resolve this monomial by monomial.  We show by induction on $J$ that there are $k_0,k_1$ with $k_0+k_1=J$ so that for each $j<k_0$, $b_j\in(\Lambda^*_{(k)})$ and for each $j<k_1$, $c_j\in(\Lambda^*_{(k)})$.  When $J=0$, this is immediate.  Suppose the claim holds for $J$; we have
\[b_{k_0}c_{k_1}=\sum_i u_{i,J}\gamma_i-\sum_{j_0<k_0}b_{j_0}c_{J-j_0}-\sum_{j_1<k_1}b_{J-j_1}c_{j_1},\]
and therefore $b_{k_0}c_{k_1}\in(\Lambda^*_{(k)})$.  Since this is a prime ideal, we have either $b_{k_0}\in(\Lambda^*_{(k)})$, in which case we increment $k_0$, or similarly with $c_{k_1}$.

When $J$ is large enough, we see that we must have either $b_j\in(\Lambda^*_{(k)})$ for all $j$ or $c_j\in(\Lambda^*_{(k)})$ for all $j$, so we have either $f\in(\Lambda^*_{(k)})\subseteq(\Lambda^H_{(k)})\subseteq(\Lambda^H_{(\mathfrak{u}_{\bf F}(d))})$  or $g\in(\Lambda^*_{(k)})\subseteq(\Lambda^H_{(\mathfrak{u}_{\bf F}(d))})$.
\end{proof}

\begin{notation}\label{note:uf_plus}
  $\mathfrak{u}^+_{\bf F}(b)=\max\{\mathfrak{u}_{\bf F}(b),\mathfrak{d}_{b+1}(2b{2b\choose b}+1)\}$.
\end{notation}

\begin{lemma}[Based on \cite{HTKM}, 4.2b]\label{thm:HTKM_4.2b}
If $\Lambda\subseteq K\{X_{[n]}\}_{\leq b}$, $|\Lambda|\leq {2b \choose b}$, and $\langle \Lambda_{(k)}^H\rangle$ is boundedly ${\bf F}$-prime up to $\mathfrak{p}_b(\mathfrak{u}_{\bf F}(b))$, then $\sat{\Lambda}\subseteq(\Lambda^H_{(\mathfrak{u}^+_{\bf F}(b))})$.
\end{lemma}
\begin{proof}
Suppose $H^N_\Lambda g\in(\Lambda)$ for some $N$. Then certainly $H^N_\Lambda g\in(\Lambda_{(b)}^H)$ since $\Lambda\subseteq\Lambda_{(b)}^H$.  Since $\langle\Lambda^H_{(k)}\rangle$ is boundedly ${\bf F}$-prime up to $\mathfrak{p}_b(\mathfrak{u}_{\bf F}(b))$, by Lemma \ref{bounded prime implies prime} it is $\mathfrak{u}_{\bf F}$-prime up to $b$ and either $H^N_\Lambda\in(\Lambda_{(\mathfrak{u}_{\bf F}(b))}^H)$ or $g\in(\Lambda_{(\mathfrak{u}_{\bf F}(b))}^H)$; in the latter case we are done.

Suppose $H^N_\Lambda\in(\Lambda_{(\mathfrak{u}_{\bf F}(b))}^H)$, so also $H^{N+\mathfrak{u}_{\bf F}(b)}_\Lambda\in(\Lambda)$.  Then by Theorem \ref{thm:rad_limit}, already $H^{\mathfrak{d}_{b+1}(2b{2b\choose b}+1)}_\Lambda\in(\Lambda)$.  But this implies that $1\in(\Lambda_{(\mathfrak{d}_{b+1}(2b{2b\choose b}+1))}^H)$, so also $g\in(\Lambda_{(\mathfrak{d}_{b+1}(2b{2b\choose b}+1))}^H)$.
\end{proof}

Later we will need a slight refinement of these notions, where we restrict to the subring of $K\{X_{[n]}\}$ containing only those indeterminates that are not proper derivatives of the leaders of $\Lambda$ (that is, those elements partially reduced with respect to $\Lambda$).

\begin{definition}
  $K\{X_{[n]}\upharpoonright\Lambda\}$ is the ring $K[Z_S]$ where $S$ is the (possibly infinite) set of indices of indeterminates which are not proper derivatives of any $\mu_\lambda$ with $\lambda\in\Lambda$.  We write $K\{X_{[n]}\upharpoonright\Lambda\}_{\leq b}$ for $K\{X_{[n]}\upharpoonright\Lambda\}\cap K\{X_{[n]}\}_{\leq b}$.

  We say $\langle\Lambda_k\rangle_k$ is \emph{pr$(\Lambda)$-${\bf F}$-prime up to $b$} if for each $k\leq b$, whenever $fg\in(\Lambda_k)\cap K\{X_{[n]}\upharpoonright\Lambda\}$, either $f\in(\Lambda_{{\bf F}(k)})$ or $g\in(\Lambda_{{\bf F}(k)})$.

We say $\langle\Lambda_k\rangle_k$ is \emph{boundedly pr$(\Lambda)$-${\bf F}$-prime up to $b$} if for each $k\leq b$, whenever $fg\in(\Lambda_k)\cap K\{X_{[n]}\upharpoonright\Lambda\}_{\leq k}$, either $f\in(\Lambda_{{\bf F}(k)})$ or $g\in(\Lambda_{{\bf F}(k)})$.  
\end{definition}

Inspection of the proofs of the previous two lemmas gives:
\begin{lemma}\label{thm:pr_bounded_implies_pr}
If $\Lambda\subseteq K\{X_{[n]}\}_{\leq b}$ is such that $\langle\Lambda^H_{(k)}\rangle$ is boundedly pr$(\Lambda)$-${\bf F}$-prime up to $\mathfrak{p}_d(\mathfrak{u}_{{\bf F}}(d))$, then $\langle\Lambda^H_{(k)}\rangle$ is pr$(\Lambda)$-$\mathfrak{u}_{{\bf F}}$-prime up to $d$.
\end{lemma}

\begin{lemma}\label{pr boundedly prime implies bound on saturation}
If $\Lambda\subseteq K\{X_{[n]}\}_{\leq b}$, $|\Lambda|\leq {2b \choose b}$, and $\langle \Lambda_{(k)}^H\rangle$ is boundedly pr$(\Lambda)$-${\bf F}$-prime up to $\mathfrak{p}_b(\mathfrak{u}_{\bf F}(b))$, then $\sat{\Lambda}\cap K\{X_{[n]}\upharpoonright\Lambda\}\subseteq(\Lambda^H_{(\mathfrak{u}^+_{\bf F}(b))})$.
\end{lemma}

Then we get:
\begin{lemma}\label{thm:HTKM_4_2_c_real}
If $\Lambda\subseteq K\{X_{[n]}\}_{\leq b}$ is autoreduced and $\langle \Lambda_{(k)}^H\rangle$ is boundedly pr$(\Lambda)$-${\bf F}$-prime up to $\mathfrak{p}_b(\mathfrak{u}_{\bf F}(b))$, then $\sat{\Lambda}\subseteq(\Lambda^H_{(\mathfrak{u}^+_{\bf F}(b))})$.
\end{lemma}
\begin{proof}
Given $g\in\sat{\Lambda}$, we have $g\in K\{X_{[n]}\upharpoonright\Lambda\}[Z_T]$ for some finite set of additional indeterminates $Z_T$, each a derivative of some leader in $\Lambda$.  Then we may write $g=\sum_i g_iM_i$ with the $g_i\in K\{X_{[n]}\upharpoonright\Lambda\}$. The $Z_T$ do not appear in $\Lambda$ (by autoreducedness of $\Lambda$), so by Lemma \ref{pr boundedly prime implies bound on saturation} and induction on the number of $Z_T$ appearing in $g$ we must have each $g_i\in\sat{\Lambda}\cap K\{X_{[n]}\upharpoonright\Lambda\}\subseteq(\Lambda^H_{(\mathfrak{u}^+_{\bf F}(b))})$. This proves $g\in(\Lambda^H_{(\mathfrak{u}^+_{\bf F}(b))})$.
\end{proof}

\begin{notation}\label{note:fF}\ \\
\vspace{-.25cm}
\begin{itemize}
\item $N = \mathfrak{d}_{\mathfrak{u}^+_{\bf F}(b)}(\mathfrak{u}^+_{\bf F}(b)) + \mathfrak{u}^+_{\bf F}(b)$.
\item $\mathfrak{f}({\bf F},b)=\mathfrak{d}_{\mathfrak{u}^+_{\bf F}(b)}({N+b\choose b}\cdot\mathfrak{u}^+_{\bf F}(b))+N$.
\end{itemize}
\end{notation}


\begin{lemma}[Based on \cite{HTKM}, 4.2c]\label{thm:HTKM_4.2c}
Suppose $\Lambda\subseteq K\{X_{[n]}\}_{\leq b}$ is autoreduced and $\langle\Lambda^H_{(k)}\rangle$ is boundedly pr$(\Lambda)$-${\bf F}$-prime up to $\mathfrak{p}_b(\mathfrak{u}_{\bf F}(b))$.  Suppose there is a $g\in\sat{\Lambda}$ which is non-zero and reduced with respect to $\Lambda$.  Then there is an $f\in(\Lambda^H_{(\mathfrak{u}^+_{\bf F}(b))})\cap K\{X_{[n]}\}_{\leq \mathfrak{u}^+_{\bf F}(b),\mathfrak{f}({\bf F},b)}$ which is non-zero and reduced with respect to $\Lambda$.

\end{lemma}
\begin{proof}
Suppose $g\in\sat{\Lambda}$ is non-zero and reduced with respect to $\Lambda$.  Then by Lemma \ref{thm:HTKM_4_2_c_real}, $g\in(\Lambda^H_{(\mathfrak{u}^+_{\bf F}(b))})$.  Let $T_{[t]}$ list the variables appearing as leaders in $\Lambda$ (so $t\leq b$) and let $Y_{[y]}$ list all variables appearing in $g$ which are not derivatives of the variables $T_{[t]}$.  Since $g$ is reduced relative to $\Lambda$, $g\in K[T_{[t]},Y_{[y]}]$. Without loss of generality we may assume that  $t+y\leq \mathfrak{u}^+_{\bf F}(b)$ and that $\Lambda^H_{(\mathfrak{u}^+_{\bf F}(b))}\subseteq K[T_{[t]},Y_{[y]}]$. We then have $g=\sum_i c_i\gamma_i$ for some $\gamma_i\in \Lambda^H_{(\mathfrak{u}^+_{\bf F}(b))}$ and $c_i\in K[T_{[t]},Y_{[y]}]$.

Consider the field $L=K(Y_{[y]})$.  Then $g\in L[T_{[t]}]_{\leq b}$ (since its degree in the leading variables of $\Lambda$ is bounded by the degrees of $\Lambda$).  By Lemma \ref{thm:faithful}, we also have $g=\sum_i c'_i\gamma_i$ where $c'_i\in L[T_{[t]}]_{\leq \mathfrak{d}_{\mathfrak{u}^+_{\bf F}(b)}(\mathfrak{u}^+_{\bf F}(b))}$.  Clearing denominators, we have $f=hg=\sum_i h_i\gamma_{i}$ where the $h_i\in K[T_{[t]},Y_{[y]}]$ have the same $T_{[t]}$-degree as the $c'_i$. Note that $f$ is still non-zero and reduced with respect to $\Lambda$.

We expand this into a system of equations with one equation for each monomial from $T_{[t]}$. There are at most ${N+t\choose t}$-many equations because there are $t$-many variables and the $T_{[t]}$-degree of the system is bounded by $N$, the sum of the $T_{[t]}$-degrees of $h_i$ and $\gamma_i$.  By Theorem \ref{thm:flat}, the system has solutions $h'_i$ with $Y_{[y]}$-degree at most $\mathfrak{d}_{\mathfrak{u}^+_{\bf F}(b)}({N+t\choose t}\cdot\mathfrak{u}^+_{\bf F}(b))$ such that   $f'=\sum_i h'_i\gamma_i$ is non-zero in the same monomials $f$ is. Further, since $f'$ is zero in all monomials $f$ is, $f'$ is still reduced with respect to $\Lambda$. By using the fact that $t\leq b$ and adding the total degrees of $h'_i$ and $\gamma_i$, we obtain the final bound on $f'$.
\end{proof}


\subsection{Chains of Autoreduced Sets and Coherent Sets}
The derivatives in $K\{X_{[n]}\}$ are well-quasiordered: that is, given any infinite sequence $u_1,u_2,\ldots,u_n,\ldots$, there must be some $i<j$ and some derivation $\theta$ so that $\theta u_i=u_j$.  





\begin{definition}
  A \emph{bad leader sequence} is a sequence of derivatives $\langle u_1,\ldots,u_m\rangle$ so that $u_1<u_2<\cdots<u_m$ and if $i<j$ then there is no $\theta$ with $\theta u_i=u_j$.
\end{definition}
It is an easy consequence of Dickson's Lemma that there are no infinite bad leader sequences; in particular, we can speak of the ``tree of bad leader sequences'', and carry out proofs by induction on this tree.

There is a standard ordering associated with autoreduced sets:

\begin{definition}
Let $\Lambda\subseteq K\{X_{[n]}\}\setminus K$ be autoreduced.  List $\Lambda=\{f_1,\ldots,f_r\}$ in order of ascending rank.  We write $\Gamma(\Lambda)$ for the sequence $\langle(\mu_{f_1},b_1),\ldots,(\mu_{f_r},b_r)\rangle$ where $\mu_{f_i}$ is the leader of $f_i$ and $b_i$ is the degree of $\mu_{f_i}$ in $f_i$---that is, $(\mu_{f_i},b_i)$ is the rank of $f_i$.

Given such a sequence $\gamma=\langle(\mu_1,b_1),\ldots,(\mu_r,b_r)\rangle$, we write $\gamma_\mu=\langle \mu_1,\ldots,\mu_r\rangle$ and $\gamma_b=\langle b_1,\ldots,b_r\rangle$.

Given two sequences $\gamma_1=\langle(\mu_{f_1},b_1),\ldots,(\mu_{f_r},b_r)\rangle$ and $\gamma_2=\langle(\mu_{g_1},b'_1),\ldots,(\mu_{g_s},b'_s)\rangle$, we say $\gamma_1$ has \emph{lower rank} than $\gamma_2$ if either:
\begin{enumerate}
\item there is an $i\leq \min\{r,s\}$ so that for all $j<i$, $(\mu_{f_j},b_j)=(\mu_{g_j},b'_j)$, but $(\mu_{f_i},b_i)<(\mu_{g_i},b'_i)$, or
\item $s>r$ and for all $j\leq r$, $(\mu_{f_j},b_j)=(\mu_{g_j},b'_j)$.
\end{enumerate}

By abuse of notation, we often say $\Lambda_1$ has lower rank than $\Lambda_2$ when $\Gamma(\Lambda_1)$ has lower rank than $\Gamma(\Lambda_2)$.
\end{definition}



Rank forms a well-order on the sequences $\Gamma(\Lambda)$; we now obtain explicit bounds on the length of decreasing sequences in this ordering.

The idea is that we suppose we have a long sequence $\Gamma(\Lambda_0),\ldots,\Gamma(\Lambda_d)$ all beginning with the same initial sequence $\gamma$.  In the worst case, where this sequence is as long as possible, $\Gamma(\Lambda_0)=\gamma$ and for all $i>0$, $\Gamma(\Lambda_i)$ must properly extend $\gamma$---that is, each $\Gamma(\Lambda_i)$ must begin with some $\gamma^\frown\langle (u_i,k_i)\rangle$.  We break the interval $[1,d]$ into subsequences based on $u_i$, and then further subsequences based on $k_i$.

We first write down hypothetical worst case bounds: we inductively bound each subinterval, and then add these all up to bound the whole interval.
\begin{notation}\label{note:hgamma}
We define a bound $\mathfrak{h}_{n,m}(\fn{D},\gamma)$ by recursion on $\gamma_\mu$, taking $\fn{D}_{i_0}(i)=\fn{D}(i_0+i)$.  We usually drop $n,m$ when they are clear from context.  We define:
\begin{itemize}
\item when $\gamma_\mu$ is maximal, $\mathfrak{h}(\fn{D},\gamma)=1$,
\item when $\gamma_\mu$ is not maximal, we define two helper sequences, $w_u$ for $u\in[-1,\fn{D}(1)]$ and $v_{u,k}$ for $u\in[-1,\fn{D}(1)]$, $k\in[0,\fn{D}(w_u)]$:
  \begin{itemize}
  \item $w_{\fn{D}(1)}=1$,
  \item if $\gamma_\mu^\frown\langle u\rangle$ is not a bad leader sequence then $w_{u-1}=w_u$,
  \item if $\gamma_\mu^\frown\langle u\rangle$ is a bad leader sequence then
    \begin{itemize}
    \item $v_{u,\fn{D}(w_u)}=w_u$,
    \item $v_{u,k-1}=v_{u,k}+\mathfrak{h}(\fn{D}_{v_{u,k}},\gamma^\frown\langle (u,k)\rangle)$,
    \item $w_{u-1}=v_{u,0}$,
    \end{itemize}
  \end{itemize}
and set $\mathfrak{h}(\fn{D},\gamma)=w_{-1}$.
\end{itemize}
We set $\mathfrak{h}(\fn{D})=\mathfrak{h}(\fn{D},\langle\rangle)$.
\end{notation}
Roughly speaking, the gap $v_{u,k-1}-v_{u,k}$ is a bound on how long the subinterval of $i$ so that $\Gamma(\Lambda_i)$ begins with $\gamma^\frown\langle (u,k)\rangle$ could be.

\begin{lemma}\label{thm:dickson_iteration_bound}
  Let a monotonic function $\fn{D}:\mathbb{N}\rightarrow\mathbb{N}$ be given and let $\gamma=\langle(\mu_1,b_1),\ldots,(\mu_r,b_r)\rangle$.  Suppose that for each $i$, $\Lambda_i$ is an autoreduced set in $K\{X_{[n]}\}_{\leq \fn{D}(i)}$ so that $\Gamma(\Lambda_i)$ begins with $\gamma$.  Then there is an $i<\mathfrak{h}(\fn{D},\gamma)$ such that $\Gamma(\Lambda_{i+1})$ does \emph{not} have lower rank than $\Gamma(\Lambda_{i})$.
\end{lemma}
\begin{proof}
We proceed by induction on $\gamma_\mu$.  When $\gamma_\mu$ is maximal, $1$ suffices: if $\Gamma(\Lambda_0)$ and $\Gamma(\Lambda_1)$ both begin with $\gamma$ then both must be equal to $\gamma$ (because $\gamma_\mu$ is a maximal bad leader sequence), so they have the same rank.

Suppose $\gamma_\mu$ is not maximal, so that other than $\Lambda_0$, we may assume each $\Gamma(\Lambda_i)$ is a \emph{proper} extension of $\gamma$.  That is, each $\Gamma(\Lambda_i)$ begins $\gamma^\frown\langle (u,b)\rangle$ for some $u\leq \fn{D}(1)$.  For each $(u,b)$, take $\hat w_u$ to be least so that $\Gamma(\Lambda_{\hat w_u})$ begins with $\gamma^\frown\langle (u',b')\rangle$ for some $u'\leq u$ and take $\hat v_{u,b}$ to be least so that $\Gamma(\Lambda_{\hat v_{u,b}})$ begins with $\gamma^\frown\langle (u,b')\rangle$ for some $b'\leq b$.

To prove that our bounds work, we compare the actual gaps $\hat v_{u,k-1}-\hat v_{u,k}$ to our bounds $v_{u,k-1}-v_{u,k}$.  The idea is there must be some first interval in which the actual interval is at least as long as our bound, and the inductive hypothesis will guarantee that we find our witness in this interval.

Taking $\hat v_{u,-1}=\hat w_{u-1}$, we look for the smallest $(u,b)$ so that $\hat v_{u,b-1}-\hat v_{u,b}> v_{u,b-1}-v_{u,b}$.  Note that we must have $\hat v_{u,b}\leq v_{u,b}$ (otherwise this would have happened for a smaller $(u,b)$).  Then for every $i\in[\hat v_{u,b},\hat v_{u,b-1}-1]$, $\Gamma(\Lambda_i)$ begins with $\gamma^\frown\langle (u,b)\rangle$, so we may apply the inductive hypothesis with $\fn{D}_{\hat v_{u,b}}$ and obtain the desired witness.
\end{proof}

In particular, when $\gamma$ is the empty sequence,
\begin{cor}\label{autoreduced_chain_bound}
  Let a monotonic function $\fn{D}:\mathbb{N}\rightarrow\mathbb{N}$ be given.  Suppose that for each $i$, $\Lambda_i$ is an autoreduced set in $K\{X_{[n]}\}_{\leq \fn{D}(i)}$.  Then there is an $i< \mathfrak{h}(\fn{D})$ such that $\Gamma(\Lambda_{i+1})$ does \emph{not} have lower rank than $\Gamma(\Lambda_{i})$.
\end{cor}

This bound is relevant for several important operations. Our first application takes as input a finite set and outputs an autoreduced set that is closely related to the original.
\begin{notation}\label{note:isat}
$\fn{D}^{\mathrm{sat}}_{b,n,m}$ is the function defined inductively by:
\begin{itemize}
\item $\fn{D}^{\mathrm{sat}}_{b,n,m} (0)=b$,
\item $\fn{D}^{\mathrm{sat}}_{b,n,m} (i+1)=\mathfrak{g}(\fn{D}^{\mathrm{sat}}_{b,n,m}(i),b)$.
\end{itemize}
We set $\mathfrak{i}^{\mathrm{sat}}_{n,m}(b)=\fn{D}^{\mathrm{sat}}_{b,n,m}(\mathfrak{h}_{n,m}(\fn{D}^{\mathrm{sat}}_{b,n,m}))$.
\end{notation}
Here $n$ is the number of indeterminates and $m$ is the number of derivations; when $n,m$ are fixed, we simply write $\mathfrak{i}^{\mathrm{sat}}(b)$.

\begin{prop} \label{autoreduced set procedure}
Let  $\Lambda \subseteq K\{X_{[n]}\}_{\leq b}$ be finite.  Then there exists an autoreduced set $\Lambda'\subseteq [\Lambda]\cap K\{X_{[n]}\}_{\leq b,\mathfrak{i}^{\mathrm{sat}}(b)}$ such that $\Lambda\subseteq \dsat{\Lambda'}$.   Further, if $\Lambda^*$ is an autoreduced subset of $\Lambda$ then $\Lambda'$ has rank less than or equal to that of $\Lambda^*$.
\end{prop}

\begin{proof}
Consider the following procedure: Let $\Lambda_0$ be an autoreduced subset of $\Lambda$ having minimal rank. We are done if $\Lambda\subseteq \dsat{\Lambda_0}$. Otherwise, there exists $f\in \Lambda$ whose remainder $\tilde{f}$ with respect to $\Lambda_0$ is non-zero. Let $\Lambda_1$ be an autoreduced subset of $\Lambda_0 \cup \{\tilde{f}\}\subseteq \Lambda$ having minimal rank. Repeat the process to recursively define a sequence $\Lambda_0,\Lambda_1,\Lambda_2,\dots$ of autoreduced subsets of $\Lambda$. By Corollary \ref{autoreduced_chain_bound} it suffices to verify that $\Lambda_{i+1}< \Lambda_i$ as autoreduced sets and that $\Lambda_i\subseteq K\{X_{[n]}\}_{\leq b, \fn{D}^{\mathrm{sat}}_b(i)}$.

Let $\tilde{f}$ be the remainder used in defining $\Lambda_{i+1}$ from $\Lambda_i$. Note that elements of $\Lambda_i$ of lower rank than $\tilde{f}$ are reduced with respect to $\tilde{f}$. Since $\tilde{f}$ is reduced with respect to $\Lambda_i$, the set $\{\text{elements of }\Lambda_i\text{ having rank lower than that of } \tilde{f}\} \cup\{\tilde{f}\}$ is an autoreduced subset of $\Lambda_i\cup \{\tilde{f}\}$ having strictly lower rank than that of $\Lambda_i$.

Lastly, note that $\Lambda_0\subseteq K\{X_{[n]}\}_{\leq b, b}=K\{X_{[n]}\}_{\leq b, \fn{D}^{\mathrm{sat}}_b(0)}$. Assume $\Lambda_i\subseteq K\{X_{[n]}\}_{\leq b, \fn{D}^{\mathrm{sat}}_b(i)}$.  By Lemma \ref{pseudodivision bounds}, pseudodividing an element of $\Lambda\subseteq K\{X_{[n]}\}_{\leq b}$ with respect to $\Lambda_i$ gives a remainder belonging to  $\mathfrak{g}(\fn{D}^{\mathrm{sat}}_b(i),b)=\fn{D}^{\mathrm{sat}}_b(i+1)$.
\end{proof}

\begin{notation}
Let $v$ be a derivative and let $\Lambda\subseteq K\{X_{[n]}\}$ be a finite set. We write $\Lambda_{[<v]}$ for the set $\{\theta\lambda\mid \lambda\in\Lambda, \theta\in \Theta, \text{ and } \mu_{\theta\lambda}<v\}$.
\end{notation}

\begin{notation}
Consider $f,g\in K\{X_{[n]}\}\setminus K$ and $\theta_g, \theta_f$ derivations such that $\theta_gf$ and $\theta_fg$ have the same leader $v$. We define the \emph{$\Delta$-$S$-polynomial} of $f$ and $g$ with respect to $v$, denoted $\Delta(f,g,v)$, to be $S_g\theta_gf -S_f\theta_fg$. If $v$ is the least such derivative, we simply write $\Delta(f,g)$. 
\end{notation}
The following property is key:
\begin{definition}
An autoreduced set $\Lambda\subseteq K\{X_{[n]}\}\setminus K$ is \emph{coherent} if for all $f,g\in \Lambda$ with derivatives that share a common leader $v$ we have $\Delta(f,g,v)\in (\Lambda_{[<v]}):H_\Lambda^\infty$.
\end{definition}

To have coherence it suffices for the least such leader to satisfy this condition.  In practice, it is convenient to work with a slightly stronger notion.  This costs us nothing, since the standard construction of a coherent set actually gives the stronger property.
\begin{definition}
An autoreduced set $\Lambda\subseteq K\{X_{[n]}\}\setminus K$ is \emph{reduction-coherent} if for all $f,g\in \Lambda$ with derivatives that share a common leader $v$, the $\Delta$-$S$-polynomial $\Delta(f,g,v)$ reduces to $0$ with respect to $\Lambda$.
\end{definition}
It is easy to check that any reduction-coherent set is coherent.  See \cite{MR1921694,MR0568864} for further details and generalizations.

Using essentially the same strategy as in \ref{autoreduced set procedure}, we give effective bounds on coherent sets. See \cite{MR1676955}[5.5.12] for the algorithm (but without the bound). 

\begin{notation}\label{note:icohere}
$\fn{D}^{\mathrm{cohere}}_b$ is the function defined inductively by:
\begin{itemize}
\item $\fn{D}^{\mathrm{cohere}}_{b,n,m}(0)=b$,
\item $\fn{D}^{\mathrm{cohere}}_{b,n,m} (i+1)=\mathfrak{g}(\fn{D}^{\mathrm{cohere}}_{b,n,m} (i),{2\fn{D}^{\mathrm{cohere}}_{b,n,m} (i)+ m-1 \choose m-1}\cdot n\cdot(\fn{D}^{\mathrm{cohere}}_{b,n,m} (i) +1))$.
\end{itemize}
We set $\mathfrak{i}^{\mathrm{cohere}}_{n,m}(b)=\fn{D}^{\mathrm{cohere}}_{b,n,m} (\mathfrak{h}_{n,m}(\fn{D}^{\mathrm{cohere}}_{b,n,m}))$.
\end{notation}
Again we usually omit $n,m$.

\begin{prop}\label{coherent bound}
Let  $\Lambda_0\subseteq K\{X_{[n]}\}_{\leq b}$ be an autoreduced set.  Then there exists a reduction-coherent set $\Lambda\subseteq [\Lambda_0] 
\cap K\{X_{[n]}\}_{\leq \mathfrak{i}^{\mathrm{cohere}}(b)}$. 
\end{prop}
\begin{proof}
Given an autoreduced set $\Lambda_0$, form the $\Delta$-$S$-polynomials corresponding to pairs of elements of $\Lambda_0$. Pseudodivide each $\Delta$-$S$-polynomial by $\Lambda_0$ and let $R_0$ be the set of remainders. If $R_0=\{0\}$, then $\Lambda_0$ is already reduction-coherent. If not, select an autoreduced subset $\Lambda_1$ of $\Lambda_0\cup R_0$ having minimal rank and repeat the process. 

The output is automatically reduction-coherent if the procedure terminates. Termination follows from Corollary \ref{autoreduced_chain_bound} because $\Lambda_{i+1}< \Lambda_i$ as autoreduced sets using the same argument as before: given $p\in R_i$, the set $\{\text{elements of }\Lambda_i\text{ having rank lower than that of } p\} \cup\{p\}$ is an autoreduced subset of $\Lambda_i\cup R_i$ having strictly lower rank than that of $\Lambda_i$.

We claim that the intermediate sets are bounded by $\fn{D}^{\mathrm{cohere}}_b$. Suppose $\Lambda_i\subseteq K\{X_{[n]}\}_{\leq \fn{D}^{\mathrm{cohere}}_b(i)}$. Then the degree and order of a $\Delta$-$S$-polynomial of two elements of $\Lambda_i$ are both at most $2\fn{D}^{\mathrm{cohere}}_b(i)$. The ranking grows more, though, because there are ${2\fn{D}^{\mathrm{cohere}}_b(i)+ m-1 \choose m-1}\cdot n$ derivatives of order $2\fn{D}^{\mathrm{cohere}}_b(i)$. Hence doubling the order of a derivative in $K\{X_{[n]}\}_{\leq \fn{D}^{\mathrm{cohere}}_b(i)}$ cannot place the result beyond  $K\{X_{[n]}\}_{\leq{2\fn{D}^{\mathrm{cohere}}_b(i)+ m-1 \choose m-1}\cdot n\cdot(\fn{D}^{\mathrm{cohere}}_b(i) +1)}$.

By Lemma \ref{pseudodivision bounds}, pseudodividing an element bounded by ${2\fn{D}^{\mathrm{cohere}}_b(i)+ m-1 \choose m-1}\cdot n\cdot(\fn{D}^{\mathrm{cohere}}_b(i) +1)$ with respect to $\Lambda_i$ gives a remainder bounded by $\mathfrak{g}(\fn{D}^{\mathrm{cohere}}_b(i),\allowbreak {2\fn{D}^{\mathrm{cohere}}_b(i)+ m-1 \choose m-1}\cdot n\cdot(\fn{D}^{\mathrm{cohere}}_b(i) +1))=\fn{D}^{\mathrm{cohere}}_b(i+1)$. Hence the algorithm terminates by step $\mathfrak{h}(\fn{D}^{\mathrm{cohere}}_b)$ and the output is bounded by $\fn{D}^{\mathrm{cohere}}_b(\mathfrak{h}(\fn{D}^{\mathrm{cohere}}_b))=\mathfrak{i}^{\mathrm{cohere}}(b)$.
\end{proof}

An easy extension of this argument gives:
\begin{prop}\label{ref:containment}
Let  $\Lambda_0\subseteq K\{X_{[n]}\}_{\leq b}$ be an autoreduced set.  Then there exists a reduction-coherent set $\Lambda\subseteq [\Lambda_0]
\cap K\{X_{[n]}\}_{\leq \mathfrak{i}^{\mathrm{cohere}}(b)}$ such that $\Lambda_0\subseteq\dsat{\Lambda}$.
\end{prop}
\begin{proof}
  As in the previous proof, we form a sequence of autoreduced sets $\Lambda_0,\Lambda_1,\ldots$ of decreasing rank.  If $\Lambda_i$ is not reduction-coherent then we choose $\Lambda_{i+1}$ as in the previous lemma.

If $\Lambda_i$ is reduction-coherent, we ask whether there is any $f\in\Lambda_0$ so that the remainder $\tilde f$ with respect to $\Lambda_i$ is non-zero.  If so, we take $\Lambda_{i+1}$ to be a minimal rank autoreduced subset of $\Lambda_{i}\cup\{\tilde f\}$.

This must still terminate within $\mathfrak{h}(\fn{D}^{\mathrm{cohere}}_b)$ steps with some $\Lambda$ which is reduction-coherent and whenever $f\in\Lambda_0$, the remainder $\tilde f$ with respect to $\Lambda$ is $0$.  Therefore by the definition of the remainder, $H^k_{\Lambda}f\in
[\Lambda]$ for some $k$ and $f\in\dsat{\Lambda}$.
\end{proof}


\begin{lemma}[Based on \cite{HTKM}, 4.4(1)]\label{thm:HTKM_4.4_1}
  Suppose $\Lambda\subseteq K\{X_{[n]}\}_{\leq b}$ is coherent and autoreduced and $\langle \Lambda_{(k)}^H\rangle$ is ${\bf F}$-prime up to $\mathfrak{p}_b(\mathfrak{u}_{\bf F}(b))$.  
  Suppose $g\in\dsat{\Lambda}\cap K\{X_{[n]}\}_{\leq d}$.  Then $g\in(\Lambda^H_{[\mathfrak{g}(b,d)+\mathfrak{u}^+_{\bf F}(b)]})$.
\end{lemma}
\begin{proof}
Let $\tilde g$ be the remainder of $g$ with respect to $\Lambda$, so also $\tilde g\in\dsat{\Lambda}$.    By \cite{MR0568864} III(8), Lemma 5, also $\tilde g\in\sat{\Lambda}$, and then by Lemma \ref{thm:HTKM_4.2b}, $\tilde g\in(\Lambda^H_{\mathfrak{u}^+_{\bf F}(b)})$, so $H_\Lambda^{\mathfrak{u}^+_{\bf F}(b)}\tilde g\in(\Lambda)\subseteq(\Lambda_{[d]})$.  Since $H_\Lambda^{\mathfrak{g}(b,d)}g-\tilde g\in(\Lambda_{[d]})$ as well by Lemma \ref{pseudodivision bounds}, we have $H_\Lambda^{\mathfrak{g}(b,d)+\mathfrak{u}^+_{\bf F}(b)}g\in(\Lambda_{[d]})$, so $g\in(\Lambda^H_{[\mathfrak{g}(b,d)+\mathfrak{u}^+_{\bf F}(b)]})$.
\end{proof}

\begin{lemma}[Based on \cite{HTKM}, 4.4(2)]
Suppose $\Lambda\subseteq K\{X_{[n]}\}_{\leq b}\setminus K$ is autoreduced and $\dsat{\Lambda}$ contains no non-zero elements of degree $\leq b$ reduced with respect to $\Lambda$, that $P$ is any prime $\Delta$-ideal, and that there is some $g\in\dsat{\Lambda}\setminus P$.  Then there is an $h\in K\{X_{[n]}\}_{\leq b}$ so that $h\in\dsat{\Lambda}\bigtriangleup P$.


\end{lemma}
\begin{proof}
If $\Lambda\not\subseteq P$, this is immediate, so assume $\Lambda\subseteq P$.  Since $g\in\dsat{\Lambda}$, also $H^d_\Lambda g\in P$ for some $d$.  Since $P$ is prime, either $g\in P$ or $I_\lambda$ or $S_\lambda $ belongs to $P$ for some $\lambda\in \Lambda$.  Since we have ruled out $g\in P$, we have $I_\lambda\in P$ or $S_\lambda\in P$.  But both $I_\lambda$ and $S_\lambda$ have degree $\leq b$ and are reduced with respect to $\Lambda$, so they do not belong to $\dsat{\Lambda}$.

\end{proof}




\begin{notation}\label{note:zk}\ 
  \begin{itemize}
  \item $\mathfrak{z}^0(d,b)=d$,
  \item $\mathfrak{z}^{k+1}(d,b)=\mathfrak{z}^k(\mathfrak{d}_{b+d}((\mathfrak{g}(b+d-1, \text{max}\{b+d-1,2b\})+d+1){2b\choose b}2b+d) + \mathfrak{g}(b+d-1, \text{max}\{b+d-1,2b\})+d+1,b)$.
  \end{itemize}
\end{notation}

\begin{lemma}[Based on \cite{MR0568864}, III.8, Lemma 5]\label{thm:Kolchin_3.8.5}
  Suppose $\Lambda\subseteq K\{X_{[n]}\}_{\leq b}$ is autoreduced and reduction-coherent.  Then if $g\in(\Lambda^H_{[d]})\cap K\{X_{[n]}\}_{\leq d}$ is partially reduced with respect to $\Lambda$, $g\in(\Lambda^H_{(\mathfrak{z}^{b+d}(d,b)))})$.
\end{lemma}
\begin{proof}
We may write
\[H^d_\Lambda g=\sum_{j\leq r}c_j(\theta_j\lambda_j)+\sum_i d_i\lambda_i\]
where $\theta_j\lambda_j\in K\{X_{[n]}\}_{\leq b+d}$.  Let $m$ be the least upper bound on indices $l$ such that $Z_l$ is the leader of some proper derivative $\theta_j\lambda_j$ appearing in the sum.  We will show by induction on $m$ that $g\in(\Lambda^H_{(\mathfrak{z}^{m}(d,b))})$.  When $m=0$ (that is, there are no $c_j\theta_j\lambda_j$ terms), this is immediate.

So suppose the claim holds for values less than $m$ and there is some $j$ with $
\mu_{\theta_j\lambda_j}=Z_m$ and $Z_m$ is the largest leader; we write
\[H^d_\Lambda g=\sum_{j< q}c_j(\theta_j\lambda_j)+\sum_{q\leq j< r}c_j(\theta_j\lambda_j)+\sum_i d_i\lambda_i\]
where $\mu_{\theta_j\lambda_j}=Z_{k_j}$ so $k_j<m$ if $j< q$ and $k_j=m$ if $q\leq j<r$.

We multiply both sides by $S_{\lambda_q}$, so 
\begin{align*}
S_{\lambda_q}H^d_\Lambda g&=\sum_{j< q}c'_j(\theta_j\lambda_j)+\sum_i d'_i\lambda_i
+\sum_{q\leq j<r}c_j(S_{\lambda_q}\theta_j\lambda_j-S_{\lambda_j}\theta_q\lambda_q)+\sum_{q\leq j<r}c_jS_{\lambda_j}\theta_q\lambda_q.
\end{align*}

By assumption, the remainder of $S$ with respect to $\Lambda_{[d-1]}$ is 0. Since $S\in K\{X_{[n]}\}_{\leq b+d-1, 2b}$ and $\Lambda_{[d-1]}\subseteq K\{X_{[n]}\}_{\leq b+d-1,b}$,  by \ref{pseudodivision bounds} we have $H_\Lambda^{\mathfrak{g}(b+d-1, \text{max}\{b+d-1,2b\})}S\in (\Lambda_{[d-1]})$. It follows that 
\[H^{\mathfrak{g}(b+d-1, \text{max}\{b+d-1,2b\})+d+1}_\Lambda g=\sum_{j<q'}c''_j(\theta_j\lambda_j)+\sum_i d''_i\lambda_i+c^*(\theta_q\lambda_q).\]
By internal faithful flatness, we may assume $c''_j,d''_i,c^*$ have degree at most $D=\mathfrak{d}_{b+d}((\mathfrak{g}(b+d-1, \text{max}\{b+d-1,2b\})+d+1){2b\choose b}2b+d)$.  Since $g$ is partially reduced with respect to $\Lambda$, in particular $Z_m=\mu_{\theta_q\lambda_q}$ does not appear on the left side of the equation.  We have $\theta_q\lambda_q=S_{\lambda_q}Z_m+h$ where $h$ has lower rank than $Z_m$, so we may replace $Z_m$ with $-\frac{h}{S_{\lambda_q}}$.  The degree of $Z_m$ is at most $D$, so we multiply by $H_\Lambda^{D}$ to clear denominators, giving

\[H^{D+\mathfrak{g}(b+d-1, \text{max}\{b+d-1,2b\})+d+1}_\Lambda g=\sum_{j<q'}c'''_j(\theta_j\lambda_j)+\sum_id'''_i\lambda_i.\]

The claim now follows by the inductive hypothesis since each $\mu_{\theta_j\lambda_j}=Z_l$ with $l<m$.
\end{proof}


\subsection{Characteristic Sets}

\begin{definition}
Let $I\subseteq K\{X_{[n]}\}$ be a $\Delta$-ideal. An autoreduced subset $\Sigma$ of $I$ having minimal rank is a \emph{characteristic set} of $I$.
\end{definition}

We use Rosenfeld's lemma and effective bounds on reduction-coherent sets to find effective bounds on characteristic sets.

\begin{notation}\label{note:ichar}
$\fn{D}^{\mathrm{char}}_b$ is the function defined inductively by:
\begin{itemize}
\item ${\bf F}_c^{\mathrm{char}}(k)=\mathfrak{z}^{c+\mathfrak{g}(c,k)}(\mathfrak{g}(c,k),c)$,
\item $\fn{D}^{\mathrm{char}}_{b,n,m}(0)=b$,
\item \begin{align*}\fn{D}^{\mathrm{char}}_{b,n,m}(i+1)=\max\{&\mathfrak{g}(\fn{D}^{\mathrm{char}}_{b,n,m}(i),{2\fn{D}^{\mathrm{char}}_{b,n,m} (i)+ m-1 \choose m-1}\cdot n\cdot(\fn{D}^{\mathrm{char}}_{b,n,m} (i) +1)),\\
&\mathfrak{p}_{\mathfrak{u}_{{\bf F}^{\mathrm{char}}_{\fn{D}^{\mathrm{char}}_{b,n,m}(i)}}^+(\fn{D}^{\mathrm{char}}_{b,n,m}(i))}
(\mathfrak{u}_{{\bf F}^{\mathrm{char}}_{\fn{D}^{\mathrm{char}}_{b,n,m}(i)}}(\mathfrak{u}_{{\bf F}^{\mathrm{char}}_{\fn{D}^{\mathrm{char}}_{b,n,m}(i)}}^+(\fn{D}^{\mathrm{char}}_{b,n,m}(i))
)),\\
&\mathfrak{f}({\bf F}^{\mathrm{char}}_{\fn{D}^{\mathrm{char}}_{b,n,m}(i)},\fn{D}^{\mathrm{char}}_{b,n,m}(i))\}.\end{align*}
\end{itemize}
We set $\mathfrak{i}^{\mathrm{char}}_{n,m}(b)=\fn{D}^{\mathrm{char}}_{b,n,m}(\mathfrak{h}_{n,m}(\fn{D}^{\mathrm{char}}_{b,n,m}))$.
\end{notation}
Once again, we usually omit $n,m$.

\begin{theorem}[Based on \cite{HTKM}, Lemma 5.6/Theorem 6.1]\label{thm:char set}
  Let $\Lambda\subseteq K\{X_{[n]}\}_{\leq b}$.  
Let $P$ be a proper $\Delta$-ideal containing $\Lambda$ such that whenever $fg\in P$ with $f\in K\{X_{[n]}\}_{\leq \mathfrak{i}^{\mathrm{char}}(b)}$, either $f\in P$ or $g\in P$.  Then $P$ contains a set $\Sigma\subseteq K\{X_{[n]}\}_{\leq \mathfrak{i}^{\mathrm{char}}(b)}$ which is the characteristic set of a prime ideal containing $\Lambda$ where $H_\Sigma\not\in P$.
\end{theorem}
\begin{proof}
We construct a series of autoreduced sets so that $\Lambda_{i+1}$ has lower rank than $\Lambda_i$.  Take $\Lambda_0\subseteq K\{X_{[n]}\}_{\leq b}$ to be some minimal rank autoreduced subset of $\Lambda$.  Given $\Lambda_i\subseteq K\{X_{[n]}\}_{\leq \fn{D}^{\mathrm{char}}_b(i)}\cap P$, we proceed as follows.

First, if $\Lambda_i$ is not reduction-coherent, we proceed as in Proposition \ref{coherent bound}: take $\Lambda_{i+1}$ to be a minimal rank autoreduced subset of $\Lambda_i\cup R$ where $R$ consists of the $\Delta$-$S$-polynomials of pairs from $\Lambda_i$.  As in Proposition \ref{coherent bound}, we have $\Lambda_{i+1}\subseteq K\{X_{[n]}\}_{\leq \mathfrak{g}(\fn{D}^{\mathrm{char}}_b(i),{2\fn{D}^{\mathrm{char}}_{b} (i)+ m-1 \choose m-1}\cdot n\cdot(\fn{D}^{\mathrm{char}}_{b} (i) +1))}$.

If $\Lambda_i$ is reduction-coherent but $\Lambda\not\subseteq\sat{\Lambda_i}$, pick $f\in\Lambda$ as in Proposition \ref{ref:containment} so that the remainder $\tilde f$ with respect to $\Lambda_i$ is non-zero and let $\Lambda_{i+1}$ be a minimal rank autoreduced subset of $\Lambda_i\cup\{\tilde f\}$.  By \ref{pseudodivision bounds} we have $\Lambda_{i+1}\subseteq K\{X_{[n]}\}_{\leq \fn{D}^{\mathrm{char}}_b(i),\mathfrak{g}(\fn{D}^{\mathrm{char}}_b(i), b)}\subseteq K\{X_{[n]}\}_{\leq \mathfrak{g}(\fn{D}^{\mathrm{char}}_b(i),{2\fn{D}^{\mathrm{char}}_{b} (i)+ m-1 \choose m-1}\cdot n\cdot(\fn{D}^{\mathrm{char}}_{b} (i) +1))}$.

If neither of the previous cases holds and $H_{\Lambda_i}\in P$, pick some $u\in \{I_f,S_f\mid f\in\Lambda_i\}\cap P$ and take a minimal rank autoreduced subset $\Lambda_{i+1}$ of $\Lambda_i\cup\{u\}$.  Then $\Lambda_{i+1}\subseteq K\{X_{[n]}\}_{\leq \fn{D}^{\mathrm{char}}_b(i)}$.

Let $C=\mathfrak{u}_{{\bf F}^{\mathrm{char}}_{\fn{D}^{\mathrm{char}}_{b,n,m}(i)}}^+(\fn{D}^{\mathrm{char}}_{b,n,m}(i))$.
If $\langle (\Lambda_i)^H_{(d)}\rangle_d$ is not boundedly pr$(\Lambda_i)$-${\bf F}^{\mathrm{char}}_{\fn{D}^{\mathrm{char}}_b(i)}$-prime up to $\mathfrak{p}_{C}
(\mathfrak{u}_{{\bf F}^{\mathrm{char}}_{\fn{D}^{\mathrm{char}}_{b,n,m}(i)}}(C))$ then there is some $k\leq \mathfrak{p}_{C} (\mathfrak{u}_{{\bf F}^{\mathrm{char}}_{\fn{D}^{\mathrm{char}}_{b,n,m}(i)}}(C))$ and some $fg\in ((\Lambda_i)^H_{(k)})\cap K\{X_{[n]}\upharpoonright \Lambda_i\}_{\leq k}$ so that $f,g\not\in((\Lambda_i)^H_{({\bf F}^{\mathrm{char}}_{\fn{D}^{\mathrm{char}}_b(i)}(k))})$.  Then $H^{k}_{\Lambda_i}fg\in P$.  Since $H_{\Lambda_i}\not\in P$, we may assume $f\in P$.  Let $\tilde f$ be the remainder of $f$ with respect to $\Lambda_i$, so there are $l_0,l_1\leq\mathfrak{g}(\fn{D}^{\mathrm{char}}_b(i),k)$ so that
\[I_{\Lambda_i}^{l_0}S_{\Lambda_i}^{l_1}f-\tilde f\in((\Lambda_i)_{[k]}).\]
Suppose $\tilde f=0$; then $H^{\mathfrak{g}(\fn{D}^{\mathrm{char}}_b(i),k)}_{\Lambda_i}f\in((\Lambda_i)_{[k]})$, so $f\in((\Lambda_i)^H_{[\mathfrak{g}(\fn{D}^{\mathrm{char}}_b(i),k)]})$.  But then by Lemma \ref{thm:Kolchin_3.8.5}, $f\in ((\Lambda_i)^H_{(\mathfrak{z}^{\fn{D}^{\mathrm{char}}_b(i)+\mathfrak{g}(\fn{D}^{\mathrm{char}}_b(i),k)}(\mathfrak{g}(\fn{D}^{\mathrm{char}}_b(i),k),\fn{D}^{\mathrm{char}}_b(i)))})=((\Lambda_i)^H_{({\bf F}^{\mathrm{char}}_{\fn{D}^{\mathrm{char}}_b(i)}(k))})$.  But this contradicts the assumption that $fg$ witnessed the failure of bounded $F^{\mathrm{char}}_{\fn{D}^{\mathrm{char}}_b(i)}$-primality.  So $\tilde f\neq 0$.  Since $\tilde f\in P$, let $\Lambda_{i+1}$ be a minimal rank autoreduced subset of $\Lambda_i\cup\{\tilde f\}$.

Suppose none of the cases above hold, and there is a $g\in\sat{\Lambda_i}$ which is non-zero and reduced with respect to $\Lambda_{i+1}$.  Then by Lemma \ref{thm:HTKM_4.2c} there is an $f\in(\Lambda^H_{(\mathfrak{u}^+_{{\bf F}^{\mathrm{char}}_{\fn{D}^{\mathrm{char}}_b(i)}}(\fn{D}^{\mathrm{char}}_b(i)))})\cap K\{X_{[n]}\}_{\leq \mathfrak{u}^+_{{\bf F}^{\mathrm{char}}_{\fn{D}^{\mathrm{char}}_b(i)}}(\fn{D}^{\mathrm{char}}_b(i)),\mathfrak{f}({\bf F}^{\mathrm{char}}_{\fn{D}^{\mathrm{char}}_b(i)},\fn{D}^{\mathrm{char}}_b(i))}$ which is non-zero and reduced with respect to $\Lambda_{i+1}$.  Then we may take $\Lambda_{i+1}$ to be a minimal rank autoreduced subset of $\Lambda_i\cup\{f\}$.

By Lemma \ref{autoreduced_chain_bound}, there is an $i<\mathfrak{h}(\fn{D}^{\mathrm{char}}_b)$ so that none of these cases occurs: $\Lambda_i$ is reduction-coherent, $H_{\Lambda_i}\not\in P$, $\langle (\Lambda_i)^H_{(d)}\rangle_d$ is boundedly pr$(\Lambda_i)$-${\bf F}^{\mathrm{char}}_{\fn{D}^{\mathrm{char}}_b(i)}$-prime up to $\mathfrak{p}_{C}
(\mathfrak{u}_{{\bf F}^{\mathrm{char}}_{\fn{D}^{\mathrm{char}}_{b,n,m}(i)}}(C))$, and there is no $g\in\sat{\Lambda_i}$ which is non-zero and reduced with respect to $\Lambda_i$.  It follows from Lemma \ref{thm:pr_bounded_implies_pr} that $\langle (\Lambda_i)^H_{(d)}\rangle_d$ is pr$(\Lambda_i)$-$\mathfrak{u}_{{\bf F}^{\mathrm{char}}_{\fn{D}^{\mathrm{char}}_b(i)}}$-prime up to $C$. Lemma \ref{thm:HTKM_4_2_c_real} implies $\sat{\Lambda_i}\subseteq ((\Lambda_i)^H_{(C)})$,  whence $\sat{\Lambda_i}$ is prime.  As a consequence of Rosenfeld's lemma (p. 399 \cite{rosenfeld1959specializations}), $\Lambda_i$ is the characteristic set of the prime $\Delta$-ideal $\dsat{\Lambda_i}$.  Since $\Lambda\subseteq\dsat{\Lambda_i}$, we are done. 
\end{proof}

\begin{cor}\label{thm:min prime ideal}
Suppose $\Lambda\subseteq K\{X_{[n]}\}_{\leq b}$ and let $P$ be a minimal prime $\Delta$-ideal containing $\Lambda$.  Then $P$ has a characteristic set $\Sigma\subseteq K\{X_{[n]}\}_{\leq \mathfrak{i}^{\mathrm{char}}(b)}$.
\end{cor}
\begin{proof}
  $P$ contains some $\Sigma$ which is the characteristic set of a prime ideal $\dsat{\Sigma}$ containing $\Lambda$ with $H_\Sigma\not\in P$.  Suppose $f\in \dsat{\Sigma}$, so $H^k_\Sigma f\in [\Sigma]\subseteq P$ for some $k$.  Since $P$ is prime and $H_\Sigma\not\in P$, we have $f\in P$, so $\dsat{\Sigma}\subseteq P$.  By minimality of $P$, we must have $\dsat{\Sigma}=P$.
\end{proof}

\begin{cor}[Based on \cite{HTKM}, Proposition 5.3/Theorem 5.4]\label{thm:partial prime bound}
Let  $\Lambda\subseteq K\{X_{[n]}\}_{\leq b}$ be given with $1\not\in[\Lambda]$. If either $f\in [\Lambda]$ or $g\in[\Lambda]$ for all $f,g\in K\{X_{[n]}\}$ with $fg\in[\Lambda]$ and $f\in K\{X_{[n]}\}_{\leq \mathfrak{i}^{\mathrm{char}}(b)}$, then $[\Lambda]$ is prime.
\end{cor}
\begin{proof}
We apply the theorem to obtain a characteristic set $\Sigma\subseteq  [\Lambda]\cap K\{X_{[n]}\}_{\leq \mathfrak{i}^{\mathrm{char}}(b)}$ with $H_\Sigma\not\in[\Lambda]$.  Since $\Lambda\subseteq \dsat{\Sigma}$, we also have $[\Lambda]\subseteq\dsat{\Sigma}$.  It remains to show the reverse containment.  If $g\in\dsat{\Sigma}$ then for some $N\in\mathbb{N}$ we have $H^N_\Sigma g\in[\Sigma]\subseteq[\Lambda]$.  Let $N$ be least so that $H^N_\Sigma g\in[\Lambda]$.  If $N>0$ then either $H_\Sigma\in[\Lambda]$ or $g\in[\Lambda]$ (because the factors of $H_\Sigma$ are bounded by $\mathfrak{i}^{\mathrm{char}}(b)$).  But in the former case $H_\Sigma \in[\Lambda]\subseteq\dsat{\Sigma}$.  Therefore $g\in[\Lambda]$.
\end{proof}

 \begin{remark} Proposition 5.3 of \cite{HTKM} states this for a radical ideal, but their argument similarly applies to a differential ideal with minor modifications.
 \end{remark}

\section{Ritt-Noetherianity}\label{sec:ritt noetherian}

In this section we give an effective version of Ritt-Noetherianity.  There are two approaches we might take, depending on the question of whether we treat membership in a radical differential ideal as decidable---that is, given $h\in K\{X_{[n]}\}_{\leq d}$ and $\Lambda\subseteq K\{X_{[n]}\}_{\leq b}$, whether there is some bound $k$ depending on $n,b,d$ so that $h\in\{\Lambda\}$ iff $h^k\in(\Lambda_{[k]})$.

Of course, there is such a bound, given first in \cite{MR2556127} with refinements in \cite{freitag2016effective, gustavson2016new, MR3448164, gustavson2016effective}.  But \cite{HTKM} uses Ritt-Noetherianity to prove the existence of such a bound, so one might choose to avoid existing bounds and use the functional interpretation to obtain an explicit version of the bound from \cite{HTKM}.  The resulting version of Ritt-Noetherianity, however, is rather unwieldy (as the functional interpretation of a $\Pi_4$ statement, it requires the use of higher-order functions on functions), and the bounds one gets are much worse than those in the literature.

Therefore in the work below, we use the bounds from \cite{MR2556127} on testing membership in radical differential ideals.  This makes Ritt-Noetherianity a $\Pi_3$ statement, directly analogous to the effective version of Noetherianity discussed in Section \ref{sec:local noetherian}.  This will have roughly the form
\begin{quote}
For any functions $\fn{D}, \fn{F}$ there is a bound $M$ so that whenever $\Lambda_i\subseteq K\{X_{[n]}\}_{\leq \fn{D}(i)}$ for all $i$, there is an $m\leq M$ so that $\Lambda_{\fn{F}(m)}\subseteq\{\Lambda_m\}$.
\end{quote}
Along the way, we will need to inductively prove cases of similar but weaker statements (roughly speaking, we will replace the conclusion with $\Lambda_{\fn{F}(m)}\subseteq\{\Lambda_m\cup\{u\}\}$ for various choices of $u$).  Once we prove several of these, we will need to arrange for their conjunction to hold \emph{uniformly}---that is, once we can find $m_1$ so that $\Lambda_{\fn{F}(m_1)}\subseteq\{\Lambda_{m_1}\cup\{u_1\}\}$ and $m_2$ so that $\Lambda_{\fn{F}(m_2)}\subseteq\{\Lambda_{m_2}\cup\{u_2\}\}$, we will need to find a single $m$ so that $\Lambda_{\fn{F}(m)}\subseteq\{\Lambda_m\cup\{u_1\}\}\cap\{\Lambda_m\cup\{u_2\}\}$.

\begin{lemma}[Knitting Lemma]
Let $J$ be a finite set.  Suppose that for each $j\in J$, any function $\fn{F}$, and any $d$, there is a $k\in[d,\mathfrak{G}_j(\fn{F},d)]$ so that for all $i\in[k,\fn{F}(k)]$, the statement $\phi_j(i)$ holds.  Then there is a functional $\mathfrak{G}_J$ so that for any $d$ and any $\fn{F}$, there is a $k\in [d,\mathfrak{G}_J(\fn{F},d)]$ so that, for each $j\in J$ and each $i\in[k,\fn{F}(k)]$, $\phi_j(i)$ holds.
\end{lemma}
The existence of such a lemma is guaranteed by the proof of correctness of the functional interpretation.  We began with statements $\forall x\exists y\forall z\phi_j(x,y,z)$, and it follows that $\forall x\exists Y\forall j\in J\exists y_j\leq Y\forall z\phi_j(x,y_j,z)$.  The functional interpretation promises that if the latter follows from the former then bounds on the functional interpretation of the latter must be derivable from bounds on the functional interpretation of the former.  The knitting lemma is simply the statement that the functional interpretation works in one specific case.\footnote{Specifically, the case in question is essentially $B\Sigma_2$---bounded collection for $\Sigma_2$ formula.  Since $B\Sigma_2$ is a consequence of $I\Sigma_2$, induction for $\Sigma_2$ formulas, we expect to obtain an interpretation using G\"odel's primitive recursive functional of type 1, which appears in the form of the map taking $\fn{D},d'$ to $\fn{F}^{d'}$ below.  A similar analysis of $B\Sigma_2$ due to Oliva \cite{10.1007/11780342_44}, also presented on page 213 of \cite{kohlenbach:MR2445721}, highlights the logical aspects more clearly, but is less useful for computing the bounds we ultimately want.  This lemma can also be compared to Lemma 6.2 of \cite{avigad09}, which deals with essentially the same issue.}

\begin{proof}
  By induction on $|J|$.  When $|J|=1$, this is trivial---$\mathfrak{G}_{\{j\}}$ is simply $\mathfrak{G}_j$.  So suppose $|J|>1$ and pick some $j_0\in J$.  The inductive hypothesis gives us a function $\mathfrak{G}_{J\setminus\{j_0\}}$.


  For any $d'$, define $\fn{F}^{d'}(d)=\fn{F}(\max\{d,d'\})$.  Let $\fn{G}(d')=\fn{F}(\mathfrak{G}_{j_0}(\fn{F}^{d'},d'))$.  Let $\mathfrak{G}_J(\fn{F},d)=\mathfrak{G}_{J\setminus\{j_0\}}(\fn{G},d)$.

For any $d$, there is a $d'\in[d,\mathfrak{G}_{J\setminus\{j_0\}}(\fn{G},d)]$ so that for all $i\in[d',\fn{G}(d')]$ and all $j\in J\setminus\{j_0\}$, $\phi_j(i)$ holds.  There is also a $k\in[d',\mathfrak{G}_{j_0}(\fn{F}^{d'},d')]$ so that, for all $i\in[k,\fn{F}^{d'}(k)]=[k,\fn{F}(k)]$, $\phi_{j_0}(i)$ holds.  Since $k\leq\mathfrak{G}_{j_0}(\fn{F}^{d'},d')$, also $\fn{F}(k)\leq\fn{F}(\mathfrak{G}_{j_0}(\fn{F}^{d'},d'))=\fn{G}(d')$, so $[k,\fn{F}(k)]\subseteq[d',\fn{G}(d')]$, so $k$ is the desired witness.
\end{proof}


\begin{theorem}[\!\!\cite{MR2556127}]\label{thm:eff diff null}
  There is a function $\mathfrak{k}$ so that whenever $\Lambda\subseteq K\{X_{[n]}\}_{\leq b}$ and $h\in K\{X_{[n]}\}_{\leq b}\cap\{\Lambda\}$, also $h\in\sqrt{\left(\Lambda_{[\mathfrak{k}(n,b)]}\right)}$.
\end{theorem}

\begin{notation}\label{not:j_function}
  We define $\mathfrak{j}_{n,m}(i_0,\fn{D},\fn{F},d,\Lambda)$ by recursion on the rank of $\Lambda$.

  Suppose $\mathfrak{j}_{n,m}(i_0,\fn{D},\fn{F},d,\Lambda')$ has been defined for all $\Lambda'$ with rank less than the rank of $\Lambda$.  Let $J$ consist of all finitely many (ranks of) autoreduced sets $\Lambda^*\subseteq K\{X_{[n]}\}_{\leq d}$ having rank less than $\Lambda$; for each $\Lambda^*\in J$ and each $\fn{D}$, we have functionals $\mathfrak{G}_{\Lambda^*,\fn{D}}(\fn{F},i)=\mathfrak{j}_{n,m}(i,\fn{D},\fn{F},d,\Lambda^*)$, so by the Knitting Lemma, we have a functional $\mathfrak{G}_{J,\fn{D}}$. 


Then we define $\mathfrak{j}_{n,m}(i_0,\fn{D},\fn{F},d,\Lambda)$ to be the maximum of:
\[\mathfrak{j}_{n,m}(\fn{F}(\mathfrak{G}_{J,\fn{D}}(\fn{F},i_0)),\fn{D},\fn{F},\mathfrak{g}(d,\beta),\Lambda_*)\]
where $\beta={\alpha+ m-1 \choose m-1}\cdot n\cdot (\alpha+1)$, $\alpha=\fn{D}(\fn{F}(\mathfrak{G}_{J,\fn{D}}(\fn{F},i_0)))+2 {2d \choose d}\mathfrak{k}(n,max\{d,\allowbreak \fn{D}(\fn{F}(\mathfrak{G}_{J,\fn{D}}(\fn{F},i_0)))\})$, and where $\Lambda_*$ ranges over autoreduced sets in $K\{X_{[n]}\}_{\leq \mathfrak{g}(d,\beta)}$ with lower rank than $\Lambda$.
\end{notation}
Again, we usually omit $n,m$ and we assume that functions $\fn{D},\fn{F}$ are monotonically increasing.

\begin{theorem}[Based on \cite{Marker:model_thy_of_fields}, Theorem 1.16, p. 47] \label{thm:local ritt}
  Let $i_0, \Lambda, \Lambda_0\subseteq\Lambda_1\subseteq\cdots, \fn{D}, \fn{F}, d$ be given such that:
  \begin{itemize}
  \item $\Lambda\subseteq K\{X_{[n]}\}_{\leq d}$ is autoreduced, and
  \item $\Lambda_i\subseteq K\{X_{[n]}\}_{\leq\fn{D}(i)}$ for all $i$.
  \end{itemize}

Then there is an $i\in[i_0, \mathfrak{j}(i_0,\fn{D},\fn{F},d,\Lambda)]$ so that $\Lambda_{\fn{F}(i)}\subseteq\{\Lambda\cup\Lambda_i\}$.
\end{theorem}
\begin{proof}
We proceed by induction on the rank of $\Lambda$.  We assume that whenever $\Lambda'$ is an autoreduced set so that $\Lambda'$ has lower rank than $\Lambda$, the claim holds of $\Lambda'$.

The functional $\mathfrak{G}_{J,\fn{D}}$ is as above.

\begin{claim}\label{claim: ritt noether 1}
There is an $i\in [i_0,\mathfrak{G}_{J,\fn{D}}(\fn{F},i_0)]$ so that for every $u\in\{I_\lambda,S_\lambda\mid \lambda\in\Lambda\}$, $\Lambda_{\fn{F}(i)}\subseteq\{\Lambda\cup\{u\}\cup\Lambda_i\}$.
\end{claim}
\begin{claimproof}
If $u\in K$, the claim is trivially true. Otherwise, there is an autoreduced set  $\Lambda_*\subseteq[\Lambda\cup\{u\}]\cap K\{X_{[n]}\}_{\leq d}$ so that the rank of $\Lambda_*$ is strictly less than the rank of $\Lambda$. (This is because $u$ is nonzero and reduced with respect to $\Lambda$; we may let $\Lambda_*$ be the set consisting of $u$ and all elements of $\Lambda$ having smaller rank than $u$. See the proof of Proposition \ref{autoreduced set procedure}.)  Necessarily $\Lambda_*\subseteq\{\Lambda\cup\{u\}\}_{\leq d}$.

  In particular, $\Lambda_*\in J$, so $\Lambda_{\fn{F}(i)}\subseteq\{\Lambda_*\cup\Lambda_i\}\subseteq\{\Lambda\cup\{u\}\cup\Lambda_i\}$.
\end{claimproof}

For each $h\in\Lambda_{\fn{F}(i)}$ and each derivation $\theta$ such that the order $ord(\theta) \leq 2 {2d \choose d}\mathfrak{k}(n,max\{d,\allowbreak \fn{D}(\fn{F}(i))\})$, consider the remainder $\tilde h$ of $\theta h$ with respect to $\Lambda$.  Suppose that for some $h,\theta$, the remainder $\tilde h\neq 0$. Note that $\tilde h$ is nonzero, belongs to $[\Lambda \cup\{h\}]$, and is reduced with respect to $\Lambda$. Let $\alpha_i$ be $\fn{D}(\fn{F}(i))+2 {2d \choose d}\mathfrak{k}(n,max\{d,\allowbreak \fn{D}(\fn{F}(i))\})$, which bounds the order of $\theta h$. By the conversion in Remark \ref{rmk:order vs ranking} (to convert bounds on order to bounds on ranking), a derivative of at most order $\alpha_i$ has ranking at most $\beta_i:={\alpha_i+ m-1 \choose m-1}\cdot n\cdot (\alpha_i+1)$.  By the bound on remainders from \ref{pseudodivision bounds} and the reasoning used to prove the claim above, there is an autoreduced set $\Lambda_*\subseteq [\Lambda\cup\{h\}]\cap K\{X_{[n]}\}_{\leq \mathfrak{g}(d,\beta_i)}$ of lower rank than $\Lambda$. 

Now the inductive hypothesis applies and we find an $i'\in[\fn{F}(i),\mathfrak{j}(\fn{F}(i),\fn{D},\fn{F},\allowbreak \mathfrak{g}(d,\beta),\Lambda_*)]$ so that $\Lambda_{\fn{F}(i')}\subseteq\{\Lambda_*\cup\Lambda_{i'}\}\subseteq\{\Lambda\cup\{h\}\cup\Lambda_{i'}\}$.  Since $\fn{F}(i)\leq i'$, we have $h\in \Lambda_{\fn{F}(i)}\subseteq \Lambda_{i'}$ and $i'$ is the desired witness.


In the remaining case, for each $h\in\Lambda_{\fn{F}(i)}$ and each derivation $\theta$ such that $ord(\theta) \leq 2 {2d \choose d}\mathfrak{k}(n,max\{d,\allowbreak \fn{D}(\fn{F}(i))\})$, the remainder with respect to $\Lambda$ is $0$. 

\begin{claim}
  For every $h\in\Lambda_{\fn{F}(i)}$ and all sets of derivations $\{\theta_u\}_{u\in\{I_\lambda,S_\lambda\mid\lambda\in\Lambda\}},\allowbreak \theta'$ with $\sum_u ord(\theta_u)+ord(\theta')\leq 2 {2d \choose d}\mathfrak{k}(n,max\{d,\allowbreak \fn{D}(\fn{F}(i))\})$, we have 
\[(\prod_u \theta_u u)(\theta' h)\in\{\Lambda\cup\Lambda_i\}.\]
\end{claim}
\begin{claimproof}
  By induction on $\sum_uord(\theta_u)$.  When $\sum_uord(\theta_u)=0$, we have $\prod_u\theta_u u=\prod_u u=H_{\Lambda}$.  Since the reduction of $\theta' h$ with respect to $\Lambda$ is $0$, we have $H^m_{\Lambda}(\theta' h)\in[\Lambda]\subseteq\{\Lambda\cup\Lambda_i\}$.  Since this is a radical ideal, also $H_{\Lambda}(\theta' h)\in\{\Lambda\cup\Lambda_i\}$.

Otherwise, take some $u_0$ with $ord(\theta_{u_0})>0$, so $\theta_{u_0}=\delta\theta'_{u_0}$.  For $u\neq u_0$, take $\theta'_u=\theta_u$.  Note that 
\[\delta((\prod_{u}\theta'_u u)(\theta' h))
=\sum_u ((\delta\theta'_u u)\prod_{u'\neq u}\theta'_{u'}u')(\theta' h)+(\prod_u\theta'_u u)(\delta\theta' h).\]
Since $(\prod_{u}\theta'_u u)(\theta' h)$ and $(\prod_u\theta'_u u)(\delta\theta' h)$ are both in $\{\Lambda\cup\Lambda_i\}$, also
\[\sum_u ((\delta\theta'_u u)\prod_{u'\neq u}\theta'_{u'}u')(\theta' h)\in\{\Lambda\cup\Lambda_i\}.\]
  Multiply this by $\prod_u\theta_u u$, so $\sum_u (((\delta\theta'_u u)\theta_u u)\prod_{u'\neq u}\theta'_{u'}u'\theta_{u'} u')(\theta' h) \in\{\Lambda\cup\Lambda_i\}$.  Any term in this sum where $u\neq u_0$ now has the form $\gamma (\prod_u\theta'_uu)(\theta' h)$ and therefore belongs to $\{\Lambda\cup\Lambda_i\}$.  Therefore the remaining term,
\[(((\delta\theta'_{u_0} u_0)\theta_{u_0} u_0)\prod_{u'\neq u_0}\theta'_{u'}u'\theta_{u'} u')(\theta' h)
=\prod_u (\theta_u u)^2(\theta' h)\in\{\Lambda\cup\Lambda_i\}.\]
Since the ideal is radical, also $(\prod_u\theta_u u)(\theta' h)\in\{\Lambda\cup\Lambda_i\}$ as desired.
\end{claimproof}

Consider some $h\in\Lambda_{\fn{F}(i)}$.  Then for each $u\in\{I_\lambda,S_\lambda\mid\lambda\in\Lambda\}$ we have $h\in\{\Lambda\cup\Lambda_i\cup\{u\}\}$, so for each $u$ there is some $m$ so
\[h^{m}=\sum_i \gamma_{i,u}\delta_{i,u}\]
where each $\delta_{i,u}$ has the form $\theta_{i,u}\mu_{i,u}$ where $ord(\theta_{i,u})\leq\mathfrak{k}(n,max\{d,\fn{D}(\fn{F}(i)))$ and $\mu_{i,u}\in\Lambda\cup\Lambda_i\cup\{u\}$.  There are at most $2|\Lambda|\leq 2{2d \choose d}$ elements $u$, whence the sum of the orders of the $\theta_{i,u}$ is at most  $2{2d \choose d}\mathfrak{k}(n,max\{d,\allowbreak \fn{D}(\fn{F}(i))\})$.

Suppose we multiply these all together, $h^{m'}=\prod_u\sum_i \gamma_{i,u}\delta_{i,u}$, and so also 
\[h^{m'+1}=h\prod_u \sum_i \gamma_{i,u}\delta_{i,u}.\]
  Expanding the product, each term has the form $h\prod_u \gamma_{i_u,u}\delta_{i_u,u}$.  Consider some such term.  If there is some $u$ with $\mu_{i_u,u}\in\Lambda\cup\Lambda_i$ then $h\prod_u \gamma_{i_u,u}\delta_{i_u,u}=\gamma\theta_{i_u,u}\mu_{i_u,u}\in \{\Lambda\cup\Lambda_i\}$.

For each term $h\prod_u \gamma_{i_u,u}\delta_{i_u,u}$ with every $\mu_{i_u,u}=u$, we have
\[h\prod_u \gamma_{i_u,u}\delta_{i_u,u}=\gamma(\prod_u\theta_{i_u,u}u)h.\]
But we have shown  in the second claim that $(\prod_u\theta_{i_u,u}u)h\in\{\Lambda\cup\Lambda_i\}$.

Therefore $h^{m'+1}$ is a sum of terms belonging to $\{\Lambda\cup\Lambda_i\}$, so we have $h\in\{\Lambda\cup\Lambda_i\}$.
\end{proof}

\section{Making Bounds Explicit}\label{sec:bounds explicit}


\subsection{Ordinal Length Iteration of Functions}



We need the concept of ordinal length iterations of a function.  We first recall some basic theory of ordinals below $\epsilon_0$ (which more than suffices for our purposes---the largest ordinals we will need are in the vicinity of $\omega^{\omega^{\omega^\omega}}$).

\begin{definition}
  Any ordinal $\alpha<\epsilon_0$ has a unique \emph{Cantor normal form} given by a finite set $I$ of ordinals below $\epsilon_0$ and, for each $\beta\in I$, a positive natural number $c_\beta$, so that $\alpha=\sum_{\beta\in I}\omega^{\beta}c_\beta$.
\end{definition}
When $\alpha=0$, we have $I=\emptyset$.  We sometimes think of this representation as being given recursively: $\alpha=\sum_{\beta\in I}\omega^\beta c_\beta$, and each $\beta$ can further be expressed in Cantor normal form.  Note that each $c_\beta$ must be strictly positive (when $c_\beta=0$, should omit $\beta$ from $I$).

\begin{definition}
  When $\alpha=\sum_{\beta\in I}\omega^{\beta}c_\beta>0$ (so $I$ is non-empty), we write $\max\alpha$ and $\min\alpha$ for $\max I$ and $\min I$, respectively.

  When $\min I>0$, we call $\alpha$ a \emph{limit ordinal}.  When $\min I=0$, we call $\alpha$ a \emph{successor ordinal}.
\end{definition}
When $\min I=0$, we have $\alpha=\alpha'+\omega^0 c_0=\alpha'+c_0$ where $\alpha'$ is a limit ordinal.

\begin{definition}
For any $\alpha=\sum_{\beta\in I}\omega^{\beta}c_\beta>0$ and any $x\in\mathbb{N}$, we define $\alpha[x]<\alpha$ recursively by $\alpha-1$ if $\min\alpha=0$ and
\[\alpha[x]=\sum_{\beta\in I\setminus\{\min\alpha\}}\omega^\beta c_\beta+\omega^{\min\alpha}(c_{\min\alpha}-1)+\omega^{(\min\alpha)[x]}x\]
otherwise.
\end{definition}
When $\alpha$ is a successor, $\alpha[x]$ is always $\alpha-1$.  When $\alpha=\gamma+\omega^n$, $\alpha[x]=\gamma+\omega^{n-1}x$.  When $\alpha=\gamma+\omega^\omega$, $\alpha[x]=\gamma+\omega^xx$, and so on.  $\alpha[x]$ is the canonical sequence of approximations to $\alpha$; in particular, when $\alpha$ is a limit, $\lim_x \alpha[x]=\alpha$.

\begin{definition}
  If $\alpha=\sum_\beta \omega^\beta c_\beta$, we define the \emph{coordinate bound} $|\alpha|\in\mathbb{N}$ recursively by $|\alpha|=\max\{c_\beta, |\beta|\}$.
\end{definition}
$|\alpha|$ is the upper bound on the coefficients that appear anywhere in the Cantor normal form of $\alpha$.

\begin{definition}
  Let $g$ be a function.  We define:
  \begin{itemize}
  \item $g^0(b)=b$,
  \item $g^\alpha(b)=g^{\alpha[b]}(g(b))$.
  \end{itemize}
\end{definition}

In particular, it is helpful to introduce a canonical list of functions to serve as benchmarks for our bounds.
\begin{definition}
  Let $\mathcal{G}(b)=b+1$.
\end{definition}
Then $\mathcal{G}^\omega(b)=2b$, $\mathcal{G}^{\omega^2}(b)\geq 2^bb$, $\mathcal{G}^{\omega^3}(b)$ is essentially a tower of exponents of height $b$, $\mathcal{G}^{\omega^\omega}$ is roughly the unary Ackermann function.

This is similar to the \emph{fast-growing functions} \cite{odifreddi1999classical}, sometimes called the Grzegorczyk hierarchy.  We have chosen a slower indexing because it matches our applications: roughly speaking, $\mathcal{G}^{\omega^\alpha}$ is the $\alpha$-th function in the fast-growing hierarchy.

We will need various properties about the behavior of these iterations which are included in Appendix \ref{sec:ordinals}.

\subsection{Some Bounds on Order of Magnitude}

\begin{lemma}\label{thm:pn_bounds}
For each $n$ and $d\geq n$, $\mathfrak{p}_n(d)\leq\mathcal{G}^{\omega^28n}(d)$.
\end{lemma}
\begin{proof}
  By induction on $n$.  Clearly $\mathfrak{p}_1(d)=d=\mathcal{G}^0(d)$.

  Suppose the claim holds for $n$.  Note that for $d\geq n$,
\begin{align*}
  \mathfrak{e}(n,d)
&= 2^{(d+\mathfrak{d}_{n-1}(d))^{n-1}+1}d+d+\mathfrak{d}_{n-1}(d)\\
&= 2^{(d+(2d)^{2^{n-1}})^{n-1}+1}d+d+(2d)^{2^{n-1}}\\
&\leq\mathcal{G}^{\omega^22+\omega}(d).
\end{align*}

Then for $d\geq n$,
\begin{align*}
  \zeta_1(n,\mathfrak{p}_n(d),d)
&=({d+n\choose n}+2)\zeta_0(n,\mathfrak{p}_n(d))\\
&=({d+n\choose n}+2){n+\mathfrak{d}_n(\mathfrak{p}_n(d))\choose n}\\
&=({d+n\choose n}+2){n+(2\mathfrak{p}_n(d)^{2^n})\choose n}\\
&<(\frac{(d+n)(n+(2\mathfrak{p}_n(d)^{2^n}))e^2}{n})^n+2(\frac{(n+(2\mathfrak{p}_n(d)^{2^n}))e^2}{n})^n\\
&\leq((d+n+2)(n+(2\mathfrak{p}_n(d)^{2^n}))e^2)^n\\
&\leq((d+n+2)(n+(2\mathcal{G}^{\omega^28n}(d)^{2^n}))e^2)^n\\
&\leq((d+n+2)(n+(2\mathcal{G}^{\omega^2}(\mathcal{G}^{\omega^28n}(d)))e^2)^n\\
&\leq((d+n+2)\mathcal{G}^{\omega^2+\omega\cdot 4+n}(\mathcal{G}^{\omega^28n}(d)))^n\\
&\leq\mathcal{G}^{\omega^2 2}(\mathcal{G}^{\omega^28n}(d)).\\
\end{align*}
and
\begin{align*}
  \zeta_2(n,\mathfrak{p}_n(d),d)
&=(\zeta_1(n,\mathfrak{p}_n(d),d)+1)^{2^{\zeta_1(n,\mathfrak{p}_n(d),d)}-1}\\
&\leq\mathcal{G}^{\omega^2\cdot 3}(\mathcal{G}^{\omega^2\cdot 2}(\mathcal{G}^{\omega^28n}(d)))\\
&\leq\mathcal{G}^{\omega^2\cdot 5}(\mathcal{G}^{\omega^28n}(d)).\\
\end{align*}

We therefore have
\begin{align*}
  \nu(n+1,d)
&=\zeta_1(n,\mathfrak{p}_n(d),d)\zeta_2(n,\mathfrak{p}_n(d),d)\\
&\leq\mathcal{G}^{\omega^2\cdot 2}(\mathcal{G}^{\omega^28n}(d))\mathcal{G}^{\omega^2\cdot 5}(\mathcal{G}^{\omega^28n}(d))\\
&\leq(\mathcal{G}^{\omega^2\cdot 5}(\mathcal{G}^{\omega^28n}(d)))^2\\
&\leq\mathcal{G}^{\omega^2\cdot 6}(\mathcal{G}^{\omega^28n}(d))
\end{align*}
and
\begin{align*}
\mathfrak{p}_{n+1}(d)
&=\max\{2{\nu(n+1,d)+n+1\choose n+1}\nu(n+1,d),\mathfrak{e}(n,d)\}\\
&\leq\max\{2(\frac{(\nu(n+1,d)+n+1)e}{n+1})^{n+1}\nu(n+1,d),\mathcal{G}^{\omega^22+\omega}(d)\}\\
&\leq \mathcal{G}^{\omega^2\cdot 7+\omega 4+n+1}(\mathcal{G}^{\omega^28n}(d))\mathcal{G}^{\omega^2\cdot 6}(\mathcal{G}^{\omega^28n}(d))\\
&\leq \mathcal{G}^{\omega^2\cdot 8}(\mathcal{G}^{\omega^28n}(d))\\
&=\mathcal{G}^{\omega^28(n+1)}(d).
\end{align*}
\end{proof}

To get bounds on $\mathfrak{m}$, we first need the following observation:
\begin{lemma}
  Suppose $\tau=\tau^0\cup\tau^1$ and for every $a\in\tau^0$ and $b\in\tau^1$, $a\leq b$.  Then
\[\mathfrak{m}_{\tau,\fn{D}}(i)=\mathfrak{m}_{\tau^0,\fn{D}}(i)+\mathfrak{m}_{\tau^1,\fn{D}}(i+\mathfrak{m}_{\tau^0,\fn{D}}(i)).\]
\end{lemma}
\begin{proof}
  We proceed by induction on $\tau^0$.  When $\tau^0=\emptyset$, this is immediate from the definition.  Suppose the claim holds for all $\hat\tau^0<_{multi}\tau^0$.  Then
  \begin{align*}
    \mathfrak{m}_{\tau,\fn{D}}(i)
&=1+\mathfrak{m}_{\tau_{\langle \min\tau,i,\fn{D}\rangle},\fn{D}}(i+1)\\
&=1+\mathfrak{m}_{\tau^0_{\langle \min\tau,i,\fn{D}\rangle}\cup\tau^1,\fn{D}}(i+1)\\
&=1+\mathfrak{m}_{\tau^0_{\langle\min\tau,i,\fn{D}\rangle},\fn{D}}(i+1)+\mathfrak{m}_{\tau^1,\fn{D}}(i+1+\mathfrak{m}_{\tau^0_{\langle\min\tau,i,\fn{D}\rangle},\fn{D}}(i+1))\\
&=\mathfrak{m}_{\tau^0,\fn{D}}(i)+\mathfrak{m}_{\tau^1,\fn{D}}(i+\mathfrak{m}_{\tau^0,\fn{D}}(i))
  \end{align*}
as needed.
\end{proof}

\begin{definition}
  If $\tau$ is a multiset with $\max\tau=n$ and, for each $i\leq n$ with $0<i$, $c_i$ copies of $i$, then
\[o(\tau)=\sum_{0<i\leq n}\omega^{i-1} c_i+2|\tau|.\]
\end{definition}

\begin{lemma}\label{thm:m_bound}
For any $\tau$ and any monotone $\fn{D}$ so that $\fn{D}(b)\geq 2b$, whenever $b\geq|\tau|$ we have
\[\mathfrak{m}_{\tau,\fn{D}}(b)\leq \fn{D}^{o(\tau)}(b).\]
\end{lemma}
\begin{proof}
We proceed by induction on $\max\tau$.

When $\max\tau=0$, $\mathfrak{m}_{\tau,\fn{D}}(b)=r$, so $\mathfrak{m}_{\tau,\fn{D}}(b)\leq \fn{D}^0(b)$ once $b\geq r=|\max\tau|$.

Suppose the claim holds for values less than $n$ and proceed by side induction on the number of copies of $n$ in $\tau$.  Suppose we are given $\tau$ and $n\in\tau$ is maximal so $\tau=\tau_0\cup\{n\}$.  Then, letting $\sigma_b$ be the multiset with $\fn{D}(b)$ copies of $n-1$,
\begin{align*}
  \mathfrak{m}_{\tau,\fn{D}}(b)
&=\mathfrak{m}_{\tau^0,\fn{D}}(b)+\mathfrak{m}_{\{n\},\fn{D}}(b+\mathfrak{m}_{\tau^0,\fn{D}}(b))\\
&=\mathfrak{m}_{\tau^0,\fn{D}}(b)+\mathfrak{m}_{\sigma_{b+\mathfrak{m}_{\tau^0,\fn{D}}(b)},\fn{D}}(b+1+\mathfrak{m}_{\tau^0,\fn{D}}(b))\\
&\leq \fn{D}^{o(\tau^0)}(b)+\fn{D}^{\omega^{n-2}(b+\fn{D}^{o(\tau^0)}(b))}(b+1+\fn{D}^{o(\tau^0)}(b))\\
&\leq \fn{D}^{o(\tau^0)}(b)+\fn{D}^{\omega^{n-1}}(b+\fn{D}^{o(\tau^0)}(b))\\
&\leq \fn{D}^{o(\tau^0)}(b)+\fn{D}^{\omega^{n-1}+o(\tau^0)+1}(b)\\
&\leq \fn{D}^{\omega^{n-1}+o(\tau^0)+2}(b)\\
&=\fn{D}^{o(\tau)}(b).
\end{align*}
\end{proof}

\begin{cor}\label{thm:m_bound_star}
  If $\fn{D}(b)\geq\max\{2b,b+1\}$ then $\mathfrak{m}^*(\fn{D},n)\leq\fn{D}^{\omega^{n-1}+1}(0)$.
\end{cor}

Recall that $\fn{F}_x(b)=\fn{F}(\mathfrak{p}_x(b))$.

\begin{lemma}\label{thm:uf_bounds}
  If $\fn{F}(b)\geq 2b$ for all $b$ then $\mathfrak{u}_{\fn{F}}(b)\leq \fn{F}_b^{\omega^b+\omega+1}(b)$.
\end{lemma}
\begin{proof}
  Let $\fn{D}_c(i)=\fn{F}^i_c(c)$.  By induction on $\alpha$, we claim that $\fn{D}_c^\alpha(i)\leq\fn{F}^{\omega\times\alpha+i}_c(c)$.  When $\alpha=0$, $\fn{D}_c^0(0)=0\leq \fn{F}^0_c(0)$.  When $\alpha>0$,
\[\fn{D}_c^\alpha(i)
=\fn{D}_c^{\alpha[i]}(\fn{D}_c(i))
=\fn{D}_c^{\alpha[i]}(\fn{F}_c^i(c))
\leq \fn{F}_c^{\omega\times(\alpha[i])}(\fn{F}^i_c(c))
\leq \fn{F}_c^{\omega\times\alpha+i}(c).\]

  So we have $\mathfrak{m}^*(i\mapsto\fn{F}_x^i(x),x)=\mathfrak{m}^*(\fn{D}_x,x)\leq\fn{D}^{\omega^{x-1}+1}_x(0)\leq\fn{F}^{\omega^x+1}_x(x)$.  Therefore
\[\mathfrak{u}_{\fn{F}}(x)\leq \fn{F}_x^{\fn{F}_x^{\omega^x+1}(x)}(x)\leq \fn{F}_x^{\omega}(\fn{F}_x^{\omega^x+1}(x))\leq\fn{F}_x^{\omega^x+\omega+1}(x).\]
\end{proof}

Note that, under the same assumptions, $\fn{F}^{\omega^b+\omega+1}(b)$ is much larger than $\mathfrak{d}_{b+1}(2b{2b\choose b}+1)$, so the same bound holds for $\mathfrak{u}^+_{\fn{F}}$.

\begin{lemma}\label{thm:fF_bounds}
  If $\fn{F}(b)\geq 2b$ for all $b$ then $\mathfrak{f}(\fn{F},b)\leq \fn{F}_b^{\omega^b+\omega^26+\omega+8}(b)$.
\end{lemma}
\begin{proof}
  We have 
  \begin{align*}
   N
&=\mathfrak{d}_{\mathfrak{u}^+_{\fn{F}}(b)}(\mathfrak{u}^+_{\fn{F}}(b))+\mathfrak{u}^+_{\fn{F}}(b)\\
&\leq (2\fn{F}_b^{\omega^b+\omega+1}(b))^{2^{\fn{F}_b^{\omega^b+\omega+1}(b)}}+\fn{F}_b^{\omega^b+\omega+1}(b)\\
&\leq \fn{F}_b^{\omega^22}(\fn{F}_b^{\omega^b+\omega+2}(b))+\fn{F}_b^{\omega^b+\omega+1}(b)\\
&\leq \fn{F}_b^{\omega^b+\omega^22+\omega+3}(b).
  \end{align*}

Then
\begin{align*}
  \mathfrak{f}(\fn{F},b)
&=\mathfrak{d}_{\mathfrak{u}^+_{\fn{F}}(b)}({N+b\choose b}\cdot\mathfrak{u}^+_{\fn{F}}(b))+N\\
&\leq \mathfrak{d}_{\mathfrak{u}^+_{\fn{F}}(b)}(\fn{F}_b^{\omega^b+\omega^24+\omega+6}(b))+\fn{F}_b^{\omega^b+\omega^22+\omega+3}(b)\\
&\leq \fn{F}_b^{\omega^b+\omega^26+\omega+7}(b)+\fn{F}_b^{\omega^b+\omega^22+\omega+3}(b)\\
&\leq \fn{F}_b^{\omega^b+\omega^26+\omega+8}(b).
\end{align*}
\end{proof}

\begin{lemma}\label{thm:zk_bounds}
  $\mathfrak{z}^k(d,b)\leq \mathcal{G}^{\omega^24k+\omega 6k+3k}(\max\{b,d\})$.
\end{lemma}
\begin{proof}
Observe that $\mathfrak{g}(b+d-1,\max\{b+d-1,2b\})\leq \mathfrak{g}(2\max\{b,d\})\leq \mathcal{G}^{\omega^2+\omega+1}(\max\{b,d\})$ and so
\[\mathfrak{d}_{b+d}(\mathfrak{g}(b+d-1,\max\{b+d-1,2b\})+d+1)\leq \mathcal{G}^{\omega^23+\omega 3+2}(\max\{b,d\}).\]

  We now proceed by induction on $k$.  When $k=0$, $\mathfrak{z}^0(d,b)=d\leq \fn{F}^0(d)$.

  Suppose the claim holds for $k$.  Then
  \begin{align*}
    \mathfrak{z}^{k+1}(d,b)
&\leq\mathfrak{z}^k(\mathcal{G}^{\omega^24+\omega 4+2}(\max\{b,d\})+\mathcal{G}^{\omega^2+\omega+1}(\max\{b,d\})+d+1,b)\\
&\leq\mathfrak{z}^k(\mathcal{G}^{\omega^24+\omega 6+3}(\max\{b,d\}),b)\\
&\leq\mathcal{G}^{\omega^2 4k+\omega 6k+3k}(\max\{b,d\},b).
  \end{align*}
\end{proof}

\subsection{Bounds on Lemma \ref{thm:dickson_iteration_bound} and its Consequences}
We first need an assignment of ordinals to bad leader sequences.
\begin{definition}
  A \emph{bad Dickson sequence} in $\mathbb{N}^n$ is a sequence $\langle\vec a_1,\ldots,\vec a_m\rangle$ of elements $\mathbb{N}^n$ so that when $i<j$, $\vec a_i\not\preceq\vec a_j$.
\end{definition}

\begin{lemma}[\cite{simpson:MR961012}]
  There is an assignment of ordinals $o(\langle\vec a_1,\ldots,\vec a_m\rangle)\leq \omega^n$ to bad Dickson sequences so that $o(\langle\rangle)=\omega^n$, $o(\langle\vec a_1,\ldots,\vec a_m,\vec a_{m+1}\rangle)<o(\langle\vec a_1,\ldots,\vec a_m\rangle)$, and $|o(\langle \vec a_1,\ldots,\vec a_m \rangle)|\leq mk^n$.
\end{lemma}




\begin{lemma}\label{thm:ord_assign_bad_leader}
With $m$ derivatives and $n$ differential indeterminates, there is an assignment of ordinals $o(\langle u_1,\ldots,u_k\rangle)\leq \omega^m\cdot n$ to bad leader sequences so that $o(\langle\rangle)=\omega^m\cdot n$, $o(\langle u_1,\ldots,u_k,u_{k+1}\rangle)<o(\langle u_1,\ldots,u_k\rangle)$, and if each $u_i\in K\{X_{[n]}\}_{\leq d}$ then $|o(\langle u_1,\ldots,u_k\rangle)|\leq nkd^m$.
\end{lemma}
\begin{proof}
  Given $u_1,\ldots,u_k$, each $u_j=\delta^{k_{1,j}}_1\cdots\delta^{k_{m,j}}_m X_{i_j}$.  For each $j$, let $\vec k_j=\langle k_{1,j},\ldots,k_{m,j}\rangle$.  For each $i$, consider the subsequence $j_1,\ldots,j_{d_i}$ with $i_j=i$; then $\langle \vec k_{j_1},\ldots,\vec k_{j_{d_i}}\rangle$ is a bad Dickson sequence with ordinal $\alpha_i=\sum_{j\leq m} \omega^j\cdot c_{i,j}$.  Taking $o(\langle u_1,\ldots,u_k\rangle)=\sum_{j\leq m}\omega^j\cdot(\sum_i c_{i,j})$ gives the desired bound.
\end{proof}
The quantity $\sum_{j\leq m}\omega^j\cdot(\sum_i c_{i,j})$ is an instance of the ``natural'' or ``commutative'' sum for ordinals.  


\begin{lemma}\label{thm:hgamma_bounds}
Let $g$ be a fixed monotonic function with $g(b)\geq 2b$ for all $b$.  For each $c$, let $\fn{D}_c$ be the function $\fn{D}_c(i) =g^{c+i}(b)$.

Then for any $\gamma$ with $o(\gamma_\mu)=\alpha$ and any $b\geq\max\{|\gamma|,n,m+2\}$, 
\[\mathfrak{h}_{n,m}(\fn{D}_c,\gamma)\leq g^{\omega^{\alpha 2}+\omega+c+2}(b).\]
In particular,
\[\mathfrak{h}_{n,m}(\fn{D})\leq g^{\omega^{\omega^m\cdot 2n}+\omega+2}(b).\]

\end{lemma}
\begin{proof}
By induction on $o(\gamma_\mu)$.  When $o(\gamma_\mu)=0$, so $\gamma$ is maximal, $\mathfrak{h}_{n,m}(\fn{D}_c,\gamma)=1\leq b=g^0(b)$.

Suppose $o(\gamma_\mu)=\alpha$ and for all $\gamma'$ with $o(\gamma'_\mu)<\alpha$, the claim holds.  Let $d=\fn{D}_c(1)=g^{c+1}(b)$ and $b'=g^{\omega}(d)\geq n(|\gamma|+1)d^m$.  Let $\beta=\alpha[b']$.  For each $u\in[-1,d]$, we will show that 
\[w_u\leq g^{\omega^{\beta 2+1}(d-u)+\omega(d-u)+(d-u)}(b').\]
When $u=b'$, this is immediate.

Let $\delta_u=\omega^{\beta 2+1}(d-u)+\omega(d-u)+(d-u)$, and suppose we have shown that $w_u\leq g^{\delta_u}(b')$.  Then $\fn{D}_c(w_u)=g^{w_u+c}(b)\leq g^{w_u}(b')$.  Let $c_u=\fn{D}_c(w_u)$.  We will show that for each $k\in[0,c_u]$,
\[v_{u,k}\leq g^{\omega^{\beta 2}(c_u-k)+\omega 2(c_u-k)+2(c_u-k)}(g^{\delta_u}(b')).\]
When $k=c_u$, this is immediate.

Suppose the claim holds for $k$.  Then we have
\begin{align*}
  v_{u,k-1}
&=v_{u,k}+\mathfrak{h}_{n,m}(\fn{D}_{v_{u,k}},\gamma^\frown\langle (u,k)\rangle)\\
&\leq v_{u,k}+g^{\omega^{o(\gamma_\mu{}^\frown\langle u\rangle)2}+\omega+v_{u,k}+2}(b)\\
&\leq v_{u,k}
+g^{\omega^{\beta 2}+\omega+v_{u,k}+2}(b)\\
&= v_{u,k}
+g^{\omega^{\beta 2}+\omega+g^{\omega^{\beta 2}(c_u-k)+\omega 2(c_u-k)+2(c_u-k)}(g^{\delta_u}(b'))+2}(b)\\
&\leq v_{u,k}
+g^{\omega^{\beta 2}+\omega 2}(g^{\omega^{\beta 2}(c_u-k)+\omega 2(c_u-k)+2(c_u-k)+1}(g^{\delta_u}(b')))\\
&\leq g^{\omega^{\beta 2}(c_u-k)+\omega 2(c_u-k)+2(c_u-k)}(g^{\delta_u}(b'))
+g^{\omega^{\beta 2}(c_u-(k+1))+\omega 2 (c_u-(k+1))+2(c_u-k)+1}(g^{\delta_u}(b'))\\
&\leq 2g^{\omega^{\beta 2}(c_u-(k+1))+\omega 2(c_u-(k+1))+2(c_u-k)+1}(g^{\delta_u}(b')))\\
&\leq g^{\omega^{\beta 2}(c_u-(k+1))+\omega 2(c_u-(k+1))+2(c_u-(k+1))}(g^{\delta_u}(b'))).\\
\end{align*}

In particular, $v_{u,0}\leq g^{\omega^{\beta 2}c_u+\omega 2 c_u+2c_u}(g^{\delta_u}(b')))$.  Therefore
\begin{align*}
  w_{u-1}
&=v_{u,0}\\
&\leq g^{\omega^{\beta 2}c_u+\omega 2 c_u+2c_u}(g^{\delta_u}(b'))\\
&\leq g^{\omega^{\beta 2+1}+1}(g^{\omega 2}(g^{\delta_u}(b')))\\
&= g^{\omega^{\beta 2+1}+1}(g^\omega(g^{\omega^{\beta 2+1}(d-u)+\omega(d-u)+(d-u)}(b')))\\
&\leq g^{\omega^{\beta 2+1}(d-(u+1))+\omega(d-(u+1))+(d-(u+1))}(b').\\
\end{align*}

Therefore 
\begin{align*}
  \mathfrak{h}_{n,m}(\fn{D}_c,\gamma)
&=w_{-1}\\
&\leq g^{\omega^{\beta 2+1}(d+1)+\omega(d+1)+d+1}(b')\\
&\leq g^{\omega^{\beta 2+2}+1}(b')\\
&\leq g^{\omega^{\alpha 2}+1}(b')\\
&\leq g^{\omega^{\alpha 2}+\omega+c+2}(b).\\
\end{align*}
\end{proof}

\begin{lemma}\label{thm:isat_bounds}
\[\mathfrak{i}_{n,m}^{\mathrm{sat}}(b)\leq \mathcal{G}^{\omega^{\omega^m2n+2}2+\omega^{\omega^m2n}+\omega^26+3}(b).\]
\end{lemma}
\begin{proof}
  Let $\mathfrak{g}_*(d)=\mathfrak{g}(d,d)=d(1+d)^d\leq\mathcal{G}^{\omega^22+1}(d)$.  Then the function $b\mapsto \fn{D}^{\mathrm{sat}}_b(i)$ is bounded by $\mathfrak{g}_*^i(b)$, and therefore
\begin{align*}
\mathfrak{h}_{n,m}(\fn{D}^{\mathrm{sat}}_b)
&\leq \mathfrak{g}_*^{\omega^{\omega^m\cdot 2n}+\omega+2}(b)\\
&\leq (\mathcal{G}^{\omega^22+1})^{\omega^{\omega^m\cdot 2n}+\omega+2}(b)\\
&\leq\mathcal{G}^{\omega^{\omega^m2n+2}2+\omega^{\omega^m2n}+\omega^24+2}(b).
\end{align*}

Then
\[\mathfrak{i}_{n,m}^{\mathrm{sat}}(b)\leq \mathfrak{g}_*(\mathcal{G}^{\omega^{\omega^m2n+2}2+\omega^{\omega^m2n}+\omega^24+2}(b))\leq\mathcal{G}^{\omega^{\omega^m2n+2}2+\omega^{\omega^m2n}+\omega^26+3}(b).\]
\end{proof}

\begin{lemma}\label{thm:icohere_bounds}
When $b\geq \max\{m,n\}$,
\[\mathfrak{i}_{n,m}^{\mathrm{cohere}}(b)\leq \mathcal{G}^{\omega^{\omega^m2n+2}2+\omega^{\omega^m2n+1}+\omega^{\omega^m2n}+\omega^26+\omega 2+3}(b).\]
\end{lemma}
\begin{proof}
  We use the same reasoning as in the previous lemma, but with the function
  \begin{align*}
    \mathfrak{g}_*(b)
&=\mathfrak{g}(b,{2b+m-1\choose m-1}n(b+1))\\
&\leq \mathcal{G}^{\omega^22+\omega+1}(b),
  \end{align*}
so
\[\mathfrak{i}^{\mathrm{cohere}}_{n,m}(b)\leq \mathcal{G}^{\omega^{\omega^m2n+2}2+\omega^{\omega^m2n+1}+\omega^{\omega^m2n}+\omega^26+\omega 2+3}(b).\]
\end{proof}

\begin{lemma}\label{thm:ichar_bounds}
When $b\geq \max\{3,m,n\}$,
\[\mathfrak{i}_{n,m}^{\mathrm{char}}(b)\leq \mathcal{G}^{\omega^{\omega^m 2n+\omega}+\omega^{\omega^m 2n+1}+\omega^{\omega+1}+\omega^\omega 2+\omega^2+\omega 2}(b).\]
\end{lemma}
\begin{proof}
Observe that
\begin{align*}
\fn{F}^{\mathrm{char}}_c(k)
&=\mathfrak{z}^{c+\mathfrak{g}(c,k)}(\mathfrak{g}(c,k),c)\\
&\leq\mathcal{G}^{\omega^2 4(c+\mathfrak{g}(c,k))+\omega 6(c+\mathfrak{g}(c,k))+3(c+\mathfrak{g}(c,k))}(\max\{c,\mathfrak{g}(c,k)\})\\
&\leq\mathcal{G}^{\omega^3}(6(c+\mathfrak{g}(c,k)))\\
&\leq\mathcal{G}^{\omega^3+\omega^22+\omega+1}(\max\{c,k\}).
\end{align*}
Also 
\begin{align*}
  \mathfrak{u}_{\fn{F}^{\mathrm{char}}_b}(b)
&\leq(\mathcal{G}^{\omega^3+\omega^22+\omega+1})^{\omega^b+\omega+1}(b)\\
&\leq\mathcal{G}^{\omega^{b+3}+\omega^{b+2}2+\omega^{b+1}+\omega^b+\omega^4+\omega^33+\omega^23+\omega 2+1}(b)\\
&\leq\mathcal{G}^{\omega^\omega+\omega}(b)
\end{align*}
and similarly $\mathfrak{u}^+_{\fn{F}^{\mathrm{char}}_b}(b)\leq\mathcal{G}^{\omega^\omega+\omega}(b)$.


  This time we use the function
\begin{align*}
\mathfrak{g}_*(b)
&=\max\{\mathfrak{g}(b,{2b + m-1\choose m-1} n(b+1)),
\mathfrak{p}_{\mathfrak{u}^+_{\fn{F}^{\mathrm{char}}_b}(b)}(\mathfrak{u}_{\fn{F}^{\mathrm{char}}_b}(b)),
\mathfrak{f}(\fn{F}^{\mathrm{char}}_b,b)\}\\
&\leq\mathcal{G}^{\omega^\omega+\omega}(b)
\end{align*}
so
\[\mathfrak{i}_{n,m}^{\mathrm{char}}(b)\leq \mathcal{G}^{\omega^{\omega^m 2n+\omega}+\omega^{\omega^m 2n+1}+\omega^{\omega+1}+\omega^\omega 2+\omega^2+\omega 2}(b).\]
\end{proof}

\subsection{Bounds on Ritt-Noetherianity}

\begin{lemma}[\!\!\cite{MR2556127}]\label{thm:diff null bounds}
When $d\geq n$, $\mathfrak{k}(n,d)\leq \mathcal{G}^{\omega^{n+8}+\omega^22}(d)$.
\end{lemma}
(The extra factor of $\omega^22$ more than covers the roughly factorial differnce between the order of terms used in \cite{MR2556127} and the ranking we use here.)

\begin{lemma}
  Suppose that $\mathfrak{G}_j(\fn{F},d)\leq \fn{F}^{\alpha}(d)$.  Then $\mathfrak{G}_J(\fn{F},d)\leq \fn{F}^{\alpha^{|J|+1}}(d)$.
\end{lemma}
\begin{proof}
  By induction on $|J|$, we show that $\mathfrak{G}_J(\fn{F},d)\leq \fn{F}^{(\alpha+1)^{\otimes n}}(d)$ where $(\alpha+1)^{\otimes n}$ is iterated commutative multiplication, as described in the appendix.  The conclusion follows since $(\alpha+1)^{\otimes n}\leq \alpha^{n+1}$.

When $|J|=1$, this is immediate.  Observe that $\fn{G}(d')\leq \fn{F}^{\alpha+1}(d')$, so $\mathfrak{G}_J(\fn{F},d)=\mathfrak{G}_{J\setminus\{j_0\}}(\fn{G},d)\leq \fn{G}^{(\alpha+1)^{\otimes(|J|-1)}}(d)\leq \fn{F}^{(\alpha+1)^{\otimes |J|}}(d)$.
\end{proof}

  We assign an explicit ordinal $o(\Lambda)$ to autoreduced sets so that when $\Lambda'$ has lower rank than $\Lambda$, $o(\Lambda')<o(\Lambda)$.

  \begin{lemma}
    With $m$ derivatives and $n$ differential indeterminates, there is an assignment of ordinals $o(\Lambda)\leq \omega^{\omega^m\cdot n}$ to autoreduced sets so that $o(\langle\rangle)=\omega^{\omega^m\cdot n}$ and if $\Lambda'$ has lower rank than $\Lambda$ then $o(\Lambda')<o(\Lambda)$.
  \end{lemma}
  \begin{proof}
    Let $\Gamma(\Lambda)=\langle (\mu_1,b_1),\ldots,(\mu_r,b_r)\rangle$.  Set
\[o(\Lambda)=\sum_{i\leq r}\omega^{o(\langle \mu_1,\ldots,\mu_i\rangle)}b_i+\omega^{o(\langle \mu_1,\ldots,\mu_r\rangle)}.\]
If $\Lambda'$ has lower rank than $\Lambda$, so $\Gamma(\Lambda')=\langle (\mu'_1,b'_1),\ldots,(\mu'_{r'},b'_{r'})\rangle$ then either there is some $i\leq\min\{r,r'\}$ so that $\sum_{j<i}\omega^{o(\langle \mu_1,\ldots,\mu_j\rangle)}b_j=\sum_{j<i}\omega^{o(\langle \mu'_1,\ldots,\mu'_j\rangle)}b'_j$ but $\omega^{o(\langle \mu_1,\ldots,\mu_i\rangle)}b_i>\omega^{o(\langle \mu'_1,\ldots,\mu'_i\rangle)}b'_i$, and therefore $\omega^{o(\langle \mu_1,\ldots,\mu_i\rangle)}b_i>\sum_{j\geq i}\omega^{o(\langle \mu'_1,\ldots,\mu'_j\rangle)}b'_j+\omega^{o(\langle \mu'_1,\ldots,\mu'_r\rangle)}$, or $r'>r$ and we have $\omega^{o(\langle \mu_1,\ldots,\mu_r\rangle)}>\sum_{j>r}\omega^{o(\langle \mu'_1,\ldots,\mu'_j\rangle)}b'_j+\omega^{o(\langle \mu'_1,\ldots,\mu'_r\rangle)}$.
  \end{proof}

\begin{definition}
Let $\Lambda$ be autoreduced and let $\Gamma(\Lambda)=\langle (\mu_1,b_1),\ldots,(\mu_r,b_r)\rangle$.  Set
\end{definition}

\begin{lemma}\label{lemma:j_function}
Assume $\fn{F}(i)\geq i+1$ for all $i$ and $\fn{F},\fn{D}$ are monotonic.  Without loss of generality, let $\fn{F}(i)\geq\fn{D}(i)$ for all $i$.

Then $\mathfrak{j}_{n,m}(i_0,\fn{D},\fn{F},d,\Lambda)\leq \fn{F}^{\omega^{\omega^{o(\Lambda)}}2+\omega^{n+8}2}(\max\{d,n,i_0\})$.
\end{lemma}
\begin{proof}
Write $o(\Lambda)$ for the ordinal rank of $\Lambda$.  We proceed by induction on $o(\Lambda)$.  When $o(\Lambda)=0$, $\mathfrak{j}_{n,m}(i_0,\fn{D},\fn{F},d,\Lambda)=i_0$.

Suppose $o(\Lambda)=\gamma$ and the claim holds for all $\Lambda_*$ with $o(\Lambda_*)<\gamma$.  Let $b=n(d+n)^n\leq n(d)^{n+1}$ when $d\geq n$, so when $\Lambda_*\subseteq K\{X_{[n]}\}_{\leq d}$, $|o(\Lambda_*)|\leq b$.

Then
\begin{align*}
  \mathfrak{G}_{J,\fn{D}}(\fn{F},i_0)
  &\leq \fn{F}^{(\omega^{\omega^{\gamma[b]}}2+\omega^{n+8}2)^{2^{{d\choose n}\choose n}}}(i_0)\\
  &\leq \fn{F}^{(\omega^{\omega^{\gamma[b]}})^\omega}(\max\{2^{(d+n)^n},i_0\})\\
  &\leq \fn{F}^{\omega^{\omega^{\gamma[b]+1}}+\omega^2}(\max\{d+n,i_0\})\\
  &\leq \fn{F}^{\omega^{\omega^{\gamma[b]+1}}+\omega^2+\omega}(\max\{d,n,i_0\}).
\end{align*}

Therefore
\[\fn{D}(\fn{F}(\mathfrak{G}_{J,\fn{D}}(\fn{F},i_0)))\leq\fn{F}^{\omega^{\omega^{\gamma[b]+1}}+\omega^22+\omega+2}(\max\{d,n,i_0\}),\]
so
\[\mathfrak{k}(n,\fn{D}(\mathbf{F}(\mathfrak{G}_{J,\fn{D}}(\fn{F},i_0))))\leq\fn{F}^{\omega^{\omega^{\gamma[b]+1}}+\omega^{n+8}+\omega^23+\omega+2}(\max\{d,n,i_0\}),\]
\[\fn{D}(\fn{F}(\mathfrak{G}_{J,\fn{D}}(\fn{F},i_0)))+2{2d\choose d}\mathfrak{k}(n,\fn{D}(\mathbf{F}(\mathfrak{G}_{J,\fn{D}}(\fn{F},i_0))))\leq\fn{F}^{\omega^{\omega^{\gamma[b]+1}}+\omega^{n+8}+\omega^24+\omega4+2}(\max\{d,n,i_0\}),\]
and, taking $\alpha$ to be the quantity in the previous line,
\[{\alpha+m-1\choose m-1}\cdot n\cdot(\alpha+1)\leq\fn{F}^{\omega^{\omega^{\gamma[b]+1}}+\omega^{n+8}+\omega^26+\omega5+2} (\max\{d,n,i_0\}).\]
Taking $\beta$ to be the quantity in the previous line, we have
\[\mathfrak{g}(d,\beta)\leq\fn{F}^{\omega^{\omega^{\gamma[b]+1}}+\omega^{n+8}+\omega^27+\omega5+2}(\max\{d,n,i_0\}).\]

We may bound this last quantity by $b=\fn{F}^{\omega^{\omega^{\gamma[b]+1}}+\omega^{n+8}2}(\max\{d,n,i_0\})$.  Now let $b''=n(\mathfrak{g}(d,\beta))^{n+1}$, so when $\Lambda_*\subseteq K\{X_{[n]}\}_{\leq\mathfrak{g}(d,\beta)}$, $|o(\Lambda_*)|\leq b''$.  Then
\begin{align*}
  \mathfrak{j}_{n,m}(i_0,\fn{D},\fn{F},d,\Lambda)
&\leq\fn{F}{\omega^{\omega^{\gamma[b'']}}2+\omega^{\omega^{\gamma[b]+1}}+\omega^{n+8}2}(\max\{d,n,i_0\})\\
&\leq\fn{F}^{\omega^{\omega^\gamma}2+\omega^{n+8}2}(\max\{d,n,i_0\}).
\end{align*}
\end{proof}

\begin{cor}\label{thm:ritt bounds}
  Assume $\fn{F}(i)\geq i+1$ for all $i$ and $\fn{F},\fn{D}$ are monotonic.  Without loss of generality, let $\fn{F}(i)\geq\fn{D}(i)$ for all $i$.  Then $\mathfrak{j}_{n,m}(i_0,\fn{D},\fn{F},d,\Lambda)\leq \fn{F}^{\omega^{\omega^{\omega^{\omega^m}n}}2+\omega^{n+8}2}(\max\{d,n,i_0\})$.
\end{cor}

\appendix

\section{Ordinal Iterations}\label{sec:ordinals}

The next several lemmas show identities relating ordinal arithmetic to function iteration.  Throughout this section we assume that $g$ is monotonic and $g(b)\geq b+1$ for all $b$ and we consider $b\geq 1$.

\begin{lemma}
Let $\alpha$ and $\beta$ be ordinals with $\min\beta\leq\max\alpha$.  Then
\[g^{\alpha+\beta}(b)=g^\alpha(g^\beta(b)).\]
\end{lemma}
\begin{proof}
  By induction on $\beta$.  When $\beta=0$, this is trivial.  
\[g^{\alpha+\beta}(b)=g^{(\alpha+\beta)[b]}(g(b))
=g^{\alpha+(\beta[b])}(g(b))
=g^\alpha(g^{\beta[b]}(g(b)))
=g^{\alpha}(g^\beta(b))\]
using the inductive hypothesis since $\beta[b]<\beta$.
\end{proof}

The main difficulty when dealing with ordinal iterations is that they are not strictly monotonic: we do not, in general, have $\alpha<\beta$ implies $g^\alpha(b)\leq g^\beta(b)$ (consider the case where $b$ is much smaller than $n$: we may have $g^n(b)>g^\omega(b)$).

When considering the effect of $\alpha$ on the size of $g^\alpha(b)$, both the size of $\alpha$ and the size of its coefficients matter.

We next establish some lemmas showing some cases when we can obtain monotonicity.  First, note that when $\max\beta\leq\min\alpha$ we have
\[g^\alpha(b)\leq g^\alpha(g^\beta(b))=g^{\alpha+\beta}(b).\]

\begin{lemma}
  $g^{\omega^\alpha}(b)> g^{\alpha}(b)$.
\end{lemma}
\begin{proof}
  By induction on $\alpha$.  When $\alpha=0$, 
\[g^{\omega^0}(b)=g(b)\geq b+1>b=g^0(b).\]

When $\alpha>0$,
\[g^{\omega^\alpha}(b)
=g^{\omega^{\alpha[b]}b}(g(b))
\geq g^{\omega^{\alpha[b]}}(g(b))
> g^{\alpha[b]}(g(b))
=g^\alpha(b).\]
\end{proof}

\begin{lemma}
  $g^{\omega^\alpha c}(b)\geq g^c(b)$.
\end{lemma}
\begin{proof}
  By induction on $\alpha$.  When $\alpha=0$ the two sides are identical.  When $\alpha>0$ we have
\begin{align*}
g^{\omega^\alpha c}(b)
&=g^{\omega^\alpha}(g^{\omega^\alpha}(\cdots (g^{\omega^\alpha}(b))\cdots))\\
&>g^{\alpha}(g^{\alpha}(\cdots (g^{\alpha}(b))\cdots))\\
&\geq g^{\omega^{\max\alpha}}(g^{\omega^{\max\alpha}}(\cdots (g^{\omega^{\max\alpha}}(b))\cdots))\\
&=g^{\omega^{\max\alpha}c}(b)\\
&\geq g^c(b).
\end{align*}
\end{proof}

\begin{lemma}
  For any $\epsilon>\delta$ and any $d\geq g^\delta(b)$, $\epsilon[d]\geq\delta$.
\end{lemma}
\begin{proof}
We proceed by induction on $\delta$.  When $\delta=0$, this is trivial.

Write $\epsilon=\epsilon'+\omega^\gamma$ where $\gamma=\min\epsilon$.  If $\gamma\leq\max\delta$ then since $\epsilon>\delta$, we must have $\epsilon'\geq\delta$, so $\epsilon[d]=\epsilon'+\omega^\gamma[d]\geq\epsilon'\geq\delta$.

So suppose $\gamma>\max\delta$.  If $\epsilon'\neq 0$ then $\epsilon'\geq\omega^\gamma>\delta$ and we are done, so assume $\epsilon'=0$ and therefore $\epsilon=\omega^{\gamma}$.  We have $\epsilon[d]=\omega^{\gamma[d]}d$.

By the inductive hypothesis, since $\gamma>\max\delta$ and $d\geq g^\delta(b)\geq g^{\omega^{\max\delta}}(b)\geq g^{\max\delta}(b)$, we also have $\delta[d]\geq\max\delta$.  If $\delta[d]>\max\delta$ then $\epsilon[d]>\delta$.

So suppose $\delta[d]=\max\delta$.  Then
\[d\geq g^\delta(b)\geq g^{\omega^{\max\delta}c_{\max\delta}}(b)\geq g^{c_{\max\delta}}(b)>c_{\max\delta}.\]
So $\epsilon[d]=\omega^{\gamma[d]}d\geq \omega^{\max\delta}(c_{\max\delta}+1)>\delta$.
\end{proof}

\begin{lemma}\label{thm:ordinal_downgrade}
  Let $\alpha=\sum_{\gamma\in I}\omega^\gamma c_\gamma$.  Let $\beta,\delta\in I$ with $\delta<\beta$, and let $c'_\beta=c_\beta-1$, $c'_\delta=c_\delta+1$, and $c'_\gamma=c_\gamma$ if $\gamma\not\in\{\beta,\delta\}$.  Let $\alpha'=\sum_{\gamma\in I}\omega^\gamma c'_\gamma$.  Then $g^{\alpha'}(b)\leq g^\alpha(b)$.
\end{lemma}
\begin{proof}
It suffices to consider the case where $I\cap(\beta,\delta)=\emptyset$, since we get the general case by applying this case several times.  So we have
\[\alpha=\alpha^++\omega^{\gamma}+\omega^\delta +\alpha^-\]
where $\min(\alpha^+)\geq \gamma$ and $\delta\geq\max(\alpha^-)$.  Since
\[g^\alpha(b)=g^{\alpha^+}(g^{\omega^{\gamma}+\omega^\delta}(g^{\alpha^-}(b)))\]
and
\[g^{\alpha'}(b)=g^{\alpha^+}(g^{\omega^\delta2}(g^{\alpha^-}(b)))\]
it suffices to show that when $\gamma>\delta$,
\[g^{\omega^\gamma+\omega^\delta}(b)\geq g^{\omega^\delta 2}(b).\]
We show this by induction on $\gamma$: 
\begin{align*}
  g^{\omega^\gamma+\omega^\delta}(b)
&=g^{\omega^\gamma}(g^{\omega^\delta}(b))\\
&=g^{\omega^{\gamma[g^{\omega^\delta}(b)]}}g((g^{\omega^\delta}(b)))\\
&\geq g^{\omega^{\gamma[g^{\omega^\delta}(b)]}+\omega^\delta}(b)\\
&\geq g^{\omega^\delta 2}(b).
\end{align*}
\end{proof}

\begin{lemma}
  If $d\leq d'$ then $g^{\alpha[d]}(b)\leq g^{\alpha[d']}(b)$.
\end{lemma}
\begin{proof}
  By induction on $\alpha$.  Write $\alpha=\alpha'+\omega^{\min\alpha}$.  We have $g^{\alpha[d]}(b)=g^{\alpha'}(g^{\omega^{\min\alpha}[d]}(b))$ and $g^{\alpha[d']}(b)=g^{\alpha'}(g^{\omega^{\min\alpha}[d]}(b))$, so it suffices to show that $g^{\omega^\beta[d]}(b)\leq g^{\omega^\beta[d']}(b)$.

If $\beta=0$ then $\omega^\beta[d]=0=\omega^\beta[d']$ and the claim is immediate.  Otherwise
\begin{align*}
  g^{\omega^\beta[d]}(b)
&=g^{\omega^{\beta[d]}d}(b)\\
&\leq g^{\omega^{\beta[d]}d'}(b)\\
&\leq g^{\omega^{\beta[d']}d'}(b)\\
&=g^{\omega^\beta[d']}(b).
\end{align*}
\end{proof}

\begin{lemma}
Suppose $\alpha,\beta,\gamma$ are ordinals with $\max\beta\leq\min\alpha$ and $\max\gamma\leq\min\beta$.   Suppose $g(b)$ is monotonic and $g(b)\geq b+1$ for all $b$.  Then
\[g^{\beta}(g^{\alpha+\gamma}(b))\leq g^{\alpha+\beta+\gamma}(b).\]
\end{lemma}
\begin{proof}
Since $g^{\beta}(g^{\alpha+\gamma}(b))=g^{\beta}(g^{\alpha}(g^\gamma(b)))$ and $g^{\alpha+\beta+\gamma}(b)=g^{\alpha+\beta}(g^\gamma(b))$, by replacing $b$ with $g^\gamma(b)$, we may assume $\gamma=0$.

We first proceed by induction on $\beta$.  Write $\beta=\beta'+\omega^\delta$ where $\delta=\min\beta$.  If $\beta'>0$ then we have
\begin{align*}
  g^\beta(g^\alpha(b))
 &=g^{\beta'+\omega^\delta}(g^\alpha(b))\\
 &=g^{\beta'}(g^{\omega^\delta}(g^\alpha(b)))\\
 &\leq g^{\beta'}(g^{\alpha}(g^{\omega^\delta}(b)))\\
 &\leq g^{\alpha+\beta'+\omega^\delta}(b)\\
 &=g^{\alpha+\beta}(b)
\end{align*}
using the inductive hypothesis twice.

So we assume $\beta=\omega^\delta$ and proceed by induction on $\alpha$.  Writing $\alpha=\alpha'+\omega^{\epsilon}$ where $\epsilon=\min\alpha$; if $\alpha'\neq 0$, we have
\begin{align*}
  g^\beta(g^\alpha(b))
&=g^\beta(g^{\alpha'}(g^{\omega^\epsilon}(b)))\\
&\leq g^{\alpha'+\beta}(g^{\omega^\epsilon}(b))\\
&=g^{\alpha'}(g^\beta(g^{\omega^\epsilon}(b)))\\
&\leq g^{\alpha'}(g^{\omega^\epsilon+\beta}(b))\\
&=g^{\alpha+\beta}(b)
\end{align*}
using the inductive hypothesis twice.  So we may reduce to the case where $\alpha=\omega^\epsilon$.

If $\epsilon=\delta$ then this follows from work above, so we may assume $\epsilon>\delta$.  Therefore $\omega^\epsilon[g^\beta(b)]\geq\beta$, which must mean that $\epsilon[g^\beta(b)]\geq \delta$.  Then we have:
\begin{align*}
  g^\beta(g^{\omega^\epsilon}(b))
&=g^\beta(g^{\omega^\epsilon[b]}(g(b)))\\
&=g^\beta(g^{\omega^{\epsilon[b]}b}(g(b)))\\
&\leq g^\beta(g^{\omega^{\epsilon[g^\beta(b)]}b}(g(b)))\\
&\leq g^\beta(g^{\omega^{\epsilon[g^\beta(b)]}(b+1)}(b))\\
&\leq g^{\omega^{\epsilon[g^\beta(b)]}(b+1)}(g^\beta(b))\\
&\leq g^{\omega^{\epsilon[g^\beta(b)]}(g^\beta(b))}(g(g^\beta(b)))\\
&= g^{\omega^\epsilon}(g^\beta(b))\\
&=g^{\omega^\epsilon+\beta}(b).
\end{align*}
\end{proof}

\begin{definition}
Let $\alpha=\sum_{\gamma\in I}\omega^\gamma c_\gamma$ and $\beta=\sum_{\delta\in J}\omega^\delta d_\delta$.  Then $\alpha\#\beta=\sum_{\gamma\in I\cup J}\omega^\gamma(c_\gamma+d_\gamma)$ (where $c_\gamma=0$ for $\gamma\not\in I$ and $d_\gamma=0$ for $\gamma\not\in J$).

Let $\alpha=\sum_{\gamma\in I}\omega^\gamma c_\gamma$ and $\beta=\sum_{\delta\in J}\omega^\delta d_\delta$.  Then $\alpha\otimes\beta=\sum_{\gamma\in I,\delta\in J}\omega^{\gamma\#\delta}c_\gamma c_\delta$.
\end{definition}
These are the ``natural'' or ``commutative'' addition and multiplication on ordinals.  Our work above shows
\begin{lemma}
  $g^\alpha(g^\beta(b))\leq g^{\alpha\#\beta}(b)$.
\end{lemma}








\begin{lemma}
  $(g^\alpha)^\beta(b)\leq g^{\alpha\otimes\beta}(b)$.
\end{lemma}
\begin{proof}
  By induction on $\beta$.  When $\beta=0$ this is immediate.  Otherwise,
  \begin{align*}
    (g^\alpha)^{\beta}(b)
&=(g^\alpha)^{\beta[b]}(g(b))\\
&\leq (g^{\alpha\otimes(\beta[b])})(g(b))\\
&\leq (g^{(\alpha\otimes\beta)[b]})(g(b))\\
&=g^{\alpha\otimes\beta}(b).
  \end{align*}
\end{proof}

\begin{lemma}
  If $|\beta|< b$ and $\beta<\alpha$ then $\beta\leq \alpha[b]$.
\end{lemma}
\begin{proof}
  By main induction on $\beta$ and side induction $\alpha$.  When $\alpha=\alpha'+1$, this is immediate from the inductive hypothesis.

  Suppose $\alpha=\alpha'+\omega^\gamma$, so $\alpha[b]=\alpha'+\omega^{\gamma[b]}b$.  If $\gamma\leq \max\beta$ then $\alpha'\geq\beta$ so $\alpha[b]\geq\beta$ as well.  So suppose $\gamma>\max\beta$.  Then by the main inductive hypothesis, $\gamma[b]\geq\max\beta$, and since $|\beta|<b$, $\beta<\omega^{\max\beta}b\leq\omega^{\gamma[b]}b\leq\alpha[b]$.
\end{proof}

\begin{lemma}
  If $\alpha>\beta$ and $|\beta|< b$ then $g^\beta(b)\leq g^\alpha(b)$.
\end{lemma}
\begin{proof}
  Let $\beta, g, b$ be fixed and proceed by induction on $\alpha$.  If $\alpha=\beta$ this is trivial, and if $\alpha=\alpha'+1$ this follows immediately from the inductive hypothesis and the monotonicity of $g$.

  If $\alpha$ is a limit ordinal then $g^\alpha(b)=g^{\alpha[b]}(g(b))\geq g^\beta(b)$ by the inductive hypothesis and the fact that $\alpha[b]\geq\beta$.
\end{proof}

\bibliographystyle{plain}
\bibliography{DiffBounds}
\end{document}